\documentclass[12pt, a4paper,reqno,english]{smfart}
\usepackage[utf8]{inputenc}
\usepackage{amsmath, amssymb, amsthm, physics, mathrsfs, bm}
\usepackage{booktabs,longtable,array,enumitem,xcolor}
\usepackage{graphicx, color}
\usepackage[colorlinks=true, linkcolor=blue, citecolor=blue, urlcolor=blue]{hyperref}
\usepackage{geometry}
\usepackage{soul}
\usepackage{cite}
\usepackage{listings}
\usepackage{longtable}
\definecolor{codegreen}{rgb}{0,0.6,0}
\definecolor{codegray}{rgb}{0.5,0.5,0.5}
\definecolor{codepurple}{rgb}{0.58,0,0.82}
\definecolor{backcolour}{rgb}{0.95,0.95,0.92}
\numberwithin{equation}{section}
\definecolor{bubblegum}{rgb}{0.99, 0.76, 0.8}
\definecolor{ietocean}{RGB}{0, 30, 140}
\definecolor{ietcoast}{RGB}{0, 150, 173}
\definecolor{ietlagoon}{RGB}{0, 216, 180}
\definecolor{veriforest}{RGB}{0, 80, 40}
\definecolor{revred}{RGB}{153, 0, 0}
\hypersetup{ 
        colorlinks,%
        citecolor={ietlagoon},%
        filecolor={ietocean},%
        linkcolor={veriforest},%
        urlcolor={revred},
        bookmarksnumbered=true,     
        bookmarksopen=true,         
        bookmarksopenlevel=1,       
        pdfstartview=Fit,           
        pdfpagemode=UseOutlines,
        pdfpagelayout=TwoPageRight
}

\definecolor{codegreen}{rgb}{0,0.6,0}
\definecolor{codegray}{rgb}{0.5,0.5,0.5}
\definecolor{codepurple}{rgb}{0.58,0,0.82}
\definecolor{backcolor}{rgb}{0.95,0.95,0.92}
\lstdefinestyle{pythonstyle}{
    backgroundcolor=\color{backcolor},
    commentstyle=\color{codegreen},
    keywordstyle=\color{magenta},
    numberstyle=\tiny\color{codegray},
    stringstyle=\color{codepurple},
    basicstyle=\ttfamily\small,
    breakatwhitespace=false,
    breaklines=true,                 
    captionpos=b,                    
    keepspaces=true,                 
    numbers=left,    
    numbersep=5pt,                  
    showspaces=false,
    showstringspaces=false,
    showtabs=false,       
    tabsize=4
}
 
\newtheorem{theorem}{Theorem}[section]
\newtheorem{proposition}[theorem]{Proposition}
\newtheorem{lemma}[theorem]{Lemma}
\newtheorem{corollary}[theorem]{Corollary}
\newtheorem*{theoremA}{Theorem A}
\newtheorem*{theoremB}{Theorem B}

\theoremstyle{definition}
\newtheorem{definition}[theorem]{Definition}
\newtheorem{example}[theorem]{Example}

\theoremstyle{remark}
\newtheorem{remark}[theorem]{Remark}

\newcommand{\Q}{\mathbb Q}
\newcommand{\Z}{\mathbb Z}
\newcommand{\C}{\mathbb C}
\newcommand{\PP}{\mathbb P}
\newcommand{\N}{\mathbb N}

\renewcommand{\ev}{\operatorname{ev}}
\newcommand{\Tot}{\operatorname{Tot}}
\newcommand{\vir}{\mathrm{vir}}
\newcommand{\vdim}{\operatorname{vdim}}
\newcommand{\ch}{\operatorname{ch}}
\newcommand{\rk}{\operatorname{rk}}
\newcommand{\Li}{\operatorname{Li}}
\newcommand{\CH}{\operatorname{CH}}

\newcommand{\im}{\operatorname{im}}

\newcommand{\Log}{\operatorname{Log}}

\newcommand{\Ish}{\mathcal I^{\sharp}}
\newcommand{\Dcal}{\mathcal D}

\newcommand{\amb}{\mathrm{amb}}
\newcommand{\thQ}{\theta_Q}

\renewcommand{\dd}{\,\mathrm d}

\title[Sunset integrals and local Gromov--Witten theory]{Equal-mass sunset
integrals and genus-zero local\\ Gromov--Witten theory of Calabi--Yau
$n$-folds}
\author{Hartmut Maennel}
\address{Google DeepMind, Zurich, Switzerland} 
\author{Pierre Vanhove}
\address{Institut de Physique Théorique, Université Paris Saclay,
  CNRS, CEA, F-91191 Gif-sur-Yvette, France}
\curraddr{Laboratoire d’Annecy-le-Vieux de Physique Th\'eorique (LAPTh), CNRS and Universit\'e
Savoie Mont-Blanc, Annecy, 74940, France}

\date{\today --- preprint number LAPTH-051/26}

\begin{document}

\begin{abstract}
We prove an exact identity, as convergent
logarithmic germs at large Euclidean momentum, between the equal-mass
two-dimensional $(n-1)$-loop sunset Feynman integral, with $n\geq4$, and the genus-zero one-point
descendant invariants of the local Calabi--Yau $n$-fold $X_n=\Tot(K_{F_n})$,
where $F_n\subset(\PP^1)^{n}$ is a smooth $(1,\ldots,1)$ hypersurface.
Writing $s$ for the momentum variable and $y=-1/s$, the identity reads
\[
 I^{(n-1)}_{\circleddash}(s)=(-1)^{n-1}yA\,\Bigl[nR^{n-1}-\theta_Q\Dcal_0^\Delta\,R^{n-2}
 +\sum_{j=0}^{n-3}c_jR^j\Bigr],
\]
where $R=\log y+O(y)$ is the local mirror coordinate, $A=\theta_yR$, $Q=e^R$,
$\Dcal_0^\Delta(Q)$ is the primary one-point potential, and each $c_j$ has the form
\[
c_j=c_j^{(0)}+\sum_{a=0}^{n-2}\bigl(c_{j,a}\,\Dcal_a^\Delta(Q)
 +c'_{j,a}\,\theta_Q\Dcal_a^\Delta(Q)\bigr)
\]
with the coefficients in $\Q[\zeta(3),\zeta(5),\ldots]$ the odd zeta values
coming from the Gamma class.  The descendant series $\Dcal_a^\Delta(Q)$ are
given explicitly by the ambient hypergeometric series and the inverse mirror
map.  We also prove global univalence of the principal mirror map, with explicit
upper and lower bounds on the $Q$-Taylor radii, extract the full primary
evaluation divisor and its integral cover transform, and reduce the ambient
endpoint constants to a finite universal Gamma-class part and
$\lfloor(n-1)/2\rfloor$ real parameters.
\end{abstract}

\maketitle

\tableofcontents

%%%%%%%%%%%%%%%%%%%%%%%%%%%%%%%%%%%%%%%%%%%%%%%%%%%%%%%%%%%%%%%%%%%%
\section{Introduction and statement of results}\label{sec:intro}

Feynman integrals are the basic building blocks of perturbative quantum field
theory. They enter in the evaluation of scattering amplitudes 
and the comparison between the theoretical
predictions and the experimental measurements in particle colliders
(see~\cite{Proceedings:2018jsb,Abreu:2022mfk,Craig:2022cef} for
instance), gravitational-waves observables~\cite{Bjerrum-Bohr:2022blt,Kosower:2022yvp,Bjerrum-Bohr:2022ows}
or in cosmology~\cite{Baumann:2022jpr,Benincasa:2022gtd}.
Their analytic structure is rich and has been the subject of intense
study in the last decades. It is now recognized that Feynman integrals are
periods of algebraic varieties~\cite{Bloch:2005bh,Brown:2009ta}, with physical observables expressed as relative
period integrals~\cite{Bloch:2005bh,Bloch:2008,Brown:2009ta,Bloch:2013tra,Weinzierl:2013fe,Bloch:2014qca,Brown:2015ylf,Brown:2015ztw,BKV,marcolli2019feynman,Vanhove:2014wqa,Bonisch:2021yfw,Doran:2023yzu}.

Amongst the simplest Feynman integrals are the equal-mass two-dimensional
$(n-1)$-loop sunset integrals, which are the subject of this paper. This family has
a particular status since it is a natural seed of many quantum field theory
computations, and it is directly connected to the geometry of Calabi--Yau
$n$-folds~\cite{BKV,KlemmEtAl,BFKNS,LairezVanhove,HV}. The two-loop case was studied in detail by
Bloch--Kerr--Vanhove~\cite{BKV}, who related the two-loop sunset integral to
the genus-zero prepotential of a local Calabi--Yau threefold through a
mirror/Legendre construction. In this paper we prove the corresponding
one-point descendant statement for all $n\geq4$.

For $n\geq2$ the two-dimensional $(n-1)$-loop sunset integral with all masses
equal to one has, for $\operatorname{Re}s<0$, the convergent Schwinger representation
\begin{equation}\label{eq:general-schwinger}
 I^{(n-1)}_\circleddash(s,1)=
 \int_{\mathbb R_+^{n}}
 e^{-\sum_{i=1}^n x_i+sW(x)}W(x)
                   \prod_{i=1}^n\frac{dx_i}{x_i},\qquad
 W(x)=\left(\sum_{i=1}^n x_i^{-1}\right)^{-1},
\end{equation}
with the momentum variable written $s$. This defines an analytic function on
$\C\setminus[n^2,\infty)$.
Set
\begin{equation}\label{eq:phidef}
 \phi(x)=\Bigl(\sum_{i=1}^nx_i\Bigr)\Bigl(\sum_{i=1}^nx_i^{-1}\Bigr),
\end{equation}
a Laurent polynomial homogeneous of degree zero, hence a regular function on
$(\C^*)^n/\C^*\cong(\C^*)^{n-1}$. 
$\phi(x)$ is the Landau--Ginzburg superpotential mirror to a smooth
$(1,\ldots,1)$ hypersurface $F_n\subset(\PP^1)^n$.

Since $W(x)\phi(x)=\sum_{i=1}^nx_i$, the exponent
in~\eqref{eq:general-schwinger} factors as
\begin{equation}\label{eq:LGrelation}
 -\sum_{i=1}^nx_i+sW(x)=-W(x)\bigl(\phi(x)-s\bigr).
\end{equation}
The projective integral representation in Proposition~\ref{prop:analyticdomain}
has denominator $\phi(x)-s$, whose zero locus is the fiber $\{\phi(x)=s\}$.
In~\cite[Section 2, equation (10)]{Iritani}  Iritani equates that oscillatory
integral with a $\widehat\Gamma$-pairing on $F_n$. Such an $F_n$ is Fano with
$K_{F_n}=\mathcal O_{F_n}(-H)$, and $X_n=\Tot(K_{F_n})$ is the local Calabi--Yau
$n$-fold whose genus-zero invariants the theorem below produces. Any smooth
member of the system may be taken as defined in Section~\ref{sec:scope}.

Our main result, Theorem~\ref{thm:assembly}, expresses the amplitude through
the genus-zero one-point descendant invariants of $X_n$. To state it, put
$y=-1/s$ and $\theta_y=y\,d/dy$, and let
\[
 c_D=\sum_{|\mathbf d|=D}\binom{D}{d_1,\ldots,d_n}^{\!2},\quad
 R=\log y+\sum_{D\geq1}\frac{(-1)^Dc_D}{D}\,y^D,\quad
 A=\theta_yR,\quad Q=e^R
\]
be the local mirror coordinate, its derivative and the Novikov variable
(see eq.~\eqref{eq:mirrorcoordinate}). For $0\leq a\leq n-2$ and a positive curve
class $\beta$ of degree $|\beta|=H\cdot\beta$, the invariants and their
diagonal generating series are
\[
 N_{a,\beta}=\int_{[\overline{\mathcal M}_{0,1}(X_n,\beta)]^\vir}
 \psi^a\ev_{X_n}^*\varpi^*\delta_a,\quad
 \delta_a=\frac{H^{n-2-a}}{(n-2-a)!},\quad
 \Dcal_a^\Delta(Q)=\sum_{\beta>0}N_{a,\beta}Q^{|\beta|},
\]
where $\varpi\colon X_n\to F_n$ is the projection (see eq.~\eqref{eq:descdef}).
Finally, the Gamma-class constants $\gamma_k$ are defined by
\[
 \sum_{k\geq0}\gamma_kt^k=\exp\left(\sum_{r\geq1}\frac{2\zeta(2r+1)}{2r+1}t^{2r+1}\right),
\]
so that $\gamma_0=1$, $\gamma_1=\gamma_2=0$ and $\gamma_3=\frac23\zeta(3)$
(see eq.~\eqref{eq:Gammapoly}).

\begin{theoremA}[Theorem~\ref{thm:assembly}]
For every $n\geq4$, as convergent logarithmic germs at $s=-\infty$,
\[
 I^{(n-1)}_\circleddash(s)=(-1)^{n-1}yA\sum_{k=0}^{n-1}(-1)^k\gamma_k\,
 \mathcal T_{n-1-k}(R,Q),
\]
where $\theta_Q=Q\,d/dQ$ and
\[
 \mathcal T_j=\frac{n!}{j!}R^j
 -\sum_{a=0}^{j-2}\frac{(n-2-a)!}{(j-a-2)!}R^{j-a-2}\Dcal_a^\Delta
 -\sum_{a=0}^{j-1}\frac{(n-2-a)!}{(j-a-1)!}R^{j-a-1}\theta_Q\Dcal_a^\Delta .
\]
In particular,
\[
 I^{(n-1)}_\circleddash(s)=(-1)^{n-1}yA\Bigl[nR^{n-1}
 -\theta_Q\Dcal_0^\Delta\,R^{n-2}
 -\bigl((n-2)\Dcal_0^\Delta+\theta_Q\Dcal_1^\Delta\bigr)R^{n-3}+\cdots\Bigr],
\]
where the omitted terms have degree at most $n-4$ in $R$; the odd zeta values
first enter there, through the term $-\gamma_3\mathcal T_{n-4}$.
\end{theoremA}

The coefficient of $R^{n-2}$ is thus $-\theta_Q$ applied to the primary
one-point potential $\Dcal_0^\Delta$, which for $n\geq4$ plays the role that
the prepotential of a local Calabi--Yau threefold plays at two
loops~\cite{BKV}. The descendant series are themselves given in closed form.
In the square-free algebra $\Q[p_1,\ldots,p_n]/(p_1^2,\ldots,p_n^2)$, with
$H=p_1+\cdots+p_n$, put
\[
 I(y,z)=e^{H\log y/z}\sum_{\mathbf d\in\N^n}(-y)^{|\mathbf d|}\,
 \frac{\prod_{r=0}^{|\mathbf d|-1}(H+rz)\,\prod_{r=1}^{|\mathbf d|}(H+rz)}
      {\prod_{i=1}^n\prod_{r=1}^{d_i}(p_i+rz)^2}.
\]
Then $I(y,z)=1+R(y)H/z+O(z^{-2})$. The sum is symmetric in the $p_i$, and
symmetric elements of degree $j$ in this algebra are multiples of $H^j$; hence
$e^{-HR/z}I(y,z)=1+\sum_{j=2}^nK_j(y)H^j/z^j$ for convergent power series
$K_j$.

\begin{theoremB}[Theorem~\ref{thm:extraction}]
The mirror map $y\mapsto Q=e^{R(y)}=y+O(y^2)$ is invertible near $y=0$, and
for $0\leq a\leq n-2$,
\[
 \Dcal_a^\Delta(Q)=-\frac{n!}{(n-2-a)!}\,K_{a+2}\bigl(y(Q)\bigr).
\]
\end{theoremB}

Corollary~\ref{cor:realcontinuation} continues the identity of Theorem~A
along the whole negative axis. Corollaries~\ref{cor:nfour}
and~\ref{cor:nfive} specialize it to three and four loops; these are the
Calabi--Yau fourfold and fivefold analogues of the two-loop theorem
of~\cite{BKV}, and they are the subject of the companion
paper~\cite{PaperII}.

\medskip 
The proof is uniform in $n$ and uses four independent parts.
\begin{enumerate}
\item
\emph{Geometry.} Stable maps to $X_n$ factor through the zero section over an
arbitrary base (Theorem~\ref{thm:factorization}), and the resulting virtual
class carries the excess factor $c_{m-1}(\mathcal V)$ with the obstruction
morphism identified, as well as its $K$-class (Theorem~\ref{thm:virtual}).

\item
\emph{Mirror symmetry.} A hypergeometric modification with two independent
characteristic classes, one Euler and one inverse Euler, puts the
series~\eqref{eq:Imixed} on the twisted Givental cone
(Theorem~\ref{thm:mirror}); monic division by $\lambda-H$ and the ambient
pairing then extract all of $N_{a,\beta}$ at once
(Theorem~\ref{thm:extraction}). Theorem~\ref{thm:fullambient}, proved from the
universal hypersurface rather than from mirror symmetry, upgrades the
comparison against ambient insertions to the full one-point cohomology classes, which
matters when $n$ is odd.

\item
\emph{Differential equations.} Definition~\ref{def:operator} is a finite
algorithm producing an order-$(n-1)$ operator $P_{n-1}$ over $\Q$, and
Theorem~\ref{thm:PF} verifies it against the complete ambient series, degree
by degree in the Novikov variable $Q=\exp(R)$, with the nonzero source $H^{n-1}y^H$. 
Proposition~\ref{thm:minimalperiod} establishes that $P_{n-1}$ is
irreducible, and hence the minimal operator for the period $A$; the proof uses
the irreducibility of its symmetric-power seed and two Fourier transforms with
checked extensions at zero.

\item
\emph{Analysis and arithmetic.} A positive Laplace transform converts the
oscillatory Gamma identity of~\cite[Section 2, equation~(10)]{Iritani} into the
signed formula~\eqref{eq:physicalGamma}, whose degree factor $(-1)^D$ is
derived in~\eqref{eq:laplace-degree-sign}. Separately, a smooth compact
Calabi--Yau double cover (Theorem~\ref{thm:CYcover}) isolates the local sector
on a target with trivial canonical bundle, which makes the published
integrality statements applicable and gives $b_\beta\in\Z$ for every $n\geq4$
(Corollary~\ref{cor:primaryintegral}).
\end{enumerate}

For every loop order, Theorem 4.1 of~\cite{Vanhove2026} establishes, by a
Mellin--Barnes residue computation rather than by any geometric input, an exact
convergent expansion of the equal-mass sunset integral having the same shape
as~\eqref{eq:allloop}:
the holomorphic period times the sum of a leading power of the flat coordinate
(since $\log y=-\log s+i\pi$ and its holomorphic period agrees with $\Pi_{n-1}$
to leading order at large Euclidean momentum; see Section~\ref{sec:dictionary}
for the normalization dictionary)
and an instanton series $\sum_{l\ge1}e^{lR}\sum_r d_r(l)R^r$ with rational
$d_r(l)$, plus a boundary combination $\sum_r\alpha_r\operatorname{Frob}_r$
whose coefficients lie in $\Q[\zeta(3),\zeta(5),\ldots]$ and are given there by
an explicit generating function. 

\medskip
Theorem~\ref{thm:assembly} supplies an enumerative identity for those
coefficients, for every $n\ge4$.
\begin{itemize}[leftmargin=*,itemsep=2pt]
\item The instanton coefficients $d_r(l)$, which~\cite{Vanhove2026} obtains as
 rational numbers from a residue calculation, are here the genus-zero one-point
 descendant invariants $\Dcal_a^\Delta$ of the local Calabi--Yau $n$-fold
 $X_n=\Tot(K_{F_n})$, with the descendant exponent $a$ and the insertion
 $\delta_a$ specified. At two loops this identification is the theorem
 of~\cite{BKV}; for $n\ge4$ it was not available.
\item The boundary coefficients $\alpha_r$ are here derived from the
 $\widehat\Gamma$-class of $F_n$ rather than read off a Mellin residue. This is
 the point left open in \cite[Remark 4.5]{Vanhove2026}, which records that the
 zeta pattern is consistent with a $\widehat\Gamma$-class prediction and states
 that establishing it rigorously beyond two loops would extend the argument
 of~\cite{Iritani} and is left for future work.
\item The arithmetic statements have no counterpart on the analytic side. The
 coefficients produced by a residue computation are rational numbers with no
 predicted denominators; Corollary~\ref{cor:primaryintegral} and
 Theorem~\ref{thm:QK} prove integrality of the transformed primaries and
 express them through quantum $K$-theory.
\end{itemize}

The leading term alone has a longer history. It was observed numerically
in~\cite{KlemmEtAl} and~\cite{BFKNS}, where it was matched against the modified
$\widehat\Gamma$-class of the degree-$(1,\ldots,1)$ Fano hypersurface, and
proved in~\cite{Iritani}; see also~\cite{Kerr}. The present
Theorem~\ref{thm:assembly} determines every coefficient below the leading one,
which is where the descendants enter. Those two sources also supply the
Calabi--Yau operator framework in which~\eqref{eq:physicalPF} sits; no priority
is claimed for the operators themselves, as recorded after
Example~\ref{ex:operators}.

\medskip

Corollary~\ref{cor:primaryintegral} and Theorem~\ref{thm:QK} concern one
specific observable, namely the highest nonconstant logarithmic coefficient of
the expansion, transformed by the multiple-cover rule with weight $r^{-2}$ for a
single primary insertion. That weight is the one-marked specialization of
$r^{\ell-3}$ for $\ell$ insertions, and it is not the $r^{-3}$ of the unmarked
genus-zero prepotential familiar from the Calabi--Yau threefold case used
in~\cite{BKV}. The exponent is fixed by the number of markings, not by the
codimension of the insertion or the dimension of the target. The two conventions
are compared in the companion paper~\cite{PaperII}.

Three things are not claimed in this paper.
First, the resulting integers are not claimed to be unsigned counts of embedded
rational curves, nor multiplicities of a specified sheaf-theoretic moduli
problem, nor physical BPS-state degeneracies. They are the output of a fixed
arithmetic transform applied to a fixed virtual invariant. Second, no
multiple-cover law is asserted for the descendant series $\Dcal_a^\Delta$ with
$a>0$; only the primary $a=0$ sector is transformed, and the rational
descendant description of the full integral is proved independently of any such
law. Third, the ideal-geometry hypotheses under which higher-dimensional
covering behaviour is usually discussed --- unobstructedness of simple rational
curves, and the moving-cover conditions of \cite[Sections 0.5 and 1.1]{PZ} ---
are not hypotheses here, and indeed fail for $X_n$, because a simple rational map
$u:\PP^1\to F_n$ of $H$-degree $m\ge2$ has $h^1(\PP^1,u^*T_{X_n})\ge m-1>0$. The
proofs below are algebraic and do not need them. This is discussed again where
it is used, at the end of Section~\ref{sec:arithmetic}.

\bigskip
This paper is organised as follows.
Section~\ref{sec:scope} fixes conventions and Section~\ref{sec:geometry} the
ambient state space. Sections~\ref{sec:local}--\ref{sec:serre} contain the
geometric and mirror-theoretic input, i.e. the zero-section factorization and
virtual reduction, the twisted mirror theorem and the extraction of all
descendants, and the quantum-Serre coefficient identity.
Section~\ref{sec:PF} gives the operator algorithm and its all-degree
certificate. Section~\ref{sec:analytic} 
derives the signed Gamma formula, and Section~\ref{sec:assembly} assembles
Theorem~\ref{thm:assembly}. Sections~\ref{sec:real}--\ref{sec:degreetwo} give
the real continuation, the global univalence of the mirror map with the
resulting $Q$-Taylor radius, and the low-degree geometry, and
Section~\ref{sec:arithmetic} the compact Calabi--Yau cover and the primary
integrality statements. Sections~\ref{sec:structural}--\ref{sec:checks} record
the ambientness and minimality statements, a specialization guide, the
normalization dictionary, and the independent numerical verifications.
Sections~\ref{ext:sec:primary} and~\ref{ext:endpoints} then extract the
full primary evaluation divisor and reduce the ambient endpoint
constants. Appendix~\ref{app:code} implements the
operator algorithm, Appendix~\ref{app:forget} derives the string and divisor
identities used with our twists, Appendix~\ref{app:structural} proves
ambientness of the one-point evaluation classes, and
Appendix~\ref{app:fourier} supplies the extension and Fourier-transform lemmas
used in the minimality argument.

%------------------------------------------------------------------
\section*{Acknowledgements and use of AI}
This work started as an application of the AI co-mathematician tool
\cite{CoMath} developed by Google DeepMind, in which a human works with AI
agents to develop a theory. Later, the output of the AI co-mathematician tool
was also reviewed by different AI tools (first Gemini, later also Claude and
ChatGPT) and the reviews were used to improve the theorems, sometimes in further
collaboration with the AI co-mathematician.\\
The final paper is human written, using parts of the AI generated proofs, and
also proofread by AI.\\
The work of HM was done at Google DeepMind.\\
The work of PV was funded by the Agence Nationale de la Recherche (ANR) under
grant Observables (ANR-24-CE31-7996).

\section{Index of notation}\label{app:notation}

The main symbols, grouped by topic. The last column gives the equation,
result or section where each is defined.

{\small
\setlength{\tabcolsep}{4pt}
\renewcommand{\arraystretch}{1.2}
\newcommand{\notagroup}[1]{\addlinespace[7pt]%
  \multicolumn{3}{@{}l}{\textbf{#1}}\\*\addlinespace[2pt]}
\begin{longtable}{@{}>{\raggedright\arraybackslash}p{0.2\linewidth}%
                    >{\raggedright\arraybackslash}p{0.62\linewidth}%
                    >{\raggedright\arraybackslash}p{0.14\linewidth}@{}}
\toprule
Symbol & Meaning & Defined\\
\midrule
\endfirsthead
\toprule
Symbol & Meaning & Defined\\
\midrule
\endhead
\bottomrule
\endfoot
\bottomrule
\endlastfoot
%-----------------------------------------------------------------------
\notagroup{Variables and geometry}
$n$ & Integer $n\ge4$; the sunset has $n-1$ loops, and $n=\dim_\C X_n$
  & \S\ref{sec:scope}\\
$s$, $y$, $T$ & Momentum $s$, the variable $y=-1/s$, and the Euclidean
  momentum $T=-s=1/y$ & \S\ref{sec:intro}, Thm.~\ref{thm:laplace}\\
$\theta_y$, $\thQ$ & Euler operators $y\,d/dy$ and $Q\,d/dQ$;
  $\theta_y=A\,\thQ$ & \S\ref{sec:PF}, Thm.~\ref{thm:assembly}\\
$Y_n$, $p_i$, $H$ & $Y_n=(\PP^1)^n$, the hyperplane classes of its factors
  ($p_i^2=0$), and $H=\sum_ip_i$ & \S\ref{sec:scope}\\
$F_n$ & Smooth $(1,\ldots,1)$ hypersurface in $Y_n$; Fano, with
  $K_{F_n}=\mathcal O_{F_n}(-H)$ & \eqref{eq:Fdefining}, Prop.~\ref{prop:geometry}\\
$X_n$, $\varpi$ & Local Calabi--Yau $n$-fold $\Tot(K_{F_n})$ and its
  projection to $F_n$ & \S\ref{sec:scope}\\
$\beta$, $m=|\beta|$ & Curve class in $H_2(F_n,\Z)\cong\Z^n$ and its degree
  $H\cdot\beta$ & Prop.~\ref{prop:geometry}\\
$\delta_a$ & Insertion $H^{n-2-a}/(n-2-a)!$ for the descendant exponent
  $0\le a\le n-2$ & Prop.~\ref{prop:geometry}\\
$N_{a,\beta}$ & One-point descendant invariant
  $\int_{[\overline{\mathcal M}_{0,1}(X_n,\beta)]^\vir}
  \psi^a\ev_{X_n}^*\varpi^*\delta_a$ & \eqref{eq:descdef}\\
$\Dcal_a^\Delta(Q)$ & Diagonal series $\sum_{\beta>0}N_{a,\beta}Q^{|\beta|}$;
  $\Dcal_0^\Delta$ is the primary one-point potential & \eqref{eq:descdef}\\
%-----------------------------------------------------------------------
\notagroup{Mirror side}
$R$, $A$, $c_D$ & Local mirror coordinate
  $R=\log y+\sum_{D\ge1}(-1)^Dc_Dy^D/D$ and period $A=\theta_yR$, where
  $c_D=\sum_{|\mathbf d|=D}\binom{D}{d_1,\ldots,d_n}^2$
  & \eqref{eq:mirrorcoordinate}\\
$Q$, $y(Q)$ & Novikov variable $Q=e^R=y+O(y^2)$ and the inverse mirror map
  & \eqref{eq:mirrorcoordinate}, Thm.~\ref{thm:extraction}\\
$I(y,z)$ & Diagonal ambient $I$-series $I^{0,\amb}_{X_n}(y,z)$, equal to
  $1+RH/z+O(z^{-2})$ & \S\ref{sec:intro}, \eqref{eq:Imixed}\\
$K_j(y)$ & $e^{-HR/z}I(y,z)=1+\sum_{j=2}^nK_jH^j/z^j$; Theorem~B:
  $\Dcal_a^\Delta=-\frac{n!}{(n-2-a)!}K_{a+2}(y(Q))$
  & \eqref{eq:Kdef}, \eqref{eq:extraction}\\
$\Ish$, $W_j$ & Signed ambient series and its coefficients,
  $\Ish(y,1)=\sum_jH^jW_j(y)$, with $W_0=A$ & \eqref{eq:Isharp}\\
$P_{n-1}$ & Sunset operator of order $n-1$, with
  $P_{n-1}\Ish(y,1)=H^{n-1}y^H$; the minimal operator of $A$
  & \eqref{eq:Palgorithm}, Thm.~\ref{thm:PF}\\
%-----------------------------------------------------------------------
\notagroup{The sunset integral}
$I^{(n-1)}_\circleddash(s)$ & Equal-mass two-dimensional $(n-1)$-loop sunset
  integral & \eqref{eq:general-schwinger}\\
$\phi$ & Landau--Ginzburg superpotential $(\sum_ix_i)(\sum_ix_i^{-1})$,
  mirror to $F_n$ & \eqref{eq:phidef}\\
$\widehat\Gamma_{X_n}$, $\gamma_k$ & Gamma class, with
  $\widehat\Gamma_{X_n}|_{F_n}=\sum_k\gamma_kH^k$ and
  $\sum_k\gamma_kt^k=\exp\bigl(\sum_{r\ge1}\frac{2\zeta(2r+1)}{2r+1}t^{2r+1}\bigr)$
  & \eqref{eq:Gammapoly}\\
$\Pi_{n-1}(s)$ & Period prefactor $(-1)^{n-1}yA(y)$ & \S\ref{sec:assembly}\\
$\mathcal T_j(R,Q)$ & Assembly polynomials; Theorem~A:
  $I^{(n-1)}_\circleddash=\Pi_{n-1}\sum_k(-1)^k\gamma_k\mathcal T_{n-1-k}$
  & \eqref{eq:Vj}, \eqref{eq:allloop}\\
$h_j(Q)$ & Coefficients of the zeta-free bracket
  $\mathcal T_{n-1}=nR^{n-1}+\sum_{j=0}^{n-2}h_jR^j$ & \eqref{eq:hj}\\
%-----------------------------------------------------------------------
\notagroup{Arithmetic and global results}
$b_\beta$, $n_m$ & Primary cover transform,
  $N_{0,\beta}=\sum_{k\alpha=\beta}b_\alpha/k^2$, and
  $n_m=\sum_{|\beta|=m}b_\beta$, so that $\Dcal_0^\Delta=\sum_mn_m\Li_2(Q^m)$;
  both are integers & Thm.~\ref{thm:integrality}, Cor.~\ref{cor:primaryintegral}\\
$\widetilde Z$ & Smooth Calabi--Yau double cover of
  $Z=\PP(\mathcal O_{F_n}\oplus K_{F_n})$ that isolates the local sector
  & Thm.~\ref{thm:CYcover}\\
$r_n$, $\rho_n$ & Bounds $r_n\le R_y\le\rho_n$ on the Taylor radius $R_y$ of
  $y(Q)$; the $\Dcal_a^\Delta$ converge on $|Q|<r_n$
  & \eqref{eq:rn}, \eqref{eq:rhon}\\
$D_\beta$, $B_\beta$ & Full primary evaluation class
  $\sum_id_{\beta,i}p_i\in H^2(F_n,\Q)$ and its integral cover transform
  $\sum_{r\mid\beta}\mu(r)r^{-2}D_{\beta/r}$, with $\mu$ the M\"obius function
  & \S\ref{ext:sec:primary}\\
$G(H)$, $b_j$ & Universal Gamma factor $\sum_k\gamma_kH^k$ and the
  $\lfloor(n-1)/2\rfloor$ real parameters of the endpoint constants
  & \eqref{ext:G}, Thm.~\ref{ext:main}\\
\end{longtable}}

\section{Scope and conventions}\label{sec:scope} 

 Throughout, work over $\C$, with $n\geq4$. We write
\[
 Y_n=(\PP^1)^n,\qquad H=\sum_{i=1}^np_i,\qquad
 F_n\in|\mathcal O_{Y_n}(H)|\text{ smooth},\qquad X_n=\Tot(K_{F_n}).
\]
Thus $F_n$ is a smooth multilinear hypersurface given by the zero locus in
$(\PP^1)^n$, with homogeneous coordinates $[u_i:v_i]$, of a section
\begin{equation}\label{eq:Fdefining}
 \sum_{\epsilon\in\{0,1\}^n}a_\epsilon\prod_{i=1}^n u_i^{\epsilon_i}v_i^{1-\epsilon_i}
 \in H^0\bigl(Y_n,\mathcal O_{Y_n}(1,\ldots,1)\bigr)
\end{equation}
whose vanishing locus is smooth, which is the generic member of the system.
No individual member is singled out because all smooth members are deformation
equivalent, so the genus-zero invariants used below do not depend on the
choice, and none of the arguments refers to the coefficients $a_\epsilon$.
By adjunction $K_{Y_n}=\mathcal O_{Y_n}(-2,\ldots,-2)$ gives
$K_{F_n}=\mathcal O_{F_n}(-H)$, so $F_n$ is Fano of index one with respect to $H$ and
$X_n=\Tot(K_{F_n})$ is a local Calabi--Yau $n$-fold; this is
Proposition~\ref{prop:geometry}.
The connection with the Feynman integral is through the Landau--Ginzburg
mirror of $F_n$ rather than through any particular defining
equation~\eqref{eq:Fdefining}; see~\eqref{eq:LGrelation} and
Section~\ref{sec:analytic}.

Write $i:F_n\hookrightarrow Y_n$, $\varpi:X_n\to F_n$, and
$\iota:F_n\hookrightarrow X_n$ for the inclusions and projection.
The symbol $H$ also denotes its restriction when the target is clear.
Cohomological degrees are real degrees $p_i\in H^2(Y_n,\Z)$.
Thus $H^j$ has complex codimension $j$.
The foundational stable-map statements are \cite[Theorem 3.14]{BM} and
\cite[Propositions 2, 4, and 5]{BehrendGW}. The actual obstruction morphism
is the unnumbered construction following Proposition 4 of~\cite{BehrendGW}, and
 Proposition~5 supplies global resolutions.
All virtual classes are the standard stable-map virtual classes, in rational
Chow groups or their cohomological realizations; integration on a proper
Deligne--Mumford stack is not replaced by integration on its coarse space.

We use $\ell$ for the number of marked points, reserving $n$ for the loop order plus one.
A curve class is denoted $\beta$, and
$m=|\beta|=H\cdot\beta$. The letter $a$ denotes a descendant exponent.
Only positive curve classes are included in local numerical potentials.
The equivariant parameter $\lambda$ is the first Chern class of the
weight-one character acting on the fibers of $K_{F_n}$; $\mu$ is a separate
weight used for the positive hypersurface twist.

\medskip\noindent\textbf{Notation.}
A superscript on $H$ carrying a space as argument, as in $H^k(F_n,\Q)$,
denotes a cohomology group; a bare power $H^k$ denotes the $k$-fold cup
product of the divisor class $H$. The letters $E_\pm$ always denote the
index bundles of Lemma~\ref{lem:basechange} on a stable-map stack,
while $\rho_\pm=\pm H$ and $\mathcal B_\pm=\mathcal O_{Y_n}(\pm H)$ denote
the corresponding cohomology class and line bundle on $Y_n$. The
operators $T_k$ of Definition~\ref{def:operator} carry a subscript and
are unrelated to the tangent bundles $T_{F_n},T_{Y_n},T_{X_n}$, to the assembly
polynomials $\mathcal T_j$ of~\eqref{eq:Vj}, or to the positive
variable $T=-s$ of Section~\ref{sec:analytic}. Similarly $P_{n-1}$ is the
differential operator and $\bar p$ is a sum of passive divisor classes.
Where an external theorem is quoted in full, it is stated at the point of
use, in the notation of its source where this causes no conflict, and the
paragraph following it fixes the translation.

\medskip\noindent\textbf{Standing caveats.}
The following apply throughout and are not repeated at every
occurrence. Rational numbers produced by a multiple-cover transform are
not asserted to count curves, sheaves, or BPS states. Integrality, where
proved, is an arithmetic statement about the specified observable.
No multiple-cover law is asserted for the descendants $\psi^a\delta_a$
with $a>0$. 
No finite computation of initial coefficients is used
as a proof of an all-degree statement; the finite checks collected in
Section~\ref{sec:checks} verify statements that are proved elsewhere in
the paper.

\medskip\noindent\textbf{Genus, dimension, and the primary observable.}
The loop number $n-1$ changes the target dimension, not the genus of the
stable maps, which is zero throughout the numerical theory considered
here. The subscript $0$ in $N_{0,\beta}$ denotes descendant exponent zero:
it does not mean that there are no marked points. The primary invariant
has one marking and the integral insertion $\delta_0$ of complex
codimension $n-2$. Its $r^{-2}$ cover transform and integrality are
statements about this specified observable. They are separate from any
multiple-cover formula for descendants or a sheaf-theoretic interpretation
of the resulting integers. The comparison with the unmarked $r^{-3}$ Calabi--Yau threefold
prepotential is carried out in the companion paper~\cite{PaperII}.

\section{Geometry and the ambient state space}\label{sec:geometry}
\begin{proposition}\label{prop:geometry}
One has
\[
 K_{F_n}=\mathcal O_{F_n}(-H),\quad K_{X_n}\cong\mathcal O_{X_n},\quad
 \int_FH^{n-1}=n!,\quad H_2(F_n,\Z)\cong\Z^n.
\]
Under the last isomorphism $\beta$ corresponds to
$(p_1\cdot\beta,\ldots,p_n\cdot\beta)$. Every effective class has nonnegative
coordinates. In particular, at most $\binom{m+n-1}{n-1}$ effective classes
have $H$-degree $m$. Moreover
\[
 \delta_a=\frac{H^{n-2-a}}{(n-2-a)!}
 =\sum_{|S|=n-2-a}\prod_{i\in S}p_i\in H^{2(n-2-a)}(F_n,\Z),
 \qquad 0\leq a\leq n-2.
\]
\end{proposition}
\begin{proof}
Adjunction gives $K_{F_n}=(-2H+H)|_{F_n}=-H$.
For a line bundle $E$ on a smooth base, the vertical tangent bundle on
$\Tot(E)$ is the pullback of $E$, so
$K_{\Tot(E)}=\varpi^*(K_{F_n}\otimes E^{-1})$; take $E=K_{F_n}$.
Also $i_*[F_n]=H$, and $H^n=n!p_1\cdots p_n$ on $Y_n$.
The Segre embedding realizes $\mathcal O_{Y_n}(H)$ as a hyperplane bundle.
Apply the integral homology statement of \cite[Corollary~7.3]{Milnor}
to this smooth projective
$n$-fold and its smooth hyperplane section. Since $2<n-1$, it gives
$H_2(F_n,\Z)=H_2(Y_n,\Z)=\Z^n$. The same theorem in degree zero
makes $F_n$ connected. Since it is smooth, it is irreducible and reduced. Each $p_i$ is nef and $H$ is ample; hence effective
nonzero classes have nonnegative coordinates with positive sum. Counting
all nonnegative tuples gives the stated upper bound, without assuming
that every tuple is represented by an irreducible rational curve.
Finally, expand $H^r$ using $p_i^2=0$, then each square-free monomial occurs
$r!$ times. 
\end{proof}

Set
\[
 \mathcal R=H^*(Y_n,\Q)=\Q[p_1,\ldots,p_n]/(p_1^2,\ldots,p_n^2),\qquad
 V_\amb=i^*H^*(Y_n,\Q)\subset H^{\mathrm{even}}(F_n,\Q).
\]
\begin{proposition}[Ambient pairings and diagonal symmetry]\label{prop:ambient}
The ordinary Poincar\'e pairing on $V_\amb$ is nondegenerate, and
\[
 \ker i^*=\ker(H\cup-:\mathcal R\to\mathcal R),\qquad
 \mathcal R^{S_n}=\Q[H]/(H^{n+1}).
\]
The radical of $\int_YabH/(\lambda-H)$ over $\Q(\lambda)$ is also
$\ker i^*$. If $n$ is even, all even cohomology of $F_n$ is ambient.
If $n$ is odd, the only possible nonambient even group is $H^{n-1}(F_n,\Q)$.
\end{proposition}
\begin{proof}
\cite[Corollary~7.3]{Milnor}, followed by \cite[Theorem~3.2]{Hatcher}
over $\Q$, identifies
cohomology below degree $n-1$ and is injective in degree $n-1$.
Poincar\'e duality applies to the closed oriented real $(2n-2)$-manifold
$F_n(\C)$ \cite[Theorem~3.30]{Hatcher}. For a degree $k>n-1$, the pairing of
$i^*H^k(Y_n)$ with
$H^{2n-2-k}(F_n)=i^*H^{2n-2-k}(Y_n)$ is represented on $Y_n$ by multiplication
by $H$. Lemma~\ref{lem:linearlefschetz} below says this multiplication has the required
full rank, so $i^*$ is surjective in degree $k$. In degree $n-1$ the
bilinear form $\int_YabH$ on $H^{n-1}(Y_n)$ is nondegenerate because
$H:H^{n-1}(Y_n)\to H^{n+1}(Y_n)$ is an isomorphism by that lemma ($\dim_\C Y_n=n$).
These facts prove nondegeneracy on the ambient image.
Since $i_*i^*a=Ha$, one inclusion of kernels is immediate; for the other,
if $Ha=0$, then $i^*a$ pairs to zero with every ambient class and is
therefore zero. Multiplication by $(\lambda-H)^{-1}$ is an automorphism
of the ambient algebra, so does not change the radical.
A square-free monomial basis of $\mathcal R$ is indexed by subsets of
$\{1,\ldots,n\}$. An invariant element has constant coefficients on
subsets of each fixed size, whose sum is $H^j/j!$. This proves the
invariant-ring assertion. Finally the only cohomology possibly not
inherited from $Y_n$ is the middle group, by the preceding Lefschetz and
duality arguments.
\end{proof}

We state a required Lefschetz linear algebra
\begin{lemma}\label{lem:linearlefschetz}
In $\mathcal R$, multiplication by $H$ from square-free degree $r$ to
degree $r+1$ is injective for $2r<n$ and surjective for $2r\geq n-1$.
\end{lemma}
\begin{proof}
Give the monomial basis $p_S$ the orthonormal positive-definite inner
product over $\mathbb R$. Let $U$ multiply by $H$ and put
$Dp_S=\sum_{i\in S}p_{S\setminus\{i\}}$. Then $D=U^*$ and direct
expansion gives $DU-UD=(n-2r)\operatorname{id}$ in degree $r$.
If $Uv=0$ and $2r<n$, then
$0=\|Dv\|^2+(n-2r)\|v\|^2$, so $v=0$.
The complement involution $p_S\mapsto p_{S^c}$ interchanges $U$ and
$D$. Injectivity in complementary degree therefore gives surjectivity
in the stated range. All maps are rational, so the rank conclusions
hold over $\mathbb Q$. In particular the middle map is an isomorphism
when its two graded pieces have the same dimension.
\end{proof}

\begin{remark}[The nonambient middle summand at $n=5$]\label{rem:primitive}
For $n=5$, $\chi(F_n)=40$, $b_2=b_6=5$, and $b_4=28$, whereas the ambient
part of $H^4(F_n,\Q)$ has dimension $\binom52=10$. Thus the nonambient middle summand
has dimension $18$. These numbers follow from
\[
 \chi(F_n)=n!\sum_{k=0}^{n-1}\frac{(-1)^{n-1-k}2^k}{k!},
\]
obtained by integrating the degree-$(n-1)$ part of
$c(TF)=\prod_{i=1}^5(1+2p_i)/(1+H)$; the remaining Betti numbers are forced by
Lefschetz and Poincar\'e duality.
The full Lefschetz-primitive middle group has dimension $28-5=23$;
it includes ambient primitive classes and is distinct from the
$18$-dimensional orthogonal nonambient summand.
This $18$-dimensional summand is the reason the ambientness statement of
Theorem~\ref{thm:fullambient} is needed when $n$ is odd. The four solutions
of the differential operator $P_4$ given in eq.~\eqref{e:P4} correspond to the ambient classes $1,H,H^2,H^3$, and nothing in the
operator sees the orthogonal complement.
\end{remark}

\begin{remark}[Relation to the Hulek--Verrill threefolds]\label{rem:HV}
The root lattice $A_4$ enters this story, but on the mirror side rather than
here. The superpotential of~\eqref{eq:phidef} is
$\phi(x)=n+\sum_{i\ne j}x_ix_j^{-1}$, so its Newton polytope is the convex
hull of the roots $e_i-e_j$, that is the reflexive polytope $\Delta_{A_{n-1}}$.
At $n=5$ the toric closures of the fibers $\{\phi=s\}$, with the small
resolutions described below and at nonsingular parameters, give the Calabi--Yau threefolds of
Hulek--Verrill~\cite{HV}, whose equation
$(\sum_{i=1}^5X_i)(\sum_{i=1}^5a_i/X_i)=a_6$ is the equal-mass member $a_i=1$, $a_6=s$
of their five-parameter family. By eq.~\eqref{eq:LGrelation} this is the
graph hypersurface of the four-loop sunset at momentum $s$.

The Fano fourfold $F_n$ corresponds to the Landau--Ginzburg model
$( (\C^*)^5/\C^*,\phi)$, whereas a compactified fiber has dimension three.
The Calabi--Yau counterpart of the latter is an anticanonical hypersurface
in $F_n$, equivalently a complete intersection of two $(1,\ldots,1)$
hypersurfaces in $(\PP^1)^5$; see \cite[Remarks 3 and 9]{Iritani}.
The numbers $h^{11}=26$, $h^{21}=16$ of~\cite[Proposition 3.1]{HV}
belong to the general smooth member of the full twenty-parameter family,
not to the sunset subfamily. For a general member of the five-parameter
family, small resolution of its thirty nodes gives $h^{11}=45$, $h^{12}=5$.
These smooth-family statements must be restricted to nonsingular parameters;
in particular they do not apply without qualification at $s=0,1,9,25$.
The period-generated subvariation is the one generated by the holomorphic
residue form and its Gauss--Manin derivatives. Under the natural coordinate
permutation action this form transforms by the sign character. The
logarithmic volume form on $(\C^*)^n/\C^*$ transforms by the determinant
of the standard permutation quotient, namely $\operatorname{sgn}(\sigma)$,
and $\phi$ is invariant. Thus this is not literally the invariant
subvariation unless the action has first been twisted by the sign character.
\end{remark}

\section{Stable maps in families and virtual reduction}\label{sec:local}
\begin{lemma}[Convexity and concavity on a tree]\label{lem:tree}
Let $C$ be a connected nodal curve of arithmetic genus zero and
$f:C\to F_n$ or $f:C\to Y_n$ a nonconstant map of total $H$-degree $m$.
Then
\[
 H^0(C,f^*\mathcal O(-H))=0,\quad
 h^1(C,f^*\mathcal O(-H))=m-1,\quad
 H^1(C,f^*\mathcal O(H))=0.
\]
The last line bundle is globally generated on $C$ and has $h^0=m+1$.
Its evaluation at any smooth marked point is surjective.
\end{lemma}
\begin{proof}
Every irreducible component is $\PP^1$ and the dual graph is a tree.
For the negative bundle a section vanishes on each noncontracted
component. On a contracted component it is constant, so matching values
at nodes and connectedness propagate vanishing to all components.
For the positive bundle all component degrees are nonnegative.
Inductively remove a leaf. Sections on the remaining tree and the leaf
can be matched at their common node since evaluation on the leaf is
surjective. The normalization exact sequence then gives $H^1=0$;
the same argument, with a prescribed value at a chosen point, gives
global generation. Riemann--Roch on a genus-zero nodal curve gives
$\chi(M)=\deg M+1$, which proves both ranks.
\end{proof}

\begin{lemma}[Artinian vanishing without a simultaneous normalization]\label{lem:artinian}
Let $A$ be a local Artinian $\C$-algebra with residue field $\C$,
$C_A\to\operatorname{Spec}A$ a flat proper family, and
$\mathcal L$ an invertible sheaf with $H^0(C_0,\mathcal L_0)=0$.
For every finite $A$-module $J$,
$H^0(C_A,\mathcal L\otimes_A J)=0$.
\end{lemma}
\begin{proof}
Choose a composition series for $J$ with successive quotients $\C$.
Flatness of $\mathcal L$ over $A$ makes the induced sequences of sheaves
exact, and their successive quotients are the pushforward of
$\mathcal L_0$ from the closed fiber. Left exactness of $H^0$ proves
vanishing by induction. Lemma~\ref{lem:tree} supplies the hypothesis
for the negative target line bundle. For a square-zero extension
$A'\twoheadrightarrow A$, the difference of two lifted fiber coordinates
lies in this group with $J=\ker(A'\to A)$, so it is zero.
Normalization is needed only on the closed fiber. A smoothing such as
$xy=\epsilon$ over $\C[\epsilon]/(\epsilon^2)$ has no section reducing
to the node and need not admit a componentwise normalization over its base;
the argument applies to it without such an assumption.
\end{proof}

\begin{lemma}[Universal base change]\label{lem:basechange}
For a flat proper finitely presented family $\pi:C\to S$ of such curves
over an arbitrary $\C$-scheme, with a map $f$ whose fibers have fixed
positive degree $m$, the complexes associated to $M_\pm=f^*\mathcal O(\pm H)$ satisfy
\[
 R\pi_*M_-\simeq E_-[-1],\qquad R\pi_*M_+\simeq E_+,
\]
where $E_-,E_+$ are locally free of ranks $m-1,m+1$ and commute with
arbitrary base change. For a marked section there is an exact sequence
\[
 0\longrightarrow E'_+\longrightarrow E_+
 \longrightarrow\ev^*\mathcal O(H)\longrightarrow0,
 \qquad\rk E'_+=m.
\]
\end{lemma}
\begin{proof}
Cohomology and base change for a perfect object and a finitely presented
flat proper-support complex, in \cite[Tag~0A1H]{Stacks}, applies with
$E=M_\pm$ and $\mathcal G^\bullet=\mathcal O_C$.
Thus $R\pi_*M_\pm$ is perfect and commutes with arbitrary base change.
Its fibers have amplitude $[0,1]$, so locally it can be represented by a
two-term complex of finite locally free modules $N^0\to N^1$.
For $M_-$ the fiber maps are injective by Lemma~\ref{lem:tree}; an
invertible maximal minor makes the map locally split injective. Its
cokernel is locally free. For $M_+$ the fiber maps are surjective and
are locally split surjective, giving a locally free kernel instead.
This proves the derived assertions, including on nonreduced bases.
The marked evaluation map is fiberwise surjective, hence locally split
surjective, and its kernel has the indicated rank. The argument applies
on atlases of the stable-map stacks and descends canonically.
\end{proof}

\begin{theorem}[Zero-section factorization over arbitrary bases]\label{thm:factorization}
For every $\beta\ne0$ and every number $\ell\geq0$ of markings,
\[
 \overline{\mathcal M}_{0,\ell}(F_n,\beta)
 \xrightarrow{\sim}\overline{\mathcal M}_{0,\ell}(X_n,\beta)
\]
is an isomorphism of stacks, compatible with arbitrary base change.
In particular the local stable-map stack is proper and Deligne--Mumford.
\end{theorem}
\begin{proof}
A map $g:C\to\Tot(K_{F_n})$ over a base $S$ consists of $f=\varpi g:C\to F_n$
and a section of $f^*K_{F_n}$. By Lemma~\ref{lem:basechange},
$\pi_*f^*K_{F_n}=0$, so this section is zero over $S$.
Conversely composing a map to $F_n$ with $\iota$ gives a
map to $X_n$. These constructions are inverse and functorial. They preserve
stability because exactly the same components are contracted; automorphism
groups agree because $\iota$ is a monomorphism. Properness and the
Deligne--Mumford property now follow from
\cite[Theorem 3.14]{BM}, $F_n$ is a
smooth projective variety over the characteristic-zero field $\C$.
\end{proof}

\begin{theorem}[Virtual reduction]\label{thm:virtual}
On the stack in Theorem~\ref{thm:factorization}, put
$\mathcal V=R^1\pi_*f^*K_{F_n}$, of rank $m-1$. Then
\[
 [\overline{\mathcal M}_{0,\ell}(X_n,\beta)]^{\vir}
 =c_{m-1}(\mathcal V)\cap
 [\overline{\mathcal M}_{0,\ell}(F_n,\beta)]^{\vir}.
\]
With the fiber scaling action, the equivariant formula replaces
$c_{m-1}$ by $e_\lambda(\mathcal V)$.
\end{theorem}
\begin{proof}
Along the zero section the differential of $\varpi$ and of $\iota$ give
the canonical split tangent sequence
$\iota^*T_{X_n}=T_{F_n}\oplus K_{F_n}$. Relative to the smooth Artin stack of
prestable marked curves, the stable-map obstruction complex is therefore
\[
 E_{X_n}=E_{F_n}\oplus\mathcal V^\vee[1],\qquad
 E_{F_n}=(R\pi_*f^*T_{F_n})^\vee.
\]
The assertion includes the morphism of obstruction theories. The projection and
zero section give $d\varpi\circ d\iota=1$, so the splitting of tangent bundles
induces a splitting after derived pushforward and duality, compatible with the
tangential stable-map obstruction morphism. For the normal component, construct
that morphism by the relative differential of the universal evaluation map and
the trace map for the universal curve. In a local trivialization of $K_{F_n}$,
its fiber coordinate pulls back to the zero section of $\mathcal O_{\mathcal C}$
over the entire base, by Theorem~\ref{thm:factorization}. Its differential in
base directions is therefore zero as a morphism of complexes. A change of
trivialization multiplies this coordinate by a unit and adds to its differential
a term proportional to the coordinate, also zero. Hence the normal component is
globally the zero morphism, including over nonreduced bases. Thus the resulting
obstruction theory is $(\phi_{F_n},0):E_{F_n}\oplus\mathcal V^\vee[1]\to\mathbb
L$, with the specified morphism, not just a splitting of its underlying
$K$-class. The intrinsic normal cone is embedded in the vector-bundle stack for
$E_{F_n}$ with zero component in $\mathcal V$. Factoring the zero-section Gysin
map for the enlarged obstruction bundle stack through that of $\mathcal V$ gives
the top Chern class. On a zero-locus chart $Z\subset U$ over an atlas $S$ of the
prestable-curve stack, the relative cone maps are explicitly
\[
 [C_{Z/U}/T_{U/S}|_Z]\hookrightarrow[E_+|_Z/T_{U/S}|_Z]
 \hookrightarrow[E_+|_Z/T_{U/S}|_Z]\oplus\mathcal V,
\]
where the second map has zero $\mathcal V$ coordinate. These are
compatible with descent. Only one zero-section self-intersection in
$\mathcal V$ occurs, producing $c_{m-1}(\mathcal V)$.
For precision, the cone and zero-section construction is that of
\cite[Sections 5 and 7, Proposition 5.3]{BF}. For a zero locus of a
vector-bundle section on a smooth ambient stack, the normal cone maps into
the restricted bundle and the ambient pushforward of the virtual class is
the top Chern class of that bundle, with no regularity of the section or
reducedness of its zero locus required.
The argument just given proves the required split excess formula from
that construction, rather than postulating a comparison from a $K$-class.
It works equivariantly by applying the same construction on finite-dimensional
approximations to the classifying space of the fiber-scaling torus.
\end{proof}

\begin{remark}[The corresponding virtual structure sheaf]\label{rem:localKexcess}
On the common stack of Theorem~\ref{thm:factorization}, the same split
obstruction morphism gives
\[
 \mathcal O^{\vir}_{X_n}
 =\lambda_{-1}(\mathcal V^\vee)\otimes\mathcal O^{\vir}_{F_n},
 \qquad
 \lambda_{-1}(\mathcal V^\vee)
 =\sum_{j=0}^{m-1}(-1)^j[\Lambda^j\mathcal V^\vee].
\]
Indeed the extra cone coordinate is zero in $\mathcal V$; the Koszul
resolution for its zero section gives the displayed factor in the
alternating Tor construction of the virtual structure sheaf
\cite[Section 2.3, Proposition 1]{Lee}. The subscripts indicate the standard stable-map
obstruction theories for the two targets, not different underlying stacks.
Thus neither virtual classes nor virtual structure sheaves for $F_n$ and
$X_n$ are identified merely from their stack isomorphism.
\end{remark}

For one marking, $\vdim_\C\overline{\mathcal M}_{0,1}(X_n,\beta)=n-2$.
Define, for $0\leq a\leq n-2$,
\begin{equation}\label{eq:descdef}
 N_{a,\beta}=\int_{[\overline{\mathcal M}_{0,1}(X_n,\beta)]^\vir}
 \psi^a\ev_{X_n}^*\varpi^*\delta_a,\quad
 \Dcal_a(\mathbf Q)=\sum_{\beta>0}N_{a,\beta}\mathbf Q^\beta,\quad
 \Dcal_a^\Delta(Q)=\sum_{\beta>0}N_{a,\beta}Q^{|\beta|}.
\end{equation}
The sums defining each coefficient of the last series are finite by
Proposition~\ref{prop:geometry}. On $F_n$ we also define proper evaluation
classes
\begin{equation}\label{eq:Cdef}
 C_{a,\beta}^\lambda=(\ev_{F_n})_*
 \left(\psi^a e_\lambda(\mathcal V)
 \cap[\overline{\mathcal M}_{0,1}(F_n,\beta)]^\vir\right).
\end{equation}
They are polynomial in $\lambda$ and at $\lambda=0$ have complex
codimension $a+1$ in $F_n$. Thus $N_{a,\beta}=\int_{F_n}\delta_a C_{a,\beta}^0$.
The evaluation target in~\eqref{eq:Cdef} is $F_n$, not $X_n$.
Equivariantly $\iota^*\iota_*C=(\lambda-H)C$; overlooking this factor
changes the descendant extraction. 

\section{The ambient twisted mirror calculation}\label{sec:mirror}
\subsection{Virtual comparison before any degenerate limit}
For positive $\beta$, let $M_{Y_n}=\overline{\mathcal M}_{0,1}(Y_n,i_*\beta)$,
$M_{F_n}=\overline{\mathcal M}_{0,1}(F_n,\beta)$, and $j:M_{F_n}\to M_{Y_n}$.
The class identification is unambiguous by Proposition~\ref{prop:geometry}.
Use $E_+,E'_+,E_-$ from Lemma~\ref{lem:basechange} on $M_{Y_n}$.

\medskip
\begin{proposition}\label{prop:virtualcompare}
For every $a\geq0$,
\begin{align}
 j_*\left(e_\lambda(E_-|_{M_{F_n}})[M_{F_n}]^\vir\right)
 &=e(E_+)e_\lambda(E_-)[M_{Y_n}]^\vir,\label{eq:pushvirtual}\\
 i_*C_{a,\beta}^\lambda&=H \mathsf E_{a,\beta}^{0,\lambda},\label{eq:pushC}
\end{align}
where
\[
 \mathsf E_{a,\beta}^{\mu,\lambda}
 =(\ev_{Y_n})_*\left(\psi^a e_\mu(E'_+)e_\lambda(E_-)[M_{Y_n}]^\vir\right).
\]
In particular, for every $b\in H^*(Y_n)$,
\begin{equation}\label{eq:ambientpair}
 \int_{F_n} i^*b\,C_{a,\beta}^\lambda
 =\int_Fi^*b\,i^*\mathsf E_{a,\beta}^{0,\lambda}.
\end{equation}

The first identity is also valid with $M_{F_n},M_{Y_n}$ replaced by the corresponding
stacks with any number $\ell\geq0$ of markings, and with the bundles formed
on those stacks. It remains valid after multiplication by products of
cotangent-line classes and evaluation pullbacks. The evaluation identities~\eqref{eq:pushC}--\eqref{eq:ambientpair} use the one-marked notation above.
\end{proposition}
\begin{proof}
We prove the virtual comparison with an arbitrary number $\ell\geq0$ of
markings and then take $\ell=1$ for the evaluation formulas. Write
$M_V^{(\ell)}=\overline{\mathcal M}_{0,\ell}(V,\beta_V)$, where
$\beta_{F_n}=\beta$ and $\beta_{Y_n}=i_*\beta$, and put $m=H\cdot\beta>0$. The letter $V$ is a placeholder for either target and is unrelated to the ambient subspace $V_\amb$.
All Chow groups in the argument have rational coefficients. The equivariant
identities are identities with an indeterminate $\lambda$; no limit of an
equivariant pairing is involved.

\medskip\noindent\textbf{1. Universal bundles and the zero stack.}
On $M_{Y_n}^{(\ell)}$, let $\pi_{Y_n}:\mathcal C_{Y_n}\to M_{Y_n}^{(\ell)}$ and
$f_{Y_n}:\mathcal C_{Y_n}\to Y_n$ be the universal family. By
Lemmas~\ref{lem:tree}--\ref{lem:basechange},
\begin{align*}
 R\pi_{Y_n*}f_{Y_n}^*\mathcal O_{Y_n}(H)&=E_+,
 \qquad R\pi_{Y_n*}f_{Y_n}^*\mathcal O_{Y_n}(-H)=E_-[-1],
 \cr \operatorname{rk}E_+&=m+1,\quad \operatorname{rk}E_-=m-1.
\end{align*}
These are identifications of perfect complexes with arbitrary derived
base change. In applying \cite[Tag~0A1H]{Stacks} on an atlas, the morphism
is the proper flat finitely presented curve family, the perfect object is
the indicated line bundle, and the flat finitely presented complex with
proper support is $\mathcal O_{\mathcal C_{Y_n}}$ in degree zero. Thus the
statement applies to nonreduced and non-Noetherian base changes as well.

Let $s\in H^0(Y_n,\mathcal O_{Y_n}(H))$ define $F_n$. Adjunction gives a section
$\sigma=\pi_{Y_n*}f_{Y_n}^*s$ of $E_+$. For any scheme $S\to M_{Y_n}^{(\ell)}$,
the pullback of $\sigma$ is zero if and only if
$f_S^*s=0$ as a section on $\mathcal C_S$. This is precisely the condition
that $f_S$ factor through the closed subscheme $F_n$ of $Y_n$. The factorization
is unique and is compatible with isomorphisms and base change. Composition
with $i$ contracts exactly the same components, so it preserves stability
and the automorphism group. The class on $F_n$ is uniquely determined by its
pushforward, by Proposition~\ref{prop:geometry}. Consequently
$j:M_{F_n}^{(\ell)}\hookrightarrow M_{Y_n}^{(\ell)}$ is the zero stack of $\sigma$,
as a closed immersion of stacks.

\medskip\noindent\textbf{2. The relative obstruction-theory morphisms.}
Let $\mathfrak B=\mathfrak M_{0,\ell}$ be the smooth Artin stack of connected
prestable marked genus-zero curves. 
One can replace it with a finite-type open substack containing 
the images of the two stable-map stacks, since these images are quasi-compact.
Its dimension is $\ell-3$.
For $V=F_n,Y_n$, the usual relative stable-map obstruction theory is
\[
 \phi_V:\mathbb E_V=(R\pi_{V*}f_V^*T_V)^\vee
       \longrightarrow\mathbb L_{M_V^{(\ell)}/\mathfrak B}.
\]
The relative stable-map obstruction morphism is the construction following
Proposition 4 of~\cite{BehrendGW}; Propositions 2, 4, and 5 give the
prestable base, the stable-map description, and global resolutions,
respectively. We construct its compatibility here.
Here no stabilization of the source is necessary after composition with
$i$. Hence the universal curve over $M_{F_n}^{(\ell)}$ is the actual pullback
of $\mathcal C_{Y_n}$, and derived base change identifies the corresponding
terms in the obstruction theories.

At this point we use the special convexity of the ambient target. Indeed,
$T_{Y_n}=\bigoplus_{r=1}^n\operatorname{pr}_r^*\mathcal O_{\PP^1}(2)$, and each
summand has nonnegative degree on every component of every genus-zero
source. Hence the normalization-sequence argument of Lemma~\ref{lem:tree}
gives $H^1(C,f^*T_{Y_n})=0$. It follows that
$M_{Y_n}^{(\ell)}\to\mathfrak B$ is smooth, with relative tangent bundle
$\pi_{Y_n*}f_{Y_n}^*T_{Y_n}$, and
$\mathbb E_{Y_n}=(\pi_{Y_n*}f_{Y_n}^*T_{Y_n})^\vee=\Omega_{M_{Y_n}^{(\ell)}/\mathfrak B}$.
This convexity statement concerns the ambient target only and does not
assert smoothness of $M_{F_n}^{(\ell)}$.

The normal sequence of the smooth hypersurface,
\[
 0\longrightarrow T_{F_n}\longrightarrow i^*T_{Y_n}
   \xrightarrow{\,ds\,}\mathcal O_{F_n}(H)\longrightarrow0,
\]
gives, after pullback to the universal curve, derived pushforward and
duality, the distinguished triangle
\[
 j^*\mathbb E_{Y_n}\longrightarrow\mathbb E_{F_n}
 \longrightarrow(j^*E_+)^\vee[1]\longrightarrow j^*\mathbb E_{Y_n}[1].
\]
We specify its compatibility with the cotangent triangle. On the zero
stack, the derivative of the universal section in a relative tangent
direction $v\in\pi_*f^*T_{Y_n}$ is $\pi_*(ds(v))$.
Indeed over a square-zero extension, locally writing a lift as
$f+\epsilon v$, the section is
$s(f+\epsilon v)=s(f)+\epsilon ds(v)$; changes of trivialization of
$\mathcal O(H)$ contribute a term proportional to $s(f)$, which vanishes
on $M_{F_n}^{(\ell)}$. Thus this differential is globally well defined and
is exactly the map obtained from the normal sequence.

More explicitly, after a smooth base change from an atlas of $\mathfrak B$
and an étale chart on the resulting smooth ambient algebraic space,
trivialize $E_+$ and write $\sigma=(\sigma_1,\ldots,\sigma_{m+1})$.
For the ideal $I=(\sigma_1,\ldots,\sigma_{m+1})$, the zero-locus
obstruction theory is the morphism
\[
 [\,(j^*E_+)^\vee\xrightarrow{\,d\sigma\,}
                  j^*\Omega_{M_{Y_n}^{(\ell)}/\mathfrak B}\,]
 \longrightarrow \mathbb L_{M_{F_n}^{(\ell)}/\mathfrak B}.
\]
The explicit two-term target on this chart is
\[
 [I/I^2\longrightarrow\Omega_{U/S}|_Z]
       \simeq\tau_{\ge-1}\mathbb L_{Z/S},
\]
where $U\to S$ is the smooth chart just described and $Z=Z(\sigma)$.
It is the truncation, not an identification with the full cotangent
complex unless further lci hypotheses hold. The canonical morphism to
$\mathbb L_{Z/S}$ is supplied by the derived Koszul zero locus; the
following component description is its truncation.
Its degree $-1$ map sends the dual basis to the classes of $\sigma_a$ in
$I/I^2$; its degree zero map is the restriction of relative differentials.
It is an obstruction theory without a regular-sequence hypothesis on
the $\sigma_a$. 
The preceding computation of $d\sigma$ identifies this two-term complex with
$(R\pi_*f^*T_{F_n})^\vee$ and the morphism with $\phi_{F_n}$, since both are
obtained by differentiating the same universal evaluation map subject to the
relation $s\circ f=0$.
These maps and their commutativity
are preserved on overlaps, since they are induced by the universal
evaluation map and $ds$, rather than chosen quasi-isomorphisms.

In particular there is a morphism of distinguished triangles
\[
\begin{array}{ccccc}
 j^*\mathbb E_{Y_n}&\longrightarrow&\mathbb E_{F_n}&\longrightarrow&(j^*E_+)^\vee[1]\\
 \downarrow&&\downarrow&&\downarrow\\
 j^*\mathbb L_{M_{Y_n}^{(\ell)}/\mathfrak B}&\longrightarrow&
 \mathbb L_{M_{F_n}^{(\ell)}/\mathfrak B}&\longrightarrow&
 \mathbb L_{M_{F_n}^{(\ell)}/M_{Y_n}^{(\ell)}}.
\end{array}
\]
This is the compatibility required in
\cite[Definition 4.5]{Manolache}; it is stronger than an equality of
classes in $K$-theory. The relative theory for $j$ is concentrated in
degree $-1$, namely $(j^*E_+)^\vee[1]$.

\medskip\noindent\textbf{3. The virtual cycle and its pushforward.}
The morphisms to $\mathfrak B$ and the closed immersion $j$ are of
Deligne--Mumford type. The stable-map stacks are proper finite-type
Deligne--Mumford stacks over $\C$ and satisfy the quotient-stratification
hypothesis by \cite[Remark 3.2]{Manolache}. The three relative theories
above are perfect of amplitude $[-1,0]$. Thus
\cite[Corollary 4.9]{Manolache}, with the compatible triple
just constructed, applies and gives
\[
 [M_{F_n}^{(\ell)}]^{\vir}=j^!_{E_+}[M_{Y_n}^{(\ell)}]^{\vir}.
\]
The virtual pullback is the one defined by the normal-cone embedding
$C_{M_{F_n}^{(\ell)}/M_{Y_n}^{(\ell)}}\hookrightarrow j^*E_+$ associated with
the section $\sigma$. In particular it is the refined pullback of the
zero section of $E_+$ in the Cartesian zero-locus square. For this
pullback the pushforward identity is
\[
 j_*j^!_{E_+}(\alpha)=c_{m+1}(E_+)\cap\alpha.
\]
One can see the identity by deforming the graph of $\sigma$ to the zero
section through the graphs of $t\sigma$, or by zero-section
self-intersection and compatibility of refined pullback with proper
pushforward. No regularity of $j$ is used. In the present smooth ambient
case the same formula is
$[M_{F_n}^{(\ell)}]^{\vir}=0^!_{j^*E_+}[C_{M_{F_n}^{(\ell)}/M_{Y_n}^{(\ell)}}]$;
this is precisely the zero-section construction of \cite[Section 5]{BF}.
This proves $j_*[M_{F_n}^{(\ell)}]^{\vir}=e(E_+)[M_{Y_n}^{(\ell)}]^{\vir}$.

\medskip\noindent\textbf{4. The negative twist and marked evaluation.}
The universal curves are Cartesian under $j$, and
$i^*\mathcal O_{Y_n}(-H)=K_{F_n}$. Lemma~\ref{lem:basechange} therefore gives the
canonical identification $j^*E_-=R^1\pi_*f^*K_{F_n}=\mathcal V$.
It respects the fiber character of weight $\lambda$. Multiplication by
$e_\lambda(E_-)$ and the projection formula prove~\eqref{eq:pushvirtual}, including its assertion for arbitrary $\ell$.
The cotangent lines and evaluations pull back under $j$, since the
universal marked curve is unchanged. The same comparison consequently
holds with any product of descendants and insertions pulled back from $Y_n$.
Theorem~\ref{thm:virtual} identifies the resulting invariants with those
of $X_n$.

For $\ell=1$, evaluation of $f^*\mathcal O_{Y_n}(H)$ at the marking is
surjective on every fiber and hence over the whole stack. Thus
$0\to E'_+\to E_+\to\ev_{Y_n}^*\mathcal O_{Y_n}(H)\to0$ is an exact sequence
of vector bundles. It implies
$e(E_+)=\ev_{Y_n}^*H\,e(E'_+)$. Push forward~\eqref{eq:pushvirtual}, multiplied
by $\psi^a$, along evaluation, and use
$\ev_{Y_n}\circ j=i\circ\ev_{F_n}$. This proves~\eqref{eq:pushC}.
Finally pairing with $b$ and using
$\int_{Y_n} H b u=\int_{F_n} i^*(bu)$ proves~\eqref{eq:ambientpair}.
Only this numerical pairing is asserted in general; no injectivity of
$i_*$ on the full cohomology of $F_n$ has been used.
\end{proof}

\subsection{The mixed hypergeometric theorem}
For $\mathbf d\in\N^n$, put $D=\sum d_i$ and define
\begin{equation}\label{eq:Imixed}
 I^{\mu,\lambda}(\mathbf x,z)=e^{\sum_{i=1}^n p_i\log x_i/z}
 \sum_{\mathbf d\geq0}\mathbf x^{\mathbf d}
 \frac{\prod_{r=1}^D(\mu+H+rz)
       \prod_{r=0}^{D-1}(\lambda-H-rz)}
      {\prod_{i=1}^n\prod_{r=1}^{d_i}(p_i+rz)^2}.
\end{equation}
Empty products are $1$. The twisting class on the moduli of maps to $Y_n$ is
\[
 e_\mu(R^\bullet\pi_*f^*\mathcal O(H))\,
 e_\lambda^{-1}(R^\bullet\pi_*f^*\mathcal O(-H)).
\]
The pairing is $\int_Yab(\mu+H)/(\lambda-H)$ over $\Q(\mu,\lambda)$.
At positive degree its twisting factor is $e_\mu(E_+)e_\lambda(E_-)$.

\subsection{The quoted mirror-theoretic inputs}\label{sec:mirrorinputs}
Four published mirror-theoretic results are used below in the precise
form needed here. We record their relevant specializations rather than
paraphrasing them, since the hypotheses, normalizations, and product
conventions enter the proof. In our application the target $Y_n$ is a smooth
projective variety with trivial stabilizers, so there is a single sector
and all stabilizer ages vanish. Cone points and
tangent spaces always carry the functorial, dilaton-shifted
formal-neighborhood meaning of \cite[Appendix B]{CCITtwisted}, over a
complete Hausdorff linearly topologized algebra $R$ over the Novikov
ring; $R\{z\}$ denotes series $\sum a_kz^k$ with $a_k\to0$, and
$R_{\rm nilp}$ the topologically nilpotent elements.

\begin{theorem}[Coates--Corti--Iritani--Tseng, {\cite[Corollary 32, with the data of Theorem 31]{CCITtoric}}]
\label{thm:CCITtoric}
Let $\boldsymbol\Sigma=(N,\Sigma,\rho)$ be a stacky fan defining a smooth
toric Deligne--Mumford stack $\mathcal X$ with projective coarse moduli
space, and let $S$ be a finite set with a map to $N_\Sigma$. Then the
$S$-extended nonequivariant $I$-function is a $\Lambda_{\rm nov}[[x,t]]$-valued
point of the nonequivariant genus-zero Givental cone of $\mathcal X$, at
$-z$ in the cone convention.
\end{theorem}

Projectivity implies the semi-projectivity condition of the source, and
completion is in curve degree and input variables. In our application
$\mathcal X=Y_n$ has trivial stabilizers and $S=\varnothing$, so the
twisted-sector and extended-degree data of \cite[Definition 23 and
equation (19)]{CCITtoric} are absent and the general $I$-function reduces to
\begin{equation}\label{eq:toricJ}
 ze^{\sum_{i=1}^n p_it_i/z}\sum_{\mathbf d\geq0}\mathbf q^{\mathbf d}
 e^{\sum_{i=1}^n d_it_i}\prod_{i=1}^n\prod_{r=1}^{d_i}(p_i+rz)^{-2},
\end{equation}
each $\PP^1$ factor contributing two equal divisor classes $p_i$ and
effective degrees being nonnegative tuples. Its positive part is
$z\mathbf1+\sum_{i=1}^np_it_i$, so uniqueness of the cone graph section makes it
the small $J$-function. This conclusion requires no positivity assumption
on $T_{F_n}$.

For the next two statements let $B=\bigoplus_{j=1}^rB_j$ be a sum of line
bundles on $Y_n$, $\rho_j=c_1(B_j)$, and
$c=\exp\bigl(\sum_{k\geq0}s_k\ch_k\bigr)$ with $s_{-1}=0$; the valuation of
$s_k$ is $k+1$, and $B_b(x)$ denotes the Bernoulli polynomial.

\begin{theorem}[Coates--Corti--Iritani--Tseng, {\cite[Theorem 4.1 and Remark 4.2]{CCITtwisted}}]
\label{thm:CCITqrr}
The operator
\[
 \Delta=\prod_{j=1}^r\exp\left(
 \sum_{a,b\geq0}s_{a+b-1}\frac{B_b(0)}{b!}\frac{\rho_j^a}{a!}z^{b-1}\right)
\]
is a symplectic isomorphism from the untwisted to the $c$-twisted loop
space, and $\mathcal L^{c,B}=\Delta\,\mathcal L^{\rm un}$.
\end{theorem}

The pairings vary with $c$, as in Remark 4.2 of the source; identifying them
by $\sqrt{c(IB)}$ instead would change the displayed operator. The general
statement carries a stabilizer age $f_{ij}\in[0,1)\cap\Q$ in place of the
argument $0$ of the Bernoulli polynomial, and is sectorwise over
$I\mathcal X$; for $Y_n$ there is one sector and all ages vanish. The
independent characteristic classes used below are handled by commuting
exponents and a proved induction, not by attributing a two-class statement
to a theorem stated for one class.

For a topological type $\mathfrak t=(0,d,(i_1,\ldots,i_m))$ let $J_\mathfrak t$
be the corresponding part of the untwisted $J$-function, the last marking
being the output. With $N_{\mathfrak t,j}=\rho_j\cdot d\in\Z$ and
$s(w)=\sum_{k\geq0}s_kw^k/k!$, set
\[
 M_\mathfrak t(z)=\prod_{j=1}^r
 \frac{\prod_{a\leq N_{\mathfrak t,j}}\exp s(\rho_j+az)}
      {\prod_{a\leq0}\exp s(\rho_j+az)},\qquad
 I^{c,B}(\tau,z)=\sum_{\mathfrak t} M_{\mathfrak t}(z)J_{\mathfrak t}(\tau,z),
\]
the quotients cancelling tails to give finite products or inverse products.

\begin{theorem}[Coates--Corti--Iritani--Tseng, {\cite[Theorem 4.6]{CCITtwisted}}]
\label{thm:CCIThyper}
The family $\tau\mapsto I^{c,B}(\tau,-z)$ is a point of $\mathcal L^{c,B}$
over $\Lambda[[s_0,s_1,\ldots]][[\tau]]$.
\end{theorem}

No convexity of the individual $B_j$ is assumed. The bundles used here are
actual line bundles on the smooth projective variety $Y_n$, and the proof
below supplies the extension to two independent characteristic classes.

\begin{proposition}[Coates--Corti--Iritani--Tseng, {\cite[Lemma B.1, Propositions B.2--B.3, Corollary B.7]{CCITtwisted}}]
\label{prop:CCITcone}
\begin{enumerate}[label=(\roman*),leftmargin=*]
\item If $I(t)\in\mathcal L^{\rm tw}(R[[t]])$ then
 $I'(t)\in T_{I(t)}\mathcal L^{\rm tw}(R[[t]])$.
\item For the untwisted cone, $f\in\mathcal L(R)$ and $x\in R_{\rm nilp}$
 imply $(1+x)f\in\mathcal L(R)$.
\item $T_f\mathcal L(R)$ is an $R\{z\}$-submodule of the ambient tangent space.
\item If $I(t)\in\mathcal L(R[[t]])$ and $\xi(t)$ is tangent at $I(t)$, then
 $z\xi'(t)$ is tangent at $I(t)$.
\end{enumerate}
\end{proposition}

Item (ii), applied over dual numbers, implies that $f$ itself is tangent.
Item (iv) does not permit arbitrary division by $z$; the proof below checks
precisely the potentially negative power. Finite-parameter quotients and
their inverse limits use coefficient rings of the kind just specified.

The small toric $J$-function of $Y_n$ is therefore the untwisted denominator
series appearing in~\eqref{eq:Imixed}, multiplied by $z$, and hypergeometric
modification of a $J$-family by an invertible multiplicative characteristic
class of a sum of line bundles lies on the corresponding twisted genus-zero
cone. The proof of Theorem~\ref{thm:mirror} supplies the additional
independent-characteristic-class step; it is not hidden inside the statement
of any theorem quoted above.

\subsection{The mixed hypergeometric theorem, continued}

Define
\[
 g(\mathbf x)=\sum_{\mathbf d\ne0}
 \frac{(-1)^D D!(D-1)!}{\prod_{i=1}^n (d_i!)^2}\mathbf x^{\mathbf d},
 \quad Q_i=x_i e^{g(\mathbf x)},\quad
 \tau=\sum_{i=1}^n  p_i\log Q_i-\lambda g(\mathbf x).
\]
\begin{theorem}[Mixed mirror theorem]\label{thm:mirror}
The series $-zI^{\mu,\lambda}(\mathbf x,-z)$ lies on the mixed-twisted
Givental cone. In the convention $J=z\mathbf1+\tau+O(z^{-1})$,
\[
 zI^{\mu,\lambda}=J_{Y_n}^{\mu,\lambda}(\tau,z).
\]
Here Novikov and divisor variables are combined by the divisor equation,
as specified in the proof. Consequently the normalized zero-input series is
\begin{equation}\label{eq:Bdef}
 B^{\mu,\lambda}=e^{-\sum_{i=1}^n  p_i\log Q_i/z}e^{\lambda g/z}
 I^{\mu,\lambda}
 =1+(\lambda-H)\sum_{\beta>0,a\geq0}
 \mathbf Q^\beta z^{-a-2}\mathsf E_{a,\beta}^{\mu,\lambda}.
\end{equation}
\end{theorem}
\begin{proof}
We separate the toric input, the formal product modification, and the
specialization and normalization steps.

\medskip\noindent\textbf{1. Toric input, variables, and formal topology.}
The ambient variety $Y_n=(\PP^1)^n$ is smooth and projective. Its fan is
complete and nonsingular, its stabilizers are trivial, and its anticanonical
class is $2H$. Thus Theorem~\ref{thm:CCITtoric} applies with no extended
variables and with the ambient torus parameters set to zero. It puts the
untwisted toric series on the untwisted cone. Its positive-degree terms
have order at most $z^{-1}$ after multiplication by $z$, since their
denominators have $z$-degree $2D\geq2$. The unique cone section with
expansion $z\mathbf1+\sum_i p_it_i+O(z^{-1})$ is therefore the small
$J$-function, which is
\[
 J_{Y_n}(\mathbf t,z)=ze^{\sum_i p_it_i/z}
 \sum_{\mathbf d\geq0}\mathbf q^{\mathbf d}e^{\sum_i d_it_i}
             \prod_{i=1}^n\prod_{r=1}^{d_i}(p_i+rz)^{-2}.
\]
Here $\mathbf q$ denotes independent Novikov variables. The substitution
$x_i=q_ie^{t_i}$ is a relabeling of divisor-weighted Novikov monomials;
in this notation the customary divisor exponential is recorded using
formal logarithms of $x_i$. 
%It does not add a second Novikov factor.
Equivalently all identities below can be kept in $(\mathbf q,\mathbf t)$
until the last step, so no evaluation of a Novikov series at $q_i=1$ is
required.

Use the full primary input $\mathbf t$ on $H^*(Y_n)$ during the cone
argument; its restriction to the displayed divisor coordinates will be
taken afterwards. We use the cone convention in which $J_{Y_n}(\mathbf t,-z)$
lies on $\mathcal L_{Y_n}$. Formal points and tangent vectors are those of
Section~\ref{sec:mirrorinputs}; positive powers of $z$ are permitted when
their coefficients tend to zero in the coefficient-ring topology.
Initially retain finitely many characteristic parameters and work modulo
a power of their maximal ideal and of the Novikov/input ideal. The
differential operators below are then finite-order operators at each
characteristic-parameter degree. All constructions are compatible with
these quotients and with enlarging the finite set of parameters. This
specifies the inverse-limit interpretation of the subsequent series.

\medskip\noindent\textbf{2. Two independent characteristic classes.}
Put $\rho_+=H$, $\rho_-=-H$, and $d_+=D$, $d_-=-D$, the cohomology
classes on $Y_n$.
%; they are not the bundles $E_\pm$ of Lemma~\ref{lem:basechange}.
 Introduce parameters
$s_{j,k}$ for $j\in\{+,-\}$ and $k\geq0$, and set
\[
 s_j(w)=\sum_{k\geq0}s_{j,k}\frac{w^k}{k!},\qquad
 G_j(w,z)=\sum_{a,b\geq0}s_{j,a+b-1}
            \frac{B_b(0)}{b!}\frac{w^a}{a!}z^{b-1},\qquad s_{j,-1}=0.
\]
The twist is
$\exp(\sum_{j,k}s_{j,k}\operatorname{ch}_k(R\pi_*f^*\mathcal B_j))$,
where $\mathcal B_\pm=\mathcal O_{Y_n}(\pm H)$. The genus-zero quantum Riemann--Roch formula of
Theorem~\ref{thm:CCITqrr} gives
\[
 \mathcal L_{s_+,s_-}=\Delta\mathcal L_{Y_n},
 \qquad \Delta=\exp G_+(\rho_+,z)\exp G_-(\rho_-,z).
\]
For independent classes, the exponents in the Riemann--Roch calculation
add up, since differentiating the twisting class with respect to
$s_{j,k}$ inserts the Chern character of the corresponding index bundle.
The Riemann--Roch expression for that insertion does not
change when the other twist is present. The two multiplication operators
commute. This proves the product form of $\Delta$. This fact alone does
not prove hypergeometric modification; that is the separate argument
in the next two steps.

The Bernoulli generating series gives
$G_j(w+z,z)-G_j(w,z)=s_j(w)$. Consequently the degree-$\mathbf d$
modification at $-z$ is
\[
 M_{\mathbf d}(-z)=
 \prod_{j=+,-}\exp\{G_j(\rho_j,z)-G_j(\rho_j-d_jz,z)\}.
\]
For positive $d_j$ this telescopes through $d_j$ terms; for negative
$d_j$ the same identity gives the inverse product with the complementary
endpoints. In particular the formula does not require either $\mathcal B_j$ to
be nef. The standard hypergeometric theorem with a single characteristic
class is Theorem~\ref{thm:CCIThyper}; we now supply its
independent-class extension explicitly.

\medskip\noindent\textbf{3. Reduction to differential operators on the untwisted cone.}
Let $\nabla_j$ be the constant primary direction $\rho_j$, so on the divisor
slice $\nabla_+=\sum_{i=1}^n\partial_{t_i}$ and $\nabla_-=-\nabla_+$.
The divisor equation, applied to the degree-$\mathbf d$ part of the full
$J$-family, reads
\[
 -z\nabla_j J_{Y_n,\mathbf d}(\mathbf t,-z)
     =(\rho_j-d_jz)J_{Y_n,\mathbf d}(\mathbf t,-z).
\]
The descendant term at the output marking supplies the cup-product
term $\rho_j$; all other inputs are primary. This also follows directly
on the divisor slice from $e^{-\sum p_it_i/z}e^{\sum d_it_i}$.
The sign here is fixed by using $J(\mathbf t,-z)$.
After applying $\Delta^{-1}$, the modified series is therefore
\[
 \mathfrak F_s=\exp\{-G_+(-z\nabla_+,z)-G_-(-z\nabla_-,z)\}
                         J_{Y_n}(\mathbf t,-z).
\]
All operators in this exponent commute. For each retained parameter,
\[
 \partial_{s_{j,k}}\mathfrak F_s=D_{j,k}\mathfrak F_s,\qquad
 D_{j,k}=-\sum_{b=0}^{k+1}
 \frac{B_b(0)}{b!(k+1-b)!}
               (-z\nabla_j)^{k+1-b}z^{b-1}.
\]
We will prove that $\mathfrak F_s$ is a point of the untwisted cone.

\medskip\noindent\textbf{4. Tangency and the complete induction.}
We use the four statements of Proposition~\ref{prop:CCITcone}:
item (i) for derivatives of cone-valued families; item (ii) for conical
scaling; item (iii) for multiplication of tangent vectors by nonnegative
powers of $z$; and item (iv) for the operation $\xi\mapsto z\partial_t\xi$
on a tangent family along a cone-valued family.
These statements hold over the formal coefficient rings specified above.
For a cone-valued family $f(\mathbf t)$, every term in $D_{j,k}f$
is tangent to the cone along $f$. The potentially negative power of
$z$ can be checked by setting $b=0$, in which case the operator reduces
to a scalar multiple of $\nabla_j(z\nabla_j)^k$.
Start with the tangent vector $\nabla_jf$ and
apply Corollary B.7 exactly $k$ times. If $b\geq1$ and
$a=k+1-b\geq1$, the operator is, up to sign,
$z^b(z\nabla_j)^{a-1}\nabla_j$; the same argument and Proposition B.3
apply. If $a=0$, it is a multiple of $z^k f$, which is tangent by
Propositions B.2 and B.3. Thus we never divide an arbitrary tangent
vector by $z$.

The induction is on ordinary total degree in a fixed finite
set of characteristic parameters, not on their indices. In the formal
polarization write the graph equations of $\mathcal L_{Y_n}$ as
$\mathscr E(f)=0$, meaning that the negative component is the differential
of the genus-zero descendant potential evaluated at the positive
component after the dilaton shift. At $s=0$ the equations hold.
Assume $\mathscr E(\mathfrak F_s)\in(s)^r$ for $r\geq1$.
Take $\widetilde{\mathfrak F}_s$ to be the graph point with the same
positive component as $\mathfrak F_s$. Then
$\mathfrak F_s-\widetilde{\mathfrak F}_s\in(s)^r$.
All $D_{j,k}$ differentiate only the primary input and hence preserve
$(s)^r$. The chain rule and the preceding tangency calculation give
\[
 \partial_{s_{j,k}}\mathscr E(\mathfrak F_s)
 =d\mathscr E_{\mathfrak F_s}(D_{j,k}\mathfrak F_s)
 =d\mathscr E_{\widetilde{\mathfrak F}_s}
       (D_{j,k}\widetilde{\mathfrak F}_s)+O((s)^r)
 =O((s)^r).
\]
Each homogeneous monomial of degree $r$ has a nonzero derivative of
degree $r-1$ with respect to some parameter. Characteristic zero now
implies that the homogeneous part of degree $r$ in every graph equation
vanishes. Thus $\mathscr E(\mathfrak F_s)\in(s)^{r+1}$, proving the
induction. It includes mixed monomials in the two parameter sets.
Passing to the compatible completions proves cone membership. Applying
$\Delta$ proves the two-class hypergeometric assertion and restricting
the full primary input to divisors gives the required family.

\medskip\noindent\textbf{5. Euler and inverse-Euler specialization.}
Specialize to $c_+(w)=\mu+w$ and $c_-(w)=(\lambda+w)^{-1}$:
\[
\begin{array}{lll}
 s_{+,0}=\log\mu,&
 s_{+,k}=(-1)^{k-1}(k-1)!\mu^{-k}&(k\geq1),\\
 s_{-,0}=-\log\lambda,&
 s_{-,k}=(-1)^k(k-1)!\lambda^{-k}&(k\geq1).
\end{array}
\]
The formal meaning is the equivariant characteristic-class specialization
described in \cite[Section 2.3]{CCITtwisted}. Its validity in this argument is justified as follows.
The rank contributions are first written as powers of $e^{s_{+,0}}$ and
$e^{s_{-,0}}$, then specialized to powers of $\mu$ and $\lambda^{-1}$.
The remaining operations are expansions in inverse weights. At fixed
curve degree and fixed insertions the non-equivariant Chern classes have
bounded degree on the proper moduli stack. Thus the resulting correlators
are rational functions of the weights; their expansions determine them
injectively. The same is true of each finite hypergeometric ratio.
The assertion is therefore a rational-function identity, not an analytic
substitution of logarithms into an arbitrary characteristic-parameter
power series. 
Equivalently, the cone graph equations can be specialized
coefficientwise. At fixed curve degree, the positive-degree
input has positive Novikov order, so only finitely many
such inputs contribute. Once the remaining input degree is
fixed, dimension bounds ensure that only finitely many
Chern and descendant operations are nonzero.
This justifies specializing the formal identity obtained in Step 4.

At $z$ the two finite ratios are
\[
 \frac{\prod_{r=-\infty}^{D}(\mu+H+rz)}
      {\prod_{r=-\infty}^{0}(\mu+H+rz)}
       =\prod_{r=1}^{D}(\mu+H+rz),
 \]
 \[
 \frac{\prod_{r=-\infty}^{-D}(\lambda-H+rz)^{-1}}
      {\prod_{r=-\infty}^{0}(\lambda-H+rz)^{-1}}
       =\prod_{r=0}^{D-1}(\lambda-H-rz).
\]
The infinite-product notation only cancels identical tails. The inverse
Euler twist of the negative line contributes the second numerator, since
its index bundle is $-E_-$ in positive degree. In degree zero that index
bundle is $\mathcal B_-$ instead, yielding the reciprocal in the pairing.
The two factors produce exactly~\eqref{eq:Imixed}, proving the asserted
cone membership over $\Q(\mu,\lambda)$.

\medskip\noindent\textbf{6. The normalized $J$-section and the evaluation series.}
For $D>0$ the numerator has $z$-degree $2D-1$, while the denominator has
degree $2D$. The factor at $r=0$ in the negative product is $\lambda-H$;
the leading coefficient of the summand is
\[
 \frac{(\lambda-H)(-1)^{D-1}D!(D-1)!}{z\prod_{i=1}^n(d_i!)^2}.
\]
It is independent of $\mu$. Consequently
$I^{\mu,\lambda}=1+\tau/z+O(z^{-2})$ with the $\tau$ in the statement.
Uniqueness of the graph section with positive part $z\mathbf1+\tau$
gives $zI^{\mu,\lambda}=J_{Y_n}^{\mu,\lambda}(\tau,z)$ in the positive-$z$
convention. The string equation, with the twist-compatible derivation in
Appendix~\ref{app:forget}, removes the unit input $-\lambda g$,
and the divisor equation removes the divisor exponential, giving the
first expression in~\eqref{eq:Bdef}.

To verify the second expression without taking a degenerate limit, define
the ordinary evaluation pushforward
\[
 \mathsf V_{a,\beta}^{\mu,\lambda}
  =(\ev_{Y_n})_*\bigl(\psi^a e_\mu(E_+)e_\lambda(E_-)
                                     [M_{Y_n}]^\vir\bigr).
\]
Setting $w=(\mu+H)/(\lambda-H)$, contraction with the inverse
twisted pairing yields $w^{-1}\mathsf V_{a,\beta}^{\mu,\lambda}$.
Indeed, pairing this class with any $b$ gives
$\int_{Y_n} b\mathsf V_{a,\beta}^{\mu,\lambda}$, as required.
By the marked exact sequence,
$\mathsf V_{a,\beta}^{\mu,\lambda}=(\mu+H)\mathsf E_{a,\beta}^{\mu,\lambda}$.
Thus $w^{-1}\mathsf V=(\lambda-H)\mathsf E$. Finally the expansion
$1/(z(z-\psi))=\sum_{a\geq0}\psi^a z^{-a-2}$ gives precisely~\eqref{eq:Bdef}. Every coefficient of $\mathsf E$ is polynomial in
$\mu,\lambda$, so this identity proves regularity at $\mu=0$ before
specialization. The comparison with $F_n$ then follows from
Proposition~\ref{prop:virtualcompare}; it initially gives a comparison against ambient
classes. Theorem~\ref{thm:fullambient} identifies the evaluation classes as
ambient, and nondegeneracy then gives equality of the full one-point cohomology classes,
including in even dimension.
\end{proof}

\begin{remark}
This is a theorem about a cohomological Givental cone, not a quantum-$K$
cone. The terminology ``quantized symplectic cone'' is not needed.
Equation~\eqref{eq:Bdef} proves regularity at $\mu=0$ before taking a
limit of a degenerate pairing. Combining it with
Proposition~\ref{prop:virtualcompare} proves the local comparison
\emph{against all ambient insertions}. The additional geometric
Theorem~\ref{thm:fullambient} proves ambientness of the full one-point
evaluation series also for odd $n$, and therefore upgrades this comparison
to the full one-point cohomology-valued series. Its proof does not use the mirror theorem.
\end{remark}

\subsection{Monic division and all numerical descendants}
Write $I_{X_n}^{\lambda,\amb}=I^{0,\lambda}$ and $B^\lambda=B^{0,\lambda}$.
After removing the displayed divisor logarithms, all coefficients lie
in $\mathcal R[\lambda][[\mathbf x,z^{-1}]]$; completion means that
polynomial degree in $\lambda$ may grow with the series indices.

\begin{lemma}\label{lem:division}
The quotient $U^\lambda=(B^\lambda-1)/(\lambda-H)$ is uniquely defined
in this completed ring. For every ambient $b$,
\[
 \int_Fi^*b\,i^*U^\lambda
 =\sum_{\beta>0,a\geq0}\mathbf Q^\beta z^{-a-2}
                         \int_Fi^*b\,C_{a,\beta}^\lambda.
\]
\end{lemma}
\begin{proof}
Substitution $\lambda=H$ annihilates every positive-degree term of
$I_{X_n}^{\lambda,\amb}$. The remaining exponentials cancel because
$\log Q_i=\log x_i+g$ and $\sum p_i=H$. Thus $B^H=1$.
The remainder theorem for the monic polynomial $\lambda-H$ applies
coefficientwise even though $H$ is nilpotent. Multiplication by this
monic polynomial is injective in each polynomial coefficient ring,
which also proves uniqueness after completion.
Equation~\eqref{eq:Bdef} identifies the quotient with the generating
series of $\mathsf E_{a,\beta}^{0,\lambda}$ on $Y_n$.
Use~\eqref{eq:ambientpair} term by term.
\end{proof}

On the diagonal set $x_i=y$, $Q_i=Q$, and define
\begin{multline}\label{eq:mirrorcoordinate}
 c_D=\sum_{|\mathbf d|=D}\binom{D}{d_1,\ldots,d_n}^{\!2},\quad
 g(y)=\sum_{D\geq1}\frac{(-1)^Dc_D}{D}y^D,\quad
 R=\log y+g(y),\cr Q=ye^{g(y)},\quad A=\theta_yR.
\end{multline}
Symmetry is applied only after summing over every ordered degree tuple
$\mathbf d$ of a fixed total degree; an individual summand need not be
symmetric. The symmetric-ring calculation uniquely defines $K_j(y)$ by
\begin{equation}\label{eq:Kdef}
 e^{-HR/z}I_{X_n}^{0,\amb}(y,z)
 =1+\sum_{j=2}^{n}\frac{H^j}{z^j}K_j(y).
\end{equation}
There are no omitted higher cohomological powers on $Y_n$.

\begin{theorem}[Uniform descendant extraction]\label{thm:extraction}
The mirror coordinate is invertible as a germ at zero. For every
$0\leq a\leq n-2$,
\begin{equation}\label{eq:extraction}
 {
 \Dcal_a^\Delta(Q)=-\frac{n!}{(n-2-a)!}\,K_{a+2}(y(Q)).}
\end{equation}
These are convergent germs with rational coefficients, and the equality
holds without any assumption about primitive middle cohomology.
\end{theorem}
\begin{proof}
Each positive-degree term of the untwisted-in-$\lambda$ local numerator
contains $-H$. Its degree-zero part is zero, and its degree-one part
is $H(-1)^D\binom D{\mathbf d}^{2}/D$. This proves the formula for $R$
and the vanishing of the $z^{-1}$ term in~\eqref{eq:Kdef}.
Symmetry and homogeneity give~\eqref{eq:Kdef}, with $K_j(0)=0$.
Specializing $(\lambda-H)U^\lambda=B^\lambda-1$ at $\lambda=0$ gives
\[
 -HU^0=\sum_{j=2}^nH^jz^{-j}K_j.
\]
In $\mathcal R^{S_n}$ the kernel of multiplication by $H$ is $\Q H^n$.
Therefore
$U^0=-\sum_{j=2}^nH^{j-1}z^{-j}K_j$ modulo that kernel.
Its restriction to $F_n$ kills $H^n$, so pairing the quotient with
$\delta_a$ gives~\eqref{eq:extraction}, using $\int_FH^{n-1}=n!$.
Only the ambient pairing assertion of Lemma~\ref{lem:division} was used.

For convergence, $c_D\leq n^{2D}$ because the sum of multinomial
coefficients is $n^D$. 
At fixed cohomological degree, expanding the finite products and
their inverses introduces at most a fixed power of
$1+\log(D+1)$ into the bound, since their coefficients are
polynomials in harmonic sums of the form $\sum_{r\leq D}r^{-k}$.
The same applies to the local series and
its exponential normalization. Thus $g$ and every $K_j$ converge for
$|y|<n^{-2}$. Since $Q=y+O(y^2)$, the inverse-function theorem gives
$y(Q)$ and the convergent germs in~\eqref{eq:extraction}.
All operations used rational coefficients.
\end{proof}

\begin{remark}[Why the ambient lift is retained]
For $a=n-2$ the required function is $K_n$. Its coefficient $H^nK_n$
is zero after restriction to $F_n$. Division before restriction recovers
$-H^{n-1}K_n$, which can be integrated on $F_n$.
Thus discarding the ambient top term before division loses a 
numerical descendant. No cancellation of a nilpotent class is being
performed on $F_n$.
\end{remark}

\section{The quantum-Serre coefficient identity}\label{sec:serre}
Define a signed ambient series, initially purely algebraic,
\begin{equation}\label{eq:Isharp}
 \Ish(y,z)=e^{H\log y/z}\sum_{\mathbf d\geq0}(-1)^Dy^D
 \frac{\prod_{r=1}^D(H+rz)^2}
      {\prod_i\prod_{r=1}^{d_i}(p_i+rz)^2}.
\end{equation}
It is useful to distinguish its coefficients from the normalized ones:
\[
 \Ish(y,1)=\sum_{j=0}^nH^jW_j(y),\qquad
 W_j=A v_j\quad(0\leq j\leq n-1),\qquad v_0=1.
\]
Each $W_j$ is a polynomial in $\log y$ with convergent power-series
coefficients. The top coefficient $W_n$ is not needed in the physical
integral over $F_n$.

\begin{theorem}\label{thm:serre}
Coefficientwise in the completed ambient ring,
\[
 z\theta_yI_{X_n}^{0,\amb}=H\Ish,\qquad
 A=\sum_{D\geq0}(-1)^Dc_Dy^D=\theta_yR.
\]
For $0\leq j\leq n-1$,
\begin{equation}\label{eq:vtriangular}
 v_j=\frac{R^j}{j!}
 +\sum_{k=2}^j\frac{R^{j-k}}{(j-k)!}K_k
 +\sum_{k=2}^{j+1}\frac{R^{j-k+1}}{(j-k+1)!}\thQ K_k,
\end{equation}
where $K_k$ is read as a function of $Q$, and empty sums are zero.
\end{theorem}
\begin{proof}
In degree $D$, $z\theta_y$ multiplies by $H+Dz$.
For $D>0$,
\[
 \prod_{r=0}^{D-1}(-H-rz)=(-1)^DH\prod_{r=1}^{D-1}(H+rz).
\]
Multiplying by $H+Dz$ completes the second positive product, giving
$H$ times the degree-$D$ summand of~\eqref{eq:Isharp}. In degree zero
both sides are $He^{H\log y/z}$. This proves the first identity.
Its coefficient of $H$ gives $A=\theta_yR$, while the degree-zero
cohomology part of~\eqref{eq:Isharp} gives its displayed series.
Now differentiate~\eqref{eq:Kdef} at $z=1$ and use
$\theta_y=A\thQ$. Comparing the coefficient of $H^{j+1}$ for
$0\leq j\leq n-1$ gives~\eqref{eq:vtriangular}.
\end{proof}

\begin{corollary}\label{cor:firstlog}
The first logarithmic period satisfies
\[
 \frac{W_1}{A}=R+\thQ K_2
 =R-\frac{(n-2)!}{n!}\thQ\Dcal_0^\Delta.
\]
More explicitly, with $\mathfrak h_D=\sum_{r=1}^D1/r$ and $\mathfrak h_0=0$,
\[
 W_1=A\log y+
 \sum_{D\geq1}(-1)^Dy^D\sum_{|\mathbf d|=D}\binom D{\mathbf d}^2
 \left(2\mathfrak h_D-\frac2n\sum_i \mathfrak h_{d_i}\right).
\]
\end{corollary}
\begin{proof}
Take $j=1$ in~\eqref{eq:vtriangular} and use
Theorem~\ref{thm:extraction}. For the explicit expansion, differentiate
the finite products to first cohomological degree; their linear term
is $2\mathfrak h_DH-2\sum_i \mathfrak h_{d_i}p_i$. Summing over all degree tuples replaces
the summed expression involving the second sum by $\frac Hn\sum_i \mathfrak h_{d_i}$, by permutation symmetry.
\end{proof}
%%%%%%%%%%%%%%%%%%%%%%%%%%%%%%%%%%%%%%%%%%%%%%%%%%%%%%%%%%%%%%%%%%%%%%%%%%%
\section{An all-loop differential operator and full-series certificate}\label{sec:PF}

This section derives the sunset differential operators. 
The derivation   is independent of the geometric mirror theorem and of
any claimed numerical match with a Feynman integral.
The fact that the differential operator is irreducible, and therefore the minimal operator annihilating $A$,
follows by two Fourier transforms, with their extensions checked, from
irreducibility of symmetric powers of the Kloosterman connection. The latter
uses the differential Galois group of the Bessel connection
\cite{Katz87,Katz90,FSY}; the full argument is recorded in
Proposition~\ref{thm:minimalperiod}, and is used nowhere in the arguments
below.

There have been many prior derivations of these operators in the mathematical~\cite{ZhouSurvey}
and physics literature~\cite{Vanhove:2014wqa,Vanhove:2018mto,LairezVanhove,KlemmEtAl,BFKNS}, so the result is not new, but our derivation 
provides a uniform construction for
all $n$, together with the all-degree certificate~\eqref{eq:PFcertificate},
which is what allows the operator to be used without a finite coefficient fit.

Throughout, $\theta_y:=y\partial_y$, and all differential operators are written
with powers of $y$ on the left, so $\theta_yy=y(\theta_y+1)$.

\begin{definition}[Finite operator algorithm]\label{def:operator}
Set $T_0=1$, $T_1=\theta_y/n$, and recursively define
\[
 T_{k+1}=\frac{\theta_yT_k+k\,yT_{k-1}}{n-k}
 \quad(1\leq k<n),\qquad
 \mathscr S_n=n!\left(\theta_yT_n+n\,yT_{n-1}\right).
\]
Expand uniquely
\[
 \mathscr S_n=\sum_{j\geq0}y^jU_j(\theta_y),\qquad
 U_0(\theta_y)=\theta_y^{n+1}.
\]
Only finitely many $U_j$ are nonzero. Define
\begin{equation}\label{eq:Palgorithm}
 {
 P_{n-1}(y,\theta_y)=\theta_y^{n-1}+
 \sum_{j\geq1}y^jU_j(\theta_y)
                     \prod_{r=1}^{j-1}(\theta_y+r)^2.}
\end{equation}
This is a finite construction over $\Q$, for each integer $n\geq4$.
\end{definition}

\begin{theorem}[Full ambient certificate]\label{thm:PF}
The operator $P_{n-1}$ has order $n-1$, and
\begin{equation}\label{eq:PFcertificate}
 {
 P_{n-1}\Ish(y,1)=H^{n-1} y^H=H^{n-1}+H^{n}\log y
 \quad\text{in }\mathcal R[[y]][\log y].}
\end{equation}
Its coefficient of $\theta_y^{n-1}$ is
\begin{equation}\label{eq:leadP}
 \prod_{0\leq j<n/2}\bigl(1+(n-2j)^2y\bigr).
\end{equation}
Thus it has no finite singular point on $y>0$.
The functions $W_0,\ldots,W_{n-2}$ form a local basis of homogeneous
solutions, $P_{n-1}W_{n-1}=1$, and $P_{n-1}W_n=\log y$.
Here $W_n$ is the ambient top coefficient already defined in
Section~\ref{sec:serre}; it must be retained before restriction to $F_n$.
\end{theorem}
\begin{proof}
We first explain $\mathscr S_n$. If $f_1,\ldots,f_n$ solve
$\theta_y^2f_i=-yf_i$ over a commutative coefficient algebra, let
$V_k$ be the average over all subsets of size $k$ of the product in
which those $k$ factors are replaced by $\theta_yf_i$.
The product rule gives
\[
 \theta_yV_k=(n-k)V_{k+1}-k\,yV_{k-1},\quad
 \theta_yV_0=nV_1,\quad\theta_yV_n=-n\,yV_{n-1}.
\]
Eliminating $V_1,\ldots,V_n$ yields
$\mathscr S_nV_0=0$, with exactly Definition~\ref{def:operator}.
Assign weights $1$ to $\theta_y$ and $2$ to $y$.
The commutator lowers weight, so induction gives
$\deg U_j\leq n+1-2j$ and $U_j=0$ for $2j>n+1$.
The leading coefficient of $\mathscr S_n$ is $1$.

Now work in the actual square-free algebra $\mathcal R$ and set
\[
 f_i(y)=\sum_{d\geq0}\frac{(-1)^dy^{d+p_i}}{\Gamma(d+1+p_i)^2}.
\]
Gamma functions of nilpotent variables mean their finite Taylor
expansions, and $y^{p_i}=1+p_i\log y$ with $y>0$, so no branch is chosen
anywhere. The recurrence of $\Gamma$ gives
\[
 (\theta_y^2+y)f_i=\frac{y^{p_i}}{\Gamma(p_i)^2}=0,
\]
because $\Gamma(w)^{-2}$ has a double zero at $w=0$ and $p_i^2=0$.
This is why the calculation uses the square-free algebra, not a
substitution $p_i=H/n$ that would destroy its relations.
Hence $h=\prod_if_i$ satisfies $\mathscr S_nh=0$.

On $y^{D+H}$ define $\Xi=\Gamma(\theta_y+1)^2$ by multiplication by
$\Gamma(D+H+1)^2$. Each multiplier is invertible.
Then
\[
 \Xi y^j\Xi^{-1}=y^j\prod_{r=1}^j(\theta_y+r)^2.
\]
Multiplication of $\Xi h$ by the constant invertible class
$\prod_i\Gamma(1+p_i)^2/\Gamma(1+H)^2$ yields exactly $\Ish(y,1)$
of~\eqref{eq:Isharp}, alternating signs included.
It follows that
\[
 \left[\theta_y^{n+1}+
 \sum_{j\geq1}y^jU_j(\theta_y)\prod_{r=1}^j(\theta_y+r)^2\right]\Ish=0.
\]
Since $\theta_y^2y^j=y^j(\theta_y+j)^2$, the operator in brackets is
$\theta_y^2P_{n-1}$, so $\theta_y^2P_{n-1}\Ish=0$.

Write $P_{n-1}\Ish=\sum_{D\geq0}y^{D+H}R_D$.
For $D>0$, $(D+H)^2R_D=0$ implies $R_D=0$ because $D+H$ is invertible.
For $D=0$, only the constant-degree part $\theta_y^{n-1}$ acts, giving
$R_0=H^{n-1}$. This proves~\eqref{eq:PFcertificate} including its
nonzero source; division by a nilpotent element has not occurred.
The degree bound for $U_j$ proves the order is at most $n-1$, and the
constant-degree coefficient proves it is exactly $n-1$.

For completeness, the highest weighted symbol of $\mathscr S_n$ is
the characteristic polynomial of the displayed $(n+1)$-dimensional
system with commuting variables. Over $\Q(\sqrt{-y})$ its eigenvalues
are $(n-2j)\sqrt{-y}$ for $0\leq j\leq n$ (the symmetric-power
representation of the two eigenvalues $\pm\sqrt{-y}$).
Pairing opposite eigenvalues, with an additional zero eigenvalue when
$n$ is even, shows that the highest coefficients of the $U_j$ are the
coefficients of the corresponding product of
$\theta_y^2+(n-2j)^2y$. The Gamma multiplication and division by
$\theta_y^2$ in~\eqref{eq:Palgorithm} convert this product to~\eqref{eq:leadP}.

Finally $W_j=(\log y)^j/j!+O(y(1+|\log y|)^j)$.
Comparing coefficients of $H^j$ in~\eqref{eq:PFcertificate} gives
$P_{n-1}W_j=0$ for $j<n-1$ and $P_{n-1}W_{n-1}=1$.
The different leading logarithmic powers prove linear independence
of the $n-1$ homogeneous solutions. They form a basis because the
operator has order $n-1$.
\end{proof}

\begin{remark}[The unsigned variable]\label{rem:unsigned}
      We use $y=-t$ throughout. This removes the alternating signs generated by
$t^j=(-1)^jy^j$ and avoids expressions such as $(-y)^H$, whose interpretation
would require a branch choice. The Euler operator is unchanged,
$t\partial_t=y\partial_y$.
\end{remark}

\begin{corollary}[Frobenius normalization]\label{cor:frobenius}
The origin is a regular singular point with indicial polynomial
$\rho^{n-1}$ and a single maximal unipotent logarithmic block.
For any constant $b$, all local solutions of $P_{n-1}f=b$ have the form
\[
 f=bW_{n-1}+\sum_{j=0}^{n-2}b_jW_j.
\]
Their boundary polynomial in $\log y$ determines the $b_j$ uniquely.
In particular a differential equation alone does not specify the
physical solution.
\end{corollary}
\begin{proof}
The coefficient of $\theta^{n-1}$ is nonzero at zero, and all other
coefficient polynomials are holomorphic there; the Euler-operator
form gives regular singularity and the indicial polynomial.
Theorem~\ref{thm:PF} gives the logarithmic basis and a particular
solution. Continuation $\log y\mapsto\log y+2\pi i$ on this basis
has a single Jordan block, as is already seen on its leading
polynomials. The boundary normalization then determines each constant.
\end{proof}

\begin{example}[The two initial specializations]\label{ex:operators} As an
      illustration we specialize to the cases of $n=4$ and $n=5$ in the symbolic
      implementation of the finite construction given in
      Appendix~\ref{app:code}.

\noindent
For $n=4$, the algorithm gives
\begin{equation}\label{e:P3}
 P_3=\theta_y^3+y(20\theta_y^3+30\theta_y^2+18\theta_y+4)
                   +64y^2(\theta_y+1)^3.
\end{equation}
For $n=5$, it gives
\begin{align}\label{e:P4}
P_4&=\theta_y^4+y Q_1(\theta_y)+y^2 Q_2(\theta_y)+y^3Q_3(\theta_y),\\      
 Q_1&=35\theta_y^4+70\theta_y^3+63\theta_y^2+28\theta_y+5,\cr
 Q_2&=259\theta_y^4+1036\theta_y^3+1580\theta_y^2+1088\theta_y+285,\cr
 \nonumber Q_3&=225\theta_y^4+1350\theta_y^3+2925\theta_y^2+2700\theta_y+900.
\end{align}
\end{example}
%%%%%%%%%%%%%%%%%%%%%%%%%%%%%%%%%%%%%%%%%%%%%%%%%%%%%%%%%%%%%%%%%%%%%%%%
\section{The physical integral and its signed Gamma formula}\label{sec:analytic}
The large-momentum behaviour of the sunset family and its
connection to the periods of the associated Calabi--Yau
geometry have been investigated using several approaches.
These include GKZ systems and relative Calabi--Yau periods
in~\cite{KlemmEtAl,BFKNS}, the all-order structure of the
multiloop sunset in~\cite{Vanhove2026}, and the two-loop
case in~\cite{BKV}. We take the asymptotic result of~\cite{Iritani} as our starting point.
%----------------------------------------------------------------------
\subsection{Convergence and the positive Laplace transform}
Define the two-dimensional equal-mass sunset normalization by
\begin{equation}\label{eq:Schwinger}
 I_{\circleddash}^{(n-1)}(s)=\int_{\mathbb R_+^n}e^{-\sum_{i=1}^n x_i+sW(x)}W(x)
                   \prod_{i=1}^n \frac{dx_i}{x_i},\qquad
 W(x)=\left(\sum_{i=1}^n  x_i^{-1}\right)^{-1}.
\end{equation}
Let define the coordinates $u_1,\ldots,u_{n}$
$
u_i=x_i/(x_1+\cdots+x_n)
$ so that the domain of integration becomes
$\Delta_n=\{u_i>0:\sum_{i=1}^n u_i=1\}$,
and put
\[
 \phi(u)=\sum_{i=1}^n u_i^{-1},\qquad
 \Omega=\frac{du_1\cdots du_{n-1}}{u_1\cdots u_n}.
\]

\begin{proposition}\label{prop:analyticdomain}
The Schwinger integral is absolutely convergent for $\Re s\leq n^2$
and holomorphic for $\Re s<n^2$. It is not absolutely convergent for
$\Re s>n^2$. Its analytic continuation to $\C\setminus[n^2,\infty)$ is
\begin{equation}\label{eq:simplex}
 I_{\circleddash}^{(n-1)}(s)=\int_{\Delta_n}\frac{\Omega}{\phi(u)-s}.
\end{equation}
For real $s<n^2$ it also equals
\begin{equation}\label{eq:Laplacepositive}
 \int_0^\infty e^{rs}\left(\int_{\Delta_n}e^{-r\phi(u)}\Omega\right)dr.
\end{equation}
\end{proposition}
\begin{proof}
Set $x_i=t u_i$. Then $\prod dx_i/x_i=(dt/t)\Omega$ and
$W(x)=t/\phi(u)$. The radial integral gives
$\int_0^\infty e^{-t(1-s/\phi)}dt/\phi=1/(\phi-s)$.
Applied to the absolute value of the integrand, the same computation gives
$1/(\phi-\Re s)$ where $\Re s<\phi$ and $+\infty$ where $\Re s\geq\phi$.
By Tonelli's theorem, the Schwinger integral therefore converges absolutely
if and only if $\int_{\Delta_n}\Omega/(\phi-\Re s)$ is finite, and when it
does, Fubini's theorem allows the radial integration to be performed first,
which gives~\eqref{eq:simplex}.
Cauchy--Schwarz gives $\phi\geq n^2$, with equality only at
$u_i=1/n$. 
The required boundary integrability follows from the estimate
$\phi^{-1}\leq\min_i u_i$. On the sector
$u_1\leq\cdots\leq u_n$, we have $u_n\geq 1/n$, which
bounds the density by $n/(u_2\cdots u_{n-1})$.
Successively integrating
$0<u_1<\cdots<u_{n-1}<1$ is finite (each integration cancels the next
denominator). There are finitely many sectors.
For $\Re s<n^2$ the density is bounded by a constant times
$\Omega/\phi$, locally uniformly in $s$.
At $s=n^2$, the minimum of $\phi$ is nondegenerate on the simplex,
so the only new interior singularity is bounded by a constant times
$\|u-u_*\|^{-2}$ in real dimension $n-1$. It is integrable because
$n\geq4$. This endpoint is one of the essential uses of the standing hypothesis
$n\geq4$, alongside the integral curve-lattice identification.
For $\Re s>n^2$, an open set near $u_*$ gives radial
exponential growth, so absolute convergence fails.
For each compact subset of the slit plane,
$|\phi-s|^{-1}\leq C/\phi$, which proves holomorphic continuation
under the integral. Finally the integrand in~\eqref{eq:Laplacepositive} 
is nonnegative and its radial integral is~\eqref{eq:simplex}. Hence Tonelli's theorem permits us to interchange the radial and simplex integrals.
\end{proof}
The simplex representation also gives the Taylor expansion at $s=0$ and its
behaviour on the circle of convergence. Write $u_*=(1/n,\ldots,1/n)$ for the
barycentre of $\Delta_n$.

\begin{proposition}\label{prop:sTaylor}
For $|s|<n^2$,
\begin{equation}\label{eq:sTaylor}
 I_{\circleddash}^{(n-1)}(s)=\sum_{m\geq0}a_m s^m,\qquad
 a_m=\int_{\Delta_n}\phi(u)^{-m-1}\,\Omega>0,
\end{equation}
and, as $m\to\infty$,
\begin{equation}\label{eq:sTaylorasympt}
 a_m={\pi^{n-1\over 2}\,n^{n/2-2-2m}\over m^{n-1\over 2}}\,\bigl(1+O(m^{-1})\bigr).
\end{equation}
Consequently the Taylor series of $I_{\circleddash}^{(n-1)}$ at $s=0$ has
radius of convergence exactly $n^2$, and it converges absolutely and
uniformly on $|s|\leq n^2$. For an integer $j\geq0$, the $j$th derivative
has a finite limit as $s\uparrow n^2$ along the real axis if and only if
$j<(n-3)/2$. In particular, for $n=4$ and $n=5$ the integral is continuous,
but not $C^1$, at the thresholds $s=16$ and $s=25$.
\end{proposition}
\begin{proof}
Since $\phi\geq n^2>|s|$ on $\Delta_n$, the expansion
$(\phi-s)^{-1}=\sum_{m\geq0}s^m\phi^{-m-1}$ is dominated by
$(\phi-|s|)^{-1}\leq n^2(n^2-|s|)^{-1}\phi^{-1}$, and
$\int_{\Delta_n}\Omega/\phi=I_{\circleddash}^{(n-1)}(0)$ is finite by
Proposition~\ref{prop:analyticdomain}. Integrating~\eqref{eq:simplex} term by
term gives~\eqref{eq:sTaylor}.

For~\eqref{eq:sTaylorasympt}, write $u=u_*+v$ with $\sum_iv_i=0$ and use
$v'=(v_1,\ldots,v_{n-1})$ as coordinates. Since $\sum_iv_i=0$,
\[
 \phi(u_*+v)=\sum_{i=1}^n\frac{n}{1+nv_i}=n^2+n^3\sum_{i=1}^nv_i^2+O(|v|^3),
 \qquad \sum_{i=1}^nv_i^2=v'^{\mathsf T}(I+J)\,v',
\]
where $J$ is the $(n-1)\times(n-1)$ matrix with all entries $1$, so that
$\det(I+J)=n$. The density of $\Omega$ at $u_*$ is $n^n$. By the
Cauchy--Schwarz argument in the proof of
Proposition~\ref{prop:analyticdomain}, $u_*$ is the only point of
$\Delta_n$ at which $\phi=n^2$. Hence for every neighbourhood $N$ of $u_*$
there is $\delta>0$ with $\phi\geq n^2(1+\delta)$ on $\Delta_n\setminus N$,
and the contribution of $\Delta_n\setminus N$ to $a_m$ is at most
$(n^2(1+\delta))^{-m}I_{\circleddash}^{(n-1)}(0)$. On $N$ write
$\phi^{-m-1}=n^{-2m-2}\exp(-(m+1)\psi)$ with
$\psi=\log(\phi/n^2)=n\sum_iv_i^2+O(|v|^3)$. The multidimensional Laplace
method, for a smooth phase with a nondegenerate minimum and a smooth
amplitude, gives
\[
 a_m=n^{n-2m-2}\,\left(\frac{\pi}{nm}\right)^{n-1\over 2}\det(I+J)^{-1/2}
 \bigl(1+O(m^{-1})\bigr),
\]
which is~\eqref{eq:sTaylorasympt}. The root test now gives the radius $n^2$.
On $|s|\leq n^2$ the terms are bounded by $a_mn^{2m}$, which is of exact
order $m^{-(n-1)/2}$ and hence summable since $n\geq4$. For the $j$th
derivative, the terms $m(m-1)\cdots(m-j+1)a_ms^{m-j}$ are nonnegative and
nondecreasing in $s\in[0,n^2)$, so by monotone convergence the limit as
$s\uparrow n^2$ equals $\sum_m m(m-1)\cdots(m-j+1)a_mn^{2(m-j)}$, whose terms
are of exact order $m^{j-(n-1)/2}$. This is finite exactly when
$j<(n-3)/2$.
\end{proof}

\begin{remark}[First correction]\label{rem:sTaylorcorrection}
The remainder in~\eqref{eq:sTaylorasympt} can be made explicit to the next
order:
\begin{equation}\label{eq:sTaylorcorrection}
 a_m={\pi^{n-1\over 2}\,n^{n/2-2-2m}\over m^{n-1\over 2}}
 \left(1-\frac{3(n-1)}{8m}+O(m^{-2})\right).
\end{equation}
Unlike Proposition~\ref{prop:sTaylor}, the argument uses the Bessel
representation~\eqref{eq:Bessel} of Section~\ref{sec:checks}; this remark is
not used in any proof. Replacing $J_0(\sqrt T\,r)$ in~\eqref{eq:Bessel} by
$I_0(\sqrt s\,r)$ gives a function of $s$ that is holomorphic on $|s|<n^2$,
because $|I_0(\sqrt s\,r)|\leq e^{\sqrt{|s|}\,r}$ and
$K_0(r)^n=O(r^{-n/2}e^{-nr})$ as $r\to\infty$. Since $J_0(\mathrm iz)=I_0(z)$,
it agrees with $I_{\circleddash}^{(n-1)}$ for $s=-T\leq0$, and hence on the
whole disk. Expanding $I_0$ termwise gives
\[
 a_m=\frac{2^{n-1}}{4^m(m!)^2}\int_0^\infty r^{2m+1}K_0(r)^n\,dr .
\]
By~\cite[\S10.40]{ext:DLMF},
$K_0(r)^n=(\pi/(2r))^{n/2}e^{-nr}\bigl(1-n/(8r)+O(r^{-2})\bigr)$ uniformly for
$r\geq1$, and the integrals over $[0,1]$ are negligible compared with the
main term. Hence
\[
 \int_0^\infty r^{2m+1}K_0(r)^n\,dr=\Bigl(\frac\pi2\Bigr)^{n/2}
 \frac{\Gamma(2m+2-n/2)}{n^{2m+2-n/2}}
 \left(1-\frac{n^2}{8(2m+1-n/2)}+O(m^{-2})\right).
\]
Finally
$\binom{2m}m=4^m(\pi m)^{-1/2}\bigl(1-\frac1{8m}+O(m^{-2})\bigr)$ and
$\Gamma(2m+2-n/2)/\Gamma(2m+1)
=(2m)^{1-n/2}\bigl(1+\frac{(2-n)(4-n)}{16m}+O(m^{-2})\bigr)$. The leading
factors reproduce~\eqref{eq:sTaylorasympt}, and the three first-order terms
add up to
\[
 \frac1{8m}\left(-1+\frac{(2-n)(4-n)}2-\frac{n^2}2\right)=-\frac{3(n-1)}{8m}.
\]
\end{remark}

\begin{remark}\label{rem:pseudothresholds}
By~\eqref{eq:leadP}, the operator of Theorem~\ref{thm:PF} is also singular at
the pseudothresholds $s=(n-2j)^2$, $1\leq j<n/2$, which lie inside the disk
of Proposition~\ref{prop:sTaylor}. 
These points are genuine singularities of the operator and have the
same local exponents as the threshold. For $n=4$, the local exponents
at both $y=-1/16$ and $y=-1/4$ are
$\{0,\frac12,1\}$. For $n=5$, the local exponents at
$y=-1/25$, $y=-1/9$, and $y=-1$ are
$\{0,1,1,2\}$.
Nevertheless, Proposition~\ref{prop:analyticdomain} shows that the physical
integral is holomorphic at the pseudothresholds. The vanishing there concerns
only its singular component, not the singularity of the operator itself.
The exponent $(n-3)/2$ in
Proposition~\ref{prop:sTaylor} is the familiar threshold exponent of
$n$-particle phase space in two dimensions. The disk of
Proposition~\ref{prop:sTaylor} lives in the physical variable $s$ and should
not be confused with the $Q$-Taylor disk of Theorem~\ref{thm:Qradius}.
\end{remark}

\subsection{The analytic literature input and the degree sign}
Define
\[
 I_{F_n}(r\mathbf1,z)=e^{H\log r/z}
 \sum_{\mathbf d\geq0}r^D
 \frac{\prod_{k=1}^D(H+kz)}
      {\prod_{i=1}^n\prod_{k=1}^{d_i}(p_i+kz)^2}.
\]

We quote the result of Iritani~\cite[Section 2, equation (10)]{Iritani},
who studies the large complex structure asymptotics of the same family of amplitudes. 

\begin{theorem}[Iritani, {\cite[Section 2, equation (10)]{Iritani}}]\label{input:Gamma}
For this smooth Fano hypersurface, the positive oscillatory integral is
\[
 \int_{\Delta_n}e^{-r\phi(u)}\Omega
 =\int_{F_n} I_{F_n}(r\mathbf1,-1)\,\widehat\Gamma_{F_n}\qquad(r>0).
\]
This is the specialization to $q_1=\cdots=q_n=r$ of
\cite[Section 2, equation (10)]{Iritani}, which states, for a smooth Fano
$F_n\subset(\PP^1)^{n}$ of degree $(1,\ldots,1)$ and $q_1,\ldots,q_{n}>0$,
\[
 \int_{\mathbb R_{+}^{n-1}}e^{-\phi_q(w)}\prod_{i=1}^{n-1}\frac{dw_i}{w_i}
                   =\int_{F_n} I_{F_n}(q,-1)\widehat\Gamma_{F_n},
\]
\[
 \phi_q(w)=\left(\sum_{i=1}^{n-1}q_iw_i+q_n\right)
           \left(\sum_{i=1}^{n-1}w_i^{-1}+1\right),
\]
with ordinary cohomological Chern roots in the Gamma class. We use this identity
only as an analytic input. More precisely, equation~(10) of~\cite{Iritani} gives
the positive-real oscillatory integral in the $(n-1)$ independent variables
$w_i$, paired with the $\widehat\Gamma$-class of the degree-$(1,\ldots,1)$ Fano
hypersurface. Specializing $q_1=\cdots=q_n=r$ and making the change of variables
$$
u_i=\frac{w_i}{1+\sum_{j=1}^{n-1}w_j},
\qquad
u_n=\frac{1}{1+\sum_{j=1}^{n-1}w_j},
$$
identifies its positive-real integral with
$$
\int_{\Delta_n}e^{-r\phi(u)}\Omega .
$$
Thus the only external analytic identity used here is the corresponding
positive-real $\widehat\Gamma$-pairing. The passage from this identity to the
physical Schwinger integral is not quoted from~\cite{Iritani}: it is the
Laplace-transform argument of Theorem~\ref{thm:laplace}, including the factor
$(-1)^D$ derived in~\eqref{eq:laplace-degree-sign}.
\end{theorem}

This is the only analytic input in the paper. Every other external result quoted
here is a foundational statement about virtual classes, including the comparison
of algebraic and symplectic virtual classes (Theorem~\ref{thm:siebert}), about
toric mirror symmetry and its twisted versions (Section~\ref{sec:mirrorinputs}),
about quantum $K$-theory (Theorem~\ref{thm:chou}), or about the symplectic
integrality of Gromov--Witten invariants (Theorem~\ref{thm:IP}). The entire
passage from the Feynman integral to Gromov--Witten theory goes through this one
identity. It is the specialization to the present family of the
$\widehat\Gamma$-integral structure of a Fano manifold, in which the oscillatory
integral of the Landau--Ginzburg mirror over its positive real form is paired
against $\widehat\Gamma_{F_n}$. We apply the cited result under its standing
hypotheses and with the convention
$
\widehat\Gamma(E)=\prod_j\Gamma(1+\delta_j),
$
for ordinary cohomological Chern roots \(\delta_j\), with no inserted factors of
\(2\pi i\). Let \(\epsilon\) be the cohomological involution \(p_i\mapsto-p_i\);
it acts trivially on scalar logarithms and on the variable \(y\).

\medskip

\begin{proposition}[Uniform Gamma polynomial]\label{prop:Gamma}
On $F_n$,
\begin{equation}\label{eq:Gammapoly}
 \widehat\Gamma_{F_n}\Gamma(1-H)=\widehat\Gamma_{X_n}|_{F_n}
 =\sum_{k=0}^{n-1}\gamma_kH^k,
 \qquad
 \sum_{k\geq0}\gamma_kt^k
 =\exp\left(\sum_{r\geq1}\frac{2\zeta(2r+1)}{2r+1}t^{2r+1}\right).
\end{equation}
Thus $\gamma_0=1$, $\gamma_1=\gamma_2=0$,
$\gamma_3=2\zeta(3)/3$, $\gamma_5=2\zeta(5)/5$,
and $\gamma_6=2\zeta(3)^2/9$.
\end{proposition}
\begin{proof}
The tangent $K$-class is
$[T_{X_n}|_{F_n}]=[T_{Y_n}|_{F_n}]-[\mathcal O(H)]+[\mathcal O(-H)]$.
Therefore the Gamma class equals
$\prod_i\Gamma(1+2p_i)\,\Gamma(1-H)/\Gamma(1+H)$.
In the Taylor expansion
$\log\Gamma(1+w)=-\gamma_Ew+\sum_{k\geq2}(-1)^k\zeta(k)w^k/k$,
the factor $p_i^2=0$ makes the first logarithm sum to $-2\gamma_EH$.
The Gamma ratio cancels this linear term and all even powers,
leaving $\sum_{r\geq1}2\zeta(2r+1)H^{2r+1}/(2r+1)$.
Truncate above degree $n-1$ and exponentiate. This also explains why
products of odd zeta values first occur when $n\geq7$.
\end{proof}

We need the following theorem about the
Laplace transform which is not in Iritani's paper~\cite{Iritani}.
For the relation between the conventions in that work and the present work we 
refer to remark~\ref{rem:analytic-source-sign}.

\begin{theorem}[Signed Laplace/Gamma formula]\label{thm:laplace}
Let $s=-T<0$ and $y=T^{-1}=-1/s>0$, with the real logarithm.
For $T$ sufficiently large (in particular for $T>n^2$),
\begin{equation}\label{eq:physicalGamma}
      {
 I_{\circleddash}^{(n-1)}(-1/y)=y\int_{F_n}\epsilon\bigl(\Ish(y,1)\bigr)
                           \widehat\Gamma_{X_n}|_{F_n}.}
\end{equation}
It is an identity of convergent logarithmic germs. In particular
\begin{equation}\label{eq:physicalPF}
 P_{n-1}\left(\frac{I_{\circleddash}^{(n-1)}(-1/y)}{y}\right)=(-1)^{n-1} n!.
\end{equation}
Its leading boundary polynomial is
\begin{equation}\label{eq:boundary}
 I_{\circleddash}^{(n-1)}(-1/y)=n!y\sum_{k=0}^{n-1}
 \frac{(-1)^{n-1-k}\gamma_k}{(n-1-k)!}(\log y)^{n-1-k}
 +O\bigl(y^2(1+|\log y|)^{n-1}\bigr).
\end{equation}
\end{theorem}
\begin{proof}
Insert Theorem~\ref{input:Gamma} into~\eqref{eq:Laplacepositive}.
The degree-$D$ coefficient has $r^{D-H}$, so its Laplace integral is
\[
 \int_0^\infty e^{-Tr}r^{D-H}dr
 =T^{-1-D+H}\Gamma(1+D-H).
\]
This identity for a nilpotent $H$ means finitely many derivatives of
the elementary Gamma integral. The key identity is
\begin{equation}\label{eq:laplace-degree-sign}
 \Gamma(1+D-H)=(-1)^D\Gamma(1-H)\prod_{k=1}^D(H-k).
\end{equation}
Multiplying the existing Fano numerator $\prod(H-k)$ gives the square
and the factor $(-1)^D$. Thus the result is
\[
 y\int_{F_n} y^{-H}\sum_{\mathbf d\geq0}(-1)^Dy^D
 \frac{\prod_{k=1}^D(H-k)^2}{\prod_i\prod_{k=1}^{d_i}(p_i-k)^2}
 \widehat\Gamma_{F_n}\Gamma(1-H),
\]
which is precisely~\eqref{eq:physicalGamma}.

To justify termwise integration, expand only to the finitely many
cohomological degrees that can contribute. The Fano scalar coefficients
are $D!/\prod d_i!^2$; integrating $r^D$ multiplies by
$D!T^{-D-1}$. Logarithmic derivatives introduce only fixed powers of
$1+\log(D+1)+|\log T|$; the same holds for finite product expansions.
The sum of the resulting absolute integrals is bounded by a constant
times $T^{-1}\sum_D(n^2/T)^D$ times such a fixed logarithmic power.
It converges for $T>n^2$, giving the required interchange.

Apply $P_{n-1}$ to~\eqref{eq:physicalGamma} and use
Theorem~\ref{thm:PF}. On $F_n$, the top-degree term is
$\epsilon(H^{n-1} y^H)=(-1)^{n-1}H^{n-1}$; the Gamma factor contributes only its
constant term. This gives~\eqref{eq:physicalPF}.
For~\eqref{eq:boundary} retain the degree-zero Novikov term
$y^{-H}$; all other terms satisfy the stated convergent-germ bound.
\end{proof}

\begin{remark}[The analytic source and the sign conventions]
\label{rem:analytic-source-sign}
Three separate conventions must be matched before any displayed formula is
transcribed between the present work and the one of Iritani~\cite{Iritani}.

First, the denominator in~\eqref{eq:simplex} is $\phi-s$, whereas the
amplitude in~\cite{Iritani} uses $s-\phi$; this is an overall sign,
independent of the curve degree. Second, the period normalizations
$\Pi_{n-1}=(-1)^{n-1}yA$ and $\pi=-yA$ differ with loop parity, as detailed in
Section~\ref{sec:dictionary}. Third, and least visible, is the degree
factor itself.
The degree factor can be incorporated in either of two equivalent ways.
Here it appears explicitly in the series~\eqref{eq:Isharp}, which contains
$(-1)^D$, with $y>0$ and the real logarithm understood.
Alternatively one may evaluate an unsigned
squared-numerator series at a negative argument, since
$(-x)^{\mathbf d}=(-1)^{D}x^{\mathbf d}$ when $|\mathbf d|=D$; this
produces the same factor automatically, but then leaves the branch of
$x^H$ at negative $x$ to be fixed (or its cohomological phase to be
removed separately). 
The displayed formula in Iritani's Theorem~8 therefore does not exhibit an additional
explicit $(-1)^D$ in that convention. 
 In the convention of the present paper the termwise Laplace transform of equation~(10) produces the factor $(-1)^D$ in~\eqref{eq:laplace-degree-sign};
this is the convention used in Theorem~\ref{thm:laplace} and in the subsequent numerical checks.
\end{remark}
%%%%%%%%%%%%%%%%%%%%%%%%%%%%%%%%%%%%%%%%%%%%%%%%%%%%%%%%%%%%%%%%%%%%%%%%
\section{The uniform descendant assembly}\label{sec:assembly}
This section gives the all-loop analytic formula~\eqref{eq:allloop} for the sunset integral, and its zeta-free bracket~\eqref{eq:rationalbracket} with rational coefficients.
The proof uses the uniform Gamma polynomial of Proposition~\ref{prop:Gamma} and the uniform differential operator of Theorem~\ref{thm:PF}, together with the top-degree extraction of
Theorem~\ref{thm:extraction} and the triangular relation of Theorem~\ref{thm:serre}.
This expansion near the large-momentum limit, where the monodromy is maximally
unipotent, is the main result of the paper. It is the uniform all-loop form of
the two-loop theorem of~\cite{BKV}. Its analytic shape is that of
\cite[Theorem 4.1]{Vanhove2026}, proved there by a Mellin--Barnes residue
computation; what is established here is the enumerative meaning of every
coefficient, as set out in the introduction.

Define the signed period prefactor
\[
 \Pi_{n-1}(s)=(-1)^{n-1} y A(y),\qquad y=-1/s.
\]
Thus $dR/ds=yA(y)=(-1)^{n-1}\Pi_{n-1}$.
In particular, $\Pi_3=-yA$ and $\Pi_4=yA$.
For $0\leq j\leq n-1$ define
\begin{equation}\label{eq:Vj}
\begin{split}
 \mathcal T_j(R,Q)={}&\frac{n!}{j!}R^j
 -\sum_{a=0}^{j-2}\frac{(n-2-a)!}{(j-a-2)!}
                      R^{j-a-2}\Dcal_a^\Delta(Q)\\
 &-\sum_{a=0}^{j-1}\frac{(n-2-a)!}{(j-a-1)!}
                      R^{j-a-1}\thQ\Dcal_a^\Delta(Q).
\end{split}
\end{equation}
An empty sum is zero, so $\mathcal T_0=n!$.

\begin{theorem}[All-loop analytic--enumerative identity]\label{thm:assembly}
As convergent logarithmic germs at $s=-\infty$,
\begin{equation}\label{eq:allloop}
      {
 I_{\circleddash}^{(n-1)}(s)=\Pi_{n-1}(s)\sum_{k=0}^{n-1}(-1)^k\gamma_k
              \mathcal T_{n-1-k}(R(s),Q(s)).}
\end{equation}
Its zeta-free bracket is
\begin{equation}\label{eq:rationalbracket}
 \mathcal T_{n-1}=n\,R^{n-1}+\sum_{j=0}^{n-2}h_j(Q)R^j,
\end{equation}
where
\begin{equation}\label{eq:hj}
 h_{n-2}=-\thQ\Dcal_0^\Delta,\qquad
 h_j=-(j+1)\Dcal_{n-3-j}^\Delta
                 -\thQ\Dcal_{n-2-j}^\Delta\quad(0\leq j\leq n-3).
\end{equation}
Here $\theta_Q=Q\partial_Q$ and ``zeta-free'' denotes the $k=0$ contribution in
the explicit Gamma expansion, and does not assume any unproved algebraic
independence of zeta values. Each $h_j$ has rational Taylor coefficients and
zero constant term.
\end{theorem}
\begin{proof}
The involution in~\eqref{eq:physicalGamma} acts by $(-1)^j$ on $H^j$.
Top-degree extraction therefore gives
\[
 I_{\circleddash}^{(n-1)}(s)=n!yA\sum_{k=0}^{n-1}(-1)^{n-1-k}\gamma_kv_{n-1-k}
       =\Pi_{n-1}\sum_{k=0}^{n-1}(-1)^k\gamma_k(n!v_{n-1-k}).
\]
Substitute~\eqref{eq:extraction} into~\eqref{eq:vtriangular}.
The resulting expression for $n!v_j$ is exactly~\eqref{eq:Vj}.
For $j=n-1$ the first coefficient is $n!/(n-1)!=n$; the two factorial
ratios simplify respectively to $n-2-a$ and $1$, giving~\eqref{eq:hj}.
Rationality and convergence follow from
Theorem~\ref{thm:extraction}. All signs and the factor of the curve
degree in $\thQ$ are retained.
\end{proof}

\begin{corollary}[Specialization to three loops]\label{cor:nfour}
For $n=4$ put $\Pi_3=-yA$ and
$\delta_0=H^2/2$. Then
\[
 I_3=\Pi_3\bigl[4R^3+f_2R^2+f_1R+f_0-16\zeta(3)\bigr],
\]
with
\[
 f_2=-\thQ\Dcal_0^\Delta,\qquad
 f_1=-2\Dcal_0^\Delta-\thQ\Dcal_1^\Delta,\qquad
 f_0=-\Dcal_1^\Delta-\thQ\Dcal_2^\Delta,
\]
and
$\Dcal_0^\Delta=-12K_2$, $\Dcal_1^\Delta=-24K_3$,
$\Dcal_2^\Delta=-24K_4$.
\end{corollary}
\begin{proof}
Use $n!=24$, $\gamma_3=2\zeta(3)/3$ and
$\mathcal T_0=24$ in~\eqref{eq:allloop}; then apply~\eqref{eq:hj} and~\eqref{eq:extraction}.
\end{proof}

\begin{corollary}[Specialization to four loops]\label{cor:nfive}
For $n=5$ put $\Pi_4=yA$. Then
\[
 I_4=\Pi_4\left[5R^4+\sum_{j=0}^3h_jR^j
      -80\zeta(3)\left(R-\frac1{20}\thQ\Dcal_0^\Delta\right)\right],
\]
where
\[
\begin{array}{ll}
 h_3=-\thQ\Dcal_0^\Delta,&
 h_2=-3\Dcal_0^\Delta-\thQ\Dcal_1^\Delta,\\
 h_1=-2\Dcal_1^\Delta-\thQ\Dcal_2^\Delta,&
 h_0=-\Dcal_2^\Delta-\thQ\Dcal_3^\Delta,
\end{array}
\]
and $(\Dcal_0^\Delta,\Dcal_1^\Delta,\Dcal_2^\Delta,\Dcal_3^\Delta)
=(-20K_2,-60K_3,-120K_4,-120K_5)$.
\end{corollary}
\begin{proof}
Now $n!=120$ and
$\mathcal T_1=120R-6\thQ\Dcal_0^\Delta$.
Only $k=0,3$ contribute in~\eqref{eq:allloop}.
\end{proof}

\begin{remark}[The two initial specializations]\label{rem:lowloop}
Corollaries~\ref{cor:nfour} and~\ref{cor:nfive} are the statements used in the
companion paper~\cite{PaperII}, that gives an independent analytic normalization
and a separate primary coefficient bridge for the $n=4$ identity from the
Picard--Fuchs equation and a Mellin transform, an explicit quantum-$K$ formula
and fourfold integrality statement at $n=4$, the numerical descendant
potentials, and the comparison of the $n=3,4,5$ normalizations with the
local--relative Euler factors. 
\end{remark}

\begin{remark}[What starts at higher loop order]
At $n=6$ the $\zeta(5)$ term first appears, in addition to the
$\zeta(3)$ term multiplying $\mathcal T_2$.
At $n=7$ the product $\zeta(3)^2$ appears.
Thus it would be incorrect to extrapolate the special $n=4,5$
pattern to a claim that only a single $\zeta(3)$ boundary term occurs
for every loop number. Formula~\eqref{eq:allloop} fixes both the zeta
polynomial and the descendant correction multiplying each of its terms.
\end{remark}

\section{A global real mirror coordinate and continuation}\label{sec:real}
Let $\nu$ be normalized Haar measure on $(S^1)^n$ and put
$\mathfrak q(\mathbf w)=|w_1+\cdots+w_n|^2$, so $0\leq \mathfrak q\leq n^2$.
The symbol $\mathfrak q(\mathbf w)$ in this section is a real function on the torus,
not a Novikov variable $Q$.

\begin{theorem}[Positive torus formula]\label{thm:torus}
The power series in~\eqref{eq:mirrorcoordinate} extend to
$\C\setminus(-\infty,-n^{-2}]$ by
\[
 A(y)=\int\frac{d\nu}{1+y\mathfrak q(\mathbf w)},\qquad
 g(y)=-\int\log(1+y\mathfrak q(\mathbf w))\,d\nu,
\]
using the principal logarithm. Their Taylor radii at zero are exactly
$n^{-2}$. For $y>0$, $A(y)>0$ and $Q(y)=ye^{g(y)}$ is strictly increasing.
It maps $(0,\infty)$ bijectively onto $(0,Q_\infty)$, where
\begin{equation}\label{eq:Qinfty}
 Q_\infty=\exp\left(-\int\log|w_1+\cdots+w_n|^2d\nu\right),
 \qquad n^{-2}\leq Q_\infty\leq1.
\end{equation}
The inverse $y(Q)$ is real analytic on that open interval.
\end{theorem}
\begin{proof}
The constant term of
$(w_1+\cdots+w_n)^D(w_1^{-1}+\cdots+w_n^{-1})^D$
is $c_D$, so $c_D=\int \mathfrak q^D d\nu$. Summing the geometric and
logarithmic series for $|y|<n^{-2}$ gives the formulas.
On compact subsets of $\C\setminus(-\infty,-n^{-2}]$ their denominators and logarithm
branches are uniformly controlled for $0\leq \mathfrak q\leq n^2$; hence they
are holomorphic there.
If $M_D=\binom{D+n-1}{n-1}$ is the number of compositions, then
Cauchy--Schwarz and the multinomial theorem give
\[
 \frac{n^{2D}}{M_D}\leq c_D\leq n^{2D}.
\]
Thus $\lim c_D^{1/D}=n^2$, proving the exact radii for $A$ and $g$.
For positive $y$, $yQ'(y)/Q(y)=1+yg'(y)=A(y)>0$.
Also
\[
 \log Q(y)=-\int\log(\mathfrak q(\mathbf w)+y^{-1})d\nu.
\]
The limiting logarithmic integral is finite. Indeed, conditioning on
$w_1,\ldots,w_{n-1}$ and writing $b=\sum_{i<n}w_i$, the elementary
one-variable Jensen formula gives
$\int_{S^1}\log|b+w_n|^2d\nu_n=2\log\max(1,|b|)$.
It is between $0$ and $2\log(n-1)$. For completeness, the circle identity follows by expanding
$\log(1+w/b)$ when $|b|>1$, and $\log(1+b/w)$ when $|b|<1$;
all nonconstant Fourier coefficients have mean zero. At $|b|=1$ the
remaining singularity is $2\log|2\sin(t/2)|$, which is integrable,
and gives the same limit. The positive part is bounded uniformly by
$2\log n$. Since the circle mean is nonnegative, the mean of the
negative part is also at most $2\log n$, uniformly for $|b|\leq n-1$.
Tonelli's theorem, applied to the two nonnegative parts, now proves absolute
integrability on the full torus. Fubini's theorem then computes
$\int\log\mathfrak q\,d\nu$ by first taking the circle mean in $w_n$ and
then integrating over $w_1,\ldots,w_{n-1}$; since that circle mean is
nonnegative, so is the integral. Monotone convergence applied after
subtracting an upper bound to $\log(\mathfrak q+y^{-1})$ proves the limit.
Since also $\mathfrak q\leq n^2$, the integral of $\log\mathfrak q$ lies
between $0$ and $2\log n$. This proves~\eqref{eq:Qinfty} and the
endpoint behavior. Strict monotonicity and the real inverse-function
theorem complete the proof.
\end{proof}

\begin{corollary}[All-negative-axis continuation]\label{cor:realcontinuation}
The germs $\Dcal_a^\Delta$ admit a specified real analytic continuation
along $0<Q<Q_\infty$. Using these continued functions,~\eqref{eq:allloop} holds for every real $s<0$.
This assertion does not say that the original Taylor series in $Q$
converges on the entire interval.
\end{corollary}
\begin{proof}
The leading coefficient~\eqref{eq:leadP} never vanishes for $y>0$.
Thus all the $W_j$ used in Theorem~\ref{thm:serre} continue along the
positive real axis as solutions of their homogeneous or inhomogeneous
ordinary differential equations. Write
$I_{X_n}^{0,\amb}(y,1)=\sum_{j=0}^nH^jJ_j(y)$.
The identity $\theta_yI_{X_n}^{0,\amb}=H\Ish$ says
$\theta_yJ_{j+1}=W_j$ for $0\leq j\leq n-1$.
It therefore continues the $J_j$ by integration, and hence the $K_j$
by the exponential normalization in~\eqref{eq:Kdef}.
Since $A>0$, $Q'(y)>0$; Theorem~\ref{thm:torus} and~\eqref{eq:extraction} give the asserted continuation of the descendants.
The physical integral is analytic on $s<0$ by
Proposition~\ref{prop:analyticdomain}. Both sides of~\eqref{eq:allloop}
continue from the same germ, so the identity persists along this axis.
No passage through a singularity of the differential operator is involved.
\end{proof}
\subsection{Global univalence and the $Q$-Taylor disk}\label{sec:univalence}
Corollary~\ref{cor:realcontinuation} leaves open where the Taylor series in $Q$
converge. We now show that the mirror map is injective on all of
$\C\setminus(-\infty,-n^{-2}]$ and bound the $Q$-Taylor radius above and below
by explicit constants. Write $\mathbb E$ for integration against $\nu$.
Integrating the cross terms gives $\mathbb E\mathfrak q=n$ and $\mathbb
E\mathfrak q^2=2n^2-n$. By Theorem~\ref{thm:torus},
\begin{equation}\label{eq:Qy}
 Q(y)=y\exp\bigl(-\mathbb E\Log(1+y\mathfrak q)\bigr)\qquad(y\in\C\setminus(-\infty,-n^{-2}]),
\end{equation}
with $Q(0)=0$ and $Q'(0)=1$. It is convenient to use the coordinate $x=1/y$
and to argue for a general random variable. Let $\mathfrak p$ be a random
variable with values in $[0,\varkappa]$, $\varkappa>0$, and put
\[
 R_{\mathfrak p}(x)=-\mathbb E\Log(x+\mathfrak p),\qquad
 \mathcal Q_{\mathfrak p}(x)=e^{R_{\mathfrak p}(x)}\qquad
 \bigl(x\in\C\setminus(-\infty,0]\bigr),
\]
and
\[
 r_{\mathfrak p}(u)=\exp\bigl(-\mathbb E\log|\mathfrak p-u|\bigr)\in(0,\infty]
 \qquad(0\leq u\leq\varkappa),
\]
which is positive because $\mathbb E\log|\mathfrak p-u|\leq\log\varkappa$.
For $\mathfrak p=\mathfrak q$ we have $\varkappa=n^2$, and
$\mathcal Q_{\mathfrak q}(1/y)=Q(y)$ on $\C\setminus(-\infty,-n^{-2}]$ once
$\mathcal Q_{\mathfrak q}$ is extended as in Theorem~\ref{thm:univalence}:
both sides are holomorphic on $\C\setminus(-\infty,-n^{-2}]$ and agree for $y>0$. We then drop
the subscript and write $r(u)=r_{\mathfrak q}(u)$.

\begin{theorem}[Global univalence]\label{thm:univalence}
The map $\mathcal Q_{\mathfrak p}$ extends holomorphically to
$\widehat\C\setminus[-\varkappa,0]$, with $\mathcal Q_{\mathfrak p}(\infty)=0$,
and is injective there. Its image contains the disk of radius
\[
 r_{\mathfrak p}=\inf_{0\leq u\leq\varkappa}r_{\mathfrak p}(u)
 \;\geq\;\frac1{\max\{\mathbb E\mathfrak p,\varkappa-\mathbb E\mathfrak p\}}.
\]
In particular $Q$ is injective on $\C\setminus(-\infty,-n^{-2}]$, and its image
$\Omega_n=Q(\C\setminus(-\infty,-n^{-2}])$ contains the disk $|Q|<r_n$, where
\begin{equation}\label{eq:rn}
 r_n=\inf_{0\leq u\leq n^2}\exp\bigl(-\mathbb E\log|\mathfrak q-u|\bigr)
 \geq\frac1{n(n-1)}.
\end{equation}
\end{theorem}
\begin{proof}
Write $R=R_{\mathfrak p}$ and $\mathcal Q=\mathcal Q_{\mathfrak p}$. On the
upper half-plane, $-\pi<\operatorname{Im}R(x)<0$ and
\[
 \operatorname{Im}R'(x)=\operatorname{Im}x\;\mathbb E|x+\mathfrak p|^{-2}>0 .
\]
For distinct $x_1,x_2$ there,
$(R(x_2)-R(x_1))/(x_2-x_1)=\int_0^1R'(x_1+t(x_2-x_1))\,dt$ has positive
imaginary part, so $R$ is injective; this is the Noshiro--Warschawski
criterion~\cite[Theorem~2.16]{Duren} applied to $-\mathrm iR$. The
exponential is injective on the strip $-\pi<\operatorname{Im}R<0$, so
$\mathcal Q$ is injective on the upper half-plane, with values in the open
lower half-plane. Since $\mathcal Q(\bar x)=\overline{\mathcal Q(x)}$, the
same holds on the lower half-plane, with values in the open upper
half-plane. For $x>0$, $\mathcal Q(x)>0$ and
$\mathcal Q'(x)=-\mathcal Q(x)\,\mathbb E(x+\mathfrak p)^{-1}<0$. For
$x<-\varkappa$ one has $x+\mathfrak p<0$, so the boundary values of $R$
from the two half-planes differ by exactly $2\pi\mathrm i$; hence
$\mathcal Q$ extends holomorphically across $(-\infty,-\varkappa)$, where
\[
 \mathcal Q(x)=-\exp(-\mathbb E\log|x+\mathfrak p|)<0,\qquad
 \mathcal Q'(x)=-\mathcal Q(x)\,\mathbb E(x+\mathfrak p)^{-1}<0 .
\]
Finally $\mathcal Q(x)=x^{-1}+O(x^{-2})$ as $x\to\infty$, so $\mathcal Q$
extends holomorphically to $\infty$, with $\mathcal Q(\infty)=0$ and nonzero
derivative in $1/x$. The images of the two half-planes, of $(0,\infty)$, of
$(-\infty,-\varkappa)$ and of $\infty$ are pairwise disjoint, and
$\mathcal Q$ is injective on each of these pieces; hence it is injective on
$\widehat\C\setminus[-\varkappa,0]$.

For the image, let $x_k\to-u$ with $u\in[0,\varkappa]$ and
$x_k\notin[-\varkappa,0]$. The functions $\log|x_k+\mathfrak p|$ are
bounded above by $\log(\varkappa+\sup_k|x_k|)$ and converge pointwise to
$\log|\mathfrak p-u|\in[-\infty,\infty)$. The reverse Fatou lemma gives
$\limsup_k\mathbb E\log|x_k+\mathfrak p|\leq\mathbb E\log|\mathfrak p-u|$,
that is,
\begin{equation}\label{eq:reverseFatou}
 \liminf_{k\to\infty}|\mathcal Q(x_k)|\geq r_{\mathfrak p}(u)\geq r_{\mathfrak p}.
\end{equation}
Let $w$ be a boundary point of the image with $|w|<r_{\mathfrak p}$, and
choose $x_k$ with $\mathcal Q(x_k)\to w$. A subsequence of $(x_k)$ converges
in $\widehat\C$. Its limit cannot lie in $\widehat\C\setminus[-\varkappa,0]$,
since then $w$ would lie in the image, which is open; so the limit is a
point $-u$ of the slit, contradicting~\eqref{eq:reverseFatou}. Thus the
image, which is open, connected and contains $0$, has no boundary point in
the disk $|w|<r_{\mathfrak p}$, and therefore contains that disk (a segment
from $0$ to a missing point of the disk would meet the boundary). Finally,
Jensen's inequality and the convexity of $u\mapsto\mathbb E|\mathfrak p-u|$
give
$\mathbb E\log|\mathfrak p-u|\leq\log\mathbb E|\mathfrak p-u|
\leq\log\max\{\mathbb E\mathfrak p,\varkappa-\mathbb E\mathfrak p\}$.
For $\mathfrak p=\mathfrak q$ one has $\varkappa=n^2$ and
$\mathbb E\mathfrak q=n$.
\end{proof}

The constant $r_n$ is attained at the right endpoint if and only if the
following inequality holds; we use it only as a hypothesis:
\begin{equation}\label{eq:endpointmax}
 \mathbb E\log|\mathfrak q-u|\leq\mathbb E\log(n^2-\mathfrak q)
 \qquad(0\leq u\leq n^2).
\end{equation}

\begin{theorem}[The $Q$-Taylor radius of the mirror map]\label{thm:Qradius}
Let $n\geq4$. Then $\mathbb E(n^2-\mathfrak q)^{-1}<\infty$, and the constant
\begin{equation}\label{eq:rhon}
 \rho_n=\exp\bigl(-\mathbb E\log(n^2-\mathfrak q)\bigr)=r(n^2)
\end{equation}
satisfies $Q(y)\to-\rho_n$ as $y\downarrow-n^{-2}$; the map $Q$ sends
$(-n^{-2},0)$ increasingly onto $(-\rho_n,0)$. The Taylor radius $R_y$ of
the inverse mirror map $y(Q)$ at $Q=0$ satisfies
\begin{equation}\label{eq:Qradius}
 \frac1{n(n-1)}\leq r_n\leq R_y\leq\rho_n\leq\frac2{n(n-1)}<\frac1n\leq Q_\infty .
\end{equation}
If~\eqref{eq:endpointmax} holds, then $r_n=R_y=\rho_n$.
\end{theorem}
\begin{proof}
Lagrange's identity gives $n^2-\mathfrak q=\sum_{j<k}|w_j-w_k|^2$. Keeping
only the pairs with $j=1$ and writing $w_k=w_1e^{\mathrm i\theta_k}$ with
$\theta_k\in(-\pi,\pi]$,
\[
 n^2-\mathfrak q\geq\sum_{k=2}^n4\sin^2(\theta_k/2)
 \geq\frac4{\pi^2}\sum_{k=2}^n\theta_k^2 .
\]
The $\theta_k$ are independent and uniform, and $|\theta|^{-2}$ is
integrable on $(-\pi,\pi]^{n-1}$ because $n-1\geq3$. Hence
$\mathbb E(n^2-\mathfrak q)^{-1}<\infty$; in particular
$\mathbb E\log(n^2-\mathfrak q)$ is finite.

For $-n^{-2}<y<0$ one has $0<1+y\mathfrak q\leq1$, so $Q(y)<0$ and, as in
the proof of Theorem~\ref{thm:torus}, $yQ'(y)/Q(y)=A(y)>0$; thus $Q'>0$. As
$y\downarrow-n^{-2}$, $1+y\mathfrak q$ decreases to $1-\mathfrak q/n^2$, and
monotone convergence gives $Q(y)\to-\rho_n$ and
$A(y)\uparrow n^2\,\mathbb E(n^2-\mathfrak q)^{-1}$. Hence
\begin{equation}\label{eq:Qprimec}
 Q'(y)\longrightarrow Q'_c:=n^4\rho_n\,\mathbb E(n^2-\mathfrak q)^{-1}\in(0,\infty)
 \qquad(y\downarrow-n^{-2}).
\end{equation}

By Theorem~\ref{thm:univalence} the inverse $y(Q)$ is holomorphic on
$\Omega_n\supset\{|Q|<r_n\}$, so $R_y\geq r_n\geq1/(n(n-1))$; and
$r_n\leq r(n^2)=\rho_n$. The arithmetic--geometric mean inequality applied to
Lagrange's identity gives
\[
 \mathbb E\log(n^2-\mathfrak q)\geq\log\binom n2
 +\binom n2^{-1}\sum_{j<k}\mathbb E\log|w_j-w_k|^2=\log\binom n2,
\]
because $w_j\bar w_k$ is uniformly distributed on the circle and
$\int_{S^1}\log|1-w|\,d\nu_1=0$. This is $\rho_n\leq2/(n(n-1))$, which is
less than $1/n$ since $n\geq4$. Jensen's inequality gives
$\mathbb E\log\mathfrak q\leq\log n$, that is, $Q_\infty\geq1/n$.

Suppose now that $R_y>\rho_n$, and let
\[
 \mathcal Y(Q):=\sum_{m\geq0}c_mQ^m
\]
be the analytic sum of the Taylor series of $y(Q)$ on $|Q|<R_y$.
It agrees with $y(Q)$ on the connected component of
$\{|Q|<R_y\}\cap\Omega_n$ that contains $0$, and this component contains
$(-\rho_n,0]$. By~\eqref{eq:Qprimec},
\[
 \mathcal Y(-\rho_n)=-n^{-2},
 \qquad
 \mathcal Y'(-\rho_n)=\frac{1}{Q'_c}\neq0.
\]
Hence $\mathcal Y$ maps a neighbourhood $N$ of $-\rho_n$, with $0\notin N$,
biholomorphically onto a disk $W$ centred at $-n^{-2}$. Let
$G\colon W\to N$ be the inverse. For
$y\in W\cap(-n^{-2},0)$ sufficiently close to $-n^{-2}$, one has
$Q(y)\in N$ and
\[
 \mathcal Y(Q(y))=y,
\]
so $G(y)=Q(y)$. Since
$W\setminus(-\infty,-n^{-2}]$ is connected, the identity theorem gives
$G=Q$ there. Thus $Q$ extends holomorphically to $W$ and has no zeros
there, contradicting the singularity at $y=-n^{-2}$.
Pick $y_0\in W$ with
$y_0<-n^{-2}$, put $u=-1/y_0\in(0,n^2)$, and let
$x_-=1/(y_0+\mathrm i\varepsilon)$ and $x_+=\overline{x_-}$. As
$\varepsilon\downarrow0$, $x_\mp\to-u$ from the lower and upper half-plane.
The function $\mathfrak q$ is a nonconstant real-analytic function on the
connected torus, so it has no atoms, and dominated convergence
($|\operatorname{Arg}|\leq\pi$) gives
$\operatorname{Im}R(x_-)\to\pi\,\nu(\mathfrak q<u)$. Since
$R(\bar x)=\overline{R(x)}$,
\[
 \frac{Q(y_0+\mathrm i\varepsilon)}{Q(y_0-\mathrm i\varepsilon)}
 =\exp\bigl(2\mathrm i\operatorname{Im}R(x_-)\bigr)
 \longrightarrow e^{2\pi\mathrm i\,\nu(\mathfrak q<u)},
\]
whereas holomorphy of $Q$ at $y_0$, with $Q(y_0)\neq0$, forces the limit to
be $1$. Since $0<\nu(\mathfrak q<u)<1$, and $\mathfrak q$ is continuous and
takes the values $0$ and $n^2$ on the connected torus, both
$\{\mathfrak q<u\}$ and $\{\mathfrak q>u\}$ are nonempty open sets.
This contradiction proves
$R_y\leq\rho_n$. Under~\eqref{eq:endpointmax}, $r(u)\geq r(n^2)=\rho_n$ for
all $u\in[0,n^2]$, so $r_n=\rho_n$, and~\eqref{eq:Qradius} forces
$R_y=\rho_n$.
\end{proof}

\begin{corollary}\label{cor:Qradius}
Let $n\geq4$. For $2\leq k\leq n$ the normalized ambient germs
$K_k(y(Q))$, with $K_k$ as in~\eqref{eq:Kdef}, and for $0\leq a\leq n-2$ the
diagonal descendant germs $\Dcal_a^\Delta(Q)$ have Taylor radius at least
$r_n$, hence at least $1/(n(n-1))$, and at least $\rho_n$
if~\eqref{eq:endpointmax} holds. In particular, if
$\Dcal_a^\Delta(Q)=\sum_md_{a,m}Q^m$, then
$\limsup_m|d_{a,m}|^{1/m}\leq1/r_n\leq n(n-1)$.
\end{corollary}
\begin{proof}
By~\eqref{eq:leadP} and Theorem~\ref{thm:PF}, the finite nonzero singular
points of the ambient equation are among $y=-1/(n-2j)^2$, $0\leq j<n/2$,
all on the removed ray $(-\infty,-n^{-2}]$. The functions $K_k(y)$ are
convergent holomorphic germs at $y=0$, so their continuations to
$\C^*\setminus(-\infty,-n^{-2}]$ have trivial monodromy around $0$, which
generates $\pi_1(\C^*\setminus(-\infty,-n^{-2}])$; hence they are single-valued
and holomorphic on $\C\setminus(-\infty,-n^{-2}]$. (Alternatively, the logarithmic monodromy
cancels directly in the normalized ambient series.) Composing with the
holomorphic inverse $y(Q)\colon\Omega_n\to\C\setminus(-\infty,-n^{-2}]$ of
Theorem~\ref{thm:univalence} shows that $K_k(y(Q))$ is holomorphic on
$|Q|<r_n$, and Theorem~\ref{thm:extraction} transfers this to
$\Dcal_a^\Delta(Q)$. The last assertion is the Cauchy--Hadamard formula.
\end{proof}

\begin{remark}\label{rem:Qradius}
Condition~\eqref{eq:endpointmax} says that the logarithmic potential of the
law of $\mathfrak q$ on $[0,n^2]$ is maximal at the right endpoint; it is
not used anywhere else in the paper. 
We have checked it numerically for $4\leq n\leq9$ by conditioning on
$w_1,\ldots,w_{n-1}$ and averaging over the last angle using the circle
formula from the proof of Theorem~\ref{thm:torus}. In each case, the maximum
over $[0,n^2]$ is attained only at $u=n^2$, and the potential is increasing
on $[n^2-1,n^2]$.
For $n=4,5$, ratio tests on the $Q$-expansions up to order
about $120$ place the dominant singularity of $y(Q)$ and of the
$\Dcal_a^\Delta(Q)$ at $Q=-\rho_n$, the image of the conifold point
$y=-n^{-2}$; accordingly the coefficients eventually alternate in sign. The
table lists the constants in~\eqref{eq:Qradius} together with the Euclidean
momenta $T_*$ defined by $Q(1/T_*)=1/(n(n-1))$ and by $Q(1/T_*)=\rho_n$.
\begin{center}
\begin{tabular}{rllll}
\toprule
$n$ & $1/(n(n-1))$ & $\rho_n$ & $T_*$ for $1/(n(n-1))$ & $T_*$ for $\rho_n$\\\midrule
$4$ & $0.083333$ & $0.089553$ & $8.42$  & $7.61$\\
$5$ & $0.050000$ & $0.051919$ & $15.43$ & $14.70$\\
$6$ & $0.033333$ & $0.034107$ & $24.43$ & $23.76$\\
$7$ & $0.023810$ & $0.024180$ & $35.44$ & $34.80$\\
\bottomrule
\end{tabular}
\end{center}
Since $Q(1/T)$ decreases in $T$, Corollary~\ref{cor:Qradius} shows that the
$Q$-series of the descendants converge at $s=-T$ for every $T$ larger than
the value in the fourth column, whereas Theorem~\ref{thm:laplace} is stated
for $T>n^2$; by Theorem~\ref{thm:Qradius} the series of $y(Q)$ diverges for
$T$ smaller than the value in the last column. Together with
$\rho_n<Q_\infty$ this makes the caveat in
Corollary~\ref{cor:realcontinuation} precise. Finally,
$\operatorname{Var}\mathfrak q=n(n-1)$, and expanding
$\mathbb E\log(n^2-\mathfrak q)$ about the mean suggests
$n(n-1)\rho_n=1+\frac1{2n(n-1)}+O(n^{-3})$. The observed values
$n(n-1)\rho_n=1.0746$, $1.0384$, $1.0232$, $1.0155$ for $n=4,\ldots,7$ are
consistent with this, so the lower bound in~\eqref{eq:Qradius} appears to be
asymptotically sharp.
\end{remark}

\section{Degree-one geometry}\label{sec:degreeone}
\begin{theorem}\label{thm:degreeone}
For each coordinate curve class $e_i$ and all $n\geq4$,
\[
 N_{0,e_i}=(n-1)(n-2),\qquad N_{1,e_i}=-2(n-1)(n-2),\qquad
 N_{a,e_i}=0\quad(2\leq a\leq n-2).
\]
Therefore the coefficient of $Q$ in the diagonal descendant tuple is
\[
 n(n-1)(n-2)\,(1,-2,0,\ldots,0).
\]
\end{theorem}
\begin{proof}
Fix $i$, put $Y_i=(\PP^1)^{n-1}$, and let
$h=\sum_{j\ne i}p_j$. The equation of $F_n$ is
$x_0a_i+x_1b_i=0$ over $Y_i$. A coordinate fiber is contained in $F_n$
exactly on $S_i=\{a_i=b_i=0\}\subset Y_i$.
At a point of $S_i$, if $da_i,db_i$ were dependent, a nonzero combination
$x_0da_i+x_1db_i$ would vanish. At that point of $\PP^1\times Y_i$ both
the fiber and base derivatives of the equation would vanish,
contradicting smoothness of $F_n$. Thus $S_i$ is a smooth regular
codimension-two locus of class $h^2$.

The ambient one-marked degree-one moduli is $\PP^1\times Y_i$;
the hypersurface moduli is its zero locus $\PP^1\times S_i$ for the induced section.
The degree-one nonconstant component maps isomorphically to its fiber.
A contracted tree attached to it cannot be stable with only one marking:
for a nonempty subtree with $v\geq1$ vertices and at most one marking, the total
number of special points is at most $2(v-1)+1+1=2v<3v$.
There are consequently no such boundary trees and no residual
map automorphisms. This description also holds for families, since
a fiberwise degree-one map to $\PP^1$ identifies the universal domain
with the trivial $\PP^1$-bundle, with its marking specifying a point.

The induced rank-two positive-bundle section on the smooth ambient
moduli has regular zero locus $\PP^1\times S_i$; its virtual class
is therefore fundamental. The local obstruction bundle has rank
$m-1=0$, so the local virtual class is fundamental as well.
On this moduli space let $p$ be the class of the marked point in the
fiber. The universal cotangent line has $\psi=-2p$, and
$\ev_{F_n}^*H=p+h$, with $p^2=0$.
Since $\int_{S_i}h^{n-3}=\int_{Y_i}h^{n-1}=(n-1)!$, the primary invariant is
$(n-1)!/(n-3)!=(n-1)(n-2)$, the first descendant is
$-2(n-1)!/(n-3)!$, and the remaining ones vanish by $p^2=0$.
There are $n$ coordinate classes, giving the final assertion.
\end{proof}

\section{Degree two and degree-dependent descendant bounds}\label{sec:degreetwo}
The primary calculations in this section use only the ambient virtual
comparison, and do not depend on the ambientness theorem or Iritani's
results in~\cite{Iritani}.
 Use $\delta_a$ for the insertions of
Proposition~\ref{prop:geometry}; the scalar Gamma coefficients remain
$\gamma_k$, and the assembly polynomials are $\mathcal T_j$.

\begin{theorem}[Degree-two primaries]\label{thm:degreetwo}
For all $n\geq4$ and distinct $i,j$,
\[
 N_{0,2e_i}=\frac{(n-1)(n-2)}4,\qquad
 N_{0,e_i+e_j}=-2(n-1)(n-2).
\]
Consequently, writing $N_m=\sum_{|\beta|=m}N_{0,\beta}$,
\[
 N_2=-\frac{n(n-1)(n-2)(4n-5)}4.
\]
The multiple-cover transform defined by
$N_{0,\beta}=\sum_{k\alpha=\beta}b_\alpha/k^2$ gives
\[
 b_{2e_i}=0,\qquad b_{e_i+e_j}=-2(n-1)(n-2),\qquad
 n_2:=\sum_{|\beta|=2}b_\beta=-n(n-2)(n-1)^2.
\]
These explicit integrality statements require no general integrality theorem.
\end{theorem}
\begin{proof}
On the ambient moduli use the bundles $E_+=\pi_*f^*\mathcal O_{Y_n}(H)$ and
$E_-=R^1\pi_*f^*\mathcal O_{Y_n}(-H)$ of Lemma~\ref{lem:basechange}.
The virtual comparison computes the invariant by integrating
$e(E_+)e(E_-)\ev^*\delta_0$. We compute the two possible degree types.

For $2e_i$, separate the passive factor $(\PP^1)^{n-1}$ and write
$\bar p=\sum_{j\ne i}p_j$. On the active unmarked moduli
$M=\overline{\mathcal M}_{0,0}(\PP^1,2)$ put
\[
 E^{\rm act}_+=\pi_*f^*\mathcal O_{\PP^1}(1),\qquad
 E^{\rm act}_-=R^1\pi_*f^*\mathcal O_{\PP^1}(-1),\qquad c=c_1(E^{\rm act}_-).
\]
These have ranks three and one. The line bundles on the target are
$\mathcal O(1)$ and $\mathcal O(-1)$, whose pullbacks have degrees
$2$ and $-2$. Let $B$ be the trace-free line of the double cover
recalled below (unrelated to the bundles $\mathcal B_\pm$ on $Y_n$), so that $f_*\mathcal O_C=\mathcal O\oplus
(\mathcal O(-1)\otimes B)$. The projection formula gives
\[
 E^{\rm act}_+=H^0(\PP^1,\mathcal O(1))\otimes\mathcal O\oplus B,
 \qquad
 E^{\rm act}_-=H^1(\PP^1,\mathcal O(-2))\otimes B\cong B,
\]
so $E^{\rm act}_+\cong\mathcal O^2\oplus E^{\rm act}_-$ and
$c_1(E^{\rm act}_+)=c$.
After passive twisting,
\[
 e(E_+)e(E_-)=\bar p^2(\bar p+c)(c-\bar p)=c^2\bar p^2-\bar p^4.
\]
Here and below these bundles pull back under forgetting the marking;
contraction of an unstable degree-zero component preserves their cohomology.

We give the stack integral used in this computation. An unmarked degree-two
map to $\PP^1$ has no contracted components, since a contracted subtree
attached to the at most two positive components cannot have three special
points per vertex.
It is a finite flat double cover, including the reducible
boundary consisting of two degree-one components. Conversely every
nonzero binary quadratic defines such a cover; a double root gives the
nodal reducible cover. Trace splitting of the rank-two algebra gives
$\mathcal O\oplus\mathcal O(-1)\otimes B$ over a base, with multiplication
specified by a section of $\mathcal O(2)\otimes B^{-2}$.
This description is compatible with base change (the finite morphism is
flat of degree two, and the trace-free line has relative degree $-1$).
It identifies
\[
 M\cong[(H^0(\PP^1,\mathcal O(2))\setminus\{0\})/\mathbb G_m]
       =\PP(2,2,2).
\]
The branch map to $\PP^2$ has generic stabilizer $\mu_2$.
If $h$ is the pullback of its hyperplane class, the character line
$E^{\rm act}_-$ has $c=\pm h/2$, as also follows by tensoring the trace splitting
with $\mathcal O(-1)$ and taking $H^1$.
Thus
$\int_M c^2=\frac12\frac14=\frac18$.
The divisor equation (Appendix~\ref{app:forget}) for the marking gives
$\int_{\overline{\mathcal M}_{0,1}(\PP^1,2)}\ev^*p_i\,c^2=2/8$.
Only the term $p_i\bar p^{\,n-3}/(n-3)!$ of $\delta_0$ contributes,
since the other terms have excessive passive degree.
Because $\int_{(\PP^1)^{n-1}}\bar p^{\,n-1}=(n-1)!$,
the answer is $(n-1)(n-2)/4$.

For $e_i+e_j$, the active unmarked moduli is $\PP^3$, the complete
linear system of $(1,1)$ divisors in $Y_A=\PP^1\times\PP^1$.
Every geometric map has at most two positive components. A contracted
unmarked subtree would need at least three external attachments by
stability, so none occurs. The map is the inclusion of a smooth graph
or the transverse nodal union of one fiber of each ruling.
Every automorphism of such an inclusion is the identity.

For a family over $S$, the map $C\to Y_A\times S$ is proper and
fiberwise quasi-finite, hence finite. The cokernel of
$\mathcal O_{Y_A\times S}\to f_*\mathcal O_C$ vanishes on fibers;
Nakayama, applied to this finite module locally, makes it zero.
The map is therefore a closed immersion. Its flat finitely presented
image is fiberwise an effective Cartier divisor, and the fiberwise
Cartier criterion makes it a relative effective Cartier divisor.
Its equation determines a line in the fixed four-dimensional space
$H^0(Y_A,\mathcal O(1,1))$, giving the projective-space functor.
These constructions commute with base change; the finite-presentation
data can first be descended to a Noetherian base if required.
This proves the stack isomorphism, including all finite stabilizers,
which here are trivial.

The one-marked moduli is the universal divisor
$U\subset\PP^3\times Y_A$ by the stable-map universal-curve theorem
\cite[Proposition 4.5 and Corollary 4.6]{BM}.
A point at the node represents the map obtained by inserting a contracted
$\PP^1$ between the two branches and placing the marking on it.
The two nodes and the marking make that component stable. Forgetting
the marking contracts it. A section through a node before this insertion
is not an admissible marked prestable curve. Pullback target bundles
retain their cohomology under the contraction, because the contracted
component has $H^0(\mathcal O)=\C$ and $H^1(\mathcal O)=0$ and values
match at its two nodes.
Let $h=c_1(\mathcal O_{\PP^3}(1))$ and
$\bar p=\sum_{k\ne i,j}p_k$, a sum of $n-2$ passive classes.
The universal divisor sequences give
\[
 c(E^{\rm act}_+)=(1-h)^{-1},\qquad E^{\rm act}_-=\mathcal O_{\PP^3}(-1).
\]
Indeed, for the positive bundle the sequence has first term
$\mathcal O_{\PP^3}(-1)\boxtimes\mathcal O$ and middle term
$\mathcal O\boxtimes\mathcal O(1,1)$; for the negative bundle the first
term is $\mathcal O(-h)\boxtimes\mathcal O(-2,-2)$ and the middle term
is $\mathcal O\boxtimes\mathcal O(-1,-1)$.
Their direct images give precisely the two assertions. Hence
\[
 e(E_+)e(E_-)=-(\bar p^4+2h\bar p^3+2h^2\bar p^2+2h^3\bar p),\qquad h^4=0.
\]
Integrate this against $(\bar p+p_i+p_j)^{n-2}/(n-2)!$ on $U$ times
the passive factors, using $[U]=h+p_i+p_j$.
Only two terms survive. The term $-2h^3\bar p$ combines with
$p_i+p_j$ from $[U]$ and the linear active term of the insertion,
giving $-4(n-2)$. The term $-2h^2\bar p^2$ combines with $h$ from $[U]$
and the quadratic active term, giving $-2(n-2)(n-3)$.
Their sum is $-2(n-1)(n-2)$, as asserted.
There are $n$ classes of the first type and $\binom{n}{2}$ of the
second. Sum them, and subtract $N_{0,e_i}/4$ in class $2e_i$ to obtain
the multiple-cover claims.
\end{proof}

\begin{theorem}[Support bound for descendants]\label{thm:supportbound}
Let $\beta=(d_1,\ldots,d_n)$ have total degree $m>0$ and let
$k=\#\{i:d_i>0\}$. For the insertions defined here,
\[
 N_{a,\beta}=0\quad\hbox{if }a>2m+k-2.
\]
In particular $N_{a,\beta}=0$ if $a>3m-2$. At degree two this gives
$N_{a,\beta}=0$ for $a\geq5$; these assertions are understood within
$0\leq a\leq n-2$.
\end{theorem}
\begin{proof}
Every passive coordinate is constant on the connected genus-zero domain. This is
functorial over arbitrary bases since $\pi_*\mathcal O_C= \mathcal O_B$. At any
geometric base point, choose an affine chart containing the constant fiber
image. Properness of the universal curve makes the image of the complement
closed in the base. Removing it puts the whole family over that neighborhood in
the chart, where $\pi_*\mathcal O_C=\mathcal O_B$ gives the factorization. The
local factorizations glue uniquely. Here $R^1\pi_*\mathcal O_C=0$ and the
isomorphism $\mathcal O_B\to\pi_*\mathcal O_C$ follow by the same split-complex
argument as Lemma~\ref{lem:basechange}, since each connected genus-zero fiber
has $h^0=1$, $h^1=0$, and the constant section generates $H^0$. Thus the ambient
one-marked moduli is
\[
 \overline{\mathcal M}_{0,1}((\PP^1)^k,\beta_{\rm act})
             \times(\PP^1)^{n-k}.
\]
Stability and the cotangent line depend only on the active map.
The active moduli is a smooth proper Deligne--Mumford stack of dimension
$k-3+2m+1=2m+k-2$, since the target is convex and the genus-zero obstruction
$H^1(C,f^*T)$ vanishes. 
Its rational Chow groups vanish above this
dimension, so $\psi^a=0$ there for the stated range. Pull back and use
the ambient virtual comparison; multiplying by its Euler classes and
the insertion cannot change this zero. Finally $k\leq m$.
\end{proof}

\begin{proposition}[Sharp diagonal bound]\label{prop:sharpbound}
For $m\geq1$ and $n\geq3m$,
\[
 [Q^m]\Dcal_{3m-2}^\Delta(Q)
 =-\frac{2^mn!}{m!\,(n-3m)!}\ne0.
\]
Thus $3m-2$ cannot be replaced by $2m$ as a uniform descendant cutoff.
For example, $n=9$, $m=3$ gives
$[Q^3]\Dcal_7^\Delta=-483840$, although $7>2m$.
\end{proposition}
\begin{proof}
Work with the local ambient $I$-series at $z=1$ and remove its bare
exponential $e^{H\log y}$. For a degree tuple of total degree $m$
and support $k$, the local numerator has degree at most $2m$ in $H$.
Because $p_i^2=0$, expanding each active denominator contributes at most
one power of $p_i$, hence at most $k$ additional cohomological powers.
The total degree is at most $2m+k\leq3m$.
A coefficient of $y^r$ in $e^{-Hg(y)}$ has degree at most $r$ in $H$.
For $r>0$, its product with a term of degree $m-r$ has degree at most
$3(m-r)+r<3m$. Thus normalization cannot change the top coefficient.
Neither can substituting $y(Q)=Q+O(Q^2)$, since lower degrees have
cohomological degree strictly less than $3m$.

The top degree is possible only for tuples with $m$ distinct entries
one. For each such tuple the leading numerator is $(-1)^mH^{2m}$,
and the top denominator term is $(-2)^m\prod_{i:d_i=1}p_i$.
Summing over these subsets gives
$2^mH^{2m}e_m(p)=2^mH^{3m}/m!$.
Apply Theorem~\ref{thm:extraction} with $a=3m-2$.
\end{proof}

\section{Compactification and primary arithmetic}\label{sec:arithmetic}
\subsection{The compact target and its curve lattice}
Use the lines convention
\[
 Z=\PP_{\rm lines}(\mathcal O_{F_n}\oplus K_{F_n}),\quad p:Z\to F_n,\quad
 \xi=c_1(\mathcal O_Z(1)).
\]
Let $D_0$ be the section from the first summand and $D_\infty$ that
from $K_{F_n}$; write $i_0:F_n=D_0\hookrightarrow Z$ and let $\ell_f$ be
the fiber curve class.

\begin{theorem}[Compactification and class comparison]\label{thm:compact}
One has
\[
 Z\setminus D_\infty=X_n,\quad [D_\infty]=\xi,\quad i_0^*\xi=0,\quad
 \mathcal O_Z(D_\infty)|_{D_\infty}=K_{F_n}^{-1},\quad c_1(TZ)=2\xi.
\]
The class $\xi$ is nef. For every positive $\beta$ and every $\ell\geq0$,
\[
 \overline{\mathcal M}_{0,\ell}(X_n,\beta)
 \cong\overline{\mathcal M}_{0,\ell}(Z,i_{0*}\beta)
\]
with their standard obstruction theories, virtual classes, and virtual
structure sheaves. Moreover
\[
 H_2(Z,\Z)=i_{0*}H_2(F_n,\Z)\oplus\Z\ell_f,\qquad\xi\cdot\ell_f=1,
\]
and $k\mathfrak b=i_{0*}\beta$ if and only if $\mathfrak b=i_{0*}\alpha$ and
$k\alpha=\beta$. For such classes, $\mathfrak b$ is effective if and only if
$\alpha$ is effective.
\end{theorem}
\begin{proof}
A line with nonzero first coordinate has a unique representative
$(1,v)$, identifying the open complement with $K_{F_n}$.
Projection of the tautological subline to $\mathcal O_{F_n}$ cuts out
$D_\infty$, proving its divisor class and its restriction to the
zero section. On the infinity section the tautological subline is
$K_{F_n}$, giving the displayed positive normal bundle.
The relative Euler sequence gives
$c_1(T_{Z/F_n})=2\xi+p^*c_1(K_{F_n})$; adding $p^*c_1(T_{F_n})$ leaves $2\xi$.
The dual bundle $\mathcal O_{F_n}\oplus K_{F_n}^{-1}$ is globally generated,
and surjects onto $\mathcal O_Z(1)$, so $\xi$ is nef.

A curve not contained in $D_\infty$ has nonnegative intersection with
that effective divisor, and a nonconstant curve contained in it has
strictly positive intersection because $K_{F_n}^{-1}$ is ample.
In a stable map of class $i_{0*}\beta$ the total $\xi$-degree is zero.
Every nonconstant component must therefore avoid $D_\infty$.
A contracted component in $D_\infty$ would connect, through contracted
components, to a nonconstant component meeting it, which is impossible.
Thus every geometric map lands in $X_n$.
For families, maps with image in an open target form an open substack:
the bad locus is the proper image from the universal curve of the
inverse image of $D_\infty$. It has no geometric points by the preceding
argument and hence is empty, over arbitrary bases. The open immersion
of moduli stacks is therefore an isomorphism including nilpotents.
Since $T_Z|_{X_n}=T_{X_n}$, the actual obstruction complexes and their maps
agree; the virtual classes and the virtual structure sheaves agree too.

The projective-bundle formula and the section give the direct-sum
lattice decomposition. Intersecting $k\mathfrak b=i_{0*}\beta$ with $\xi$
annihilates the fiber coefficient, and projection to $F_n$ gives
$k\alpha=\beta$. For effectivity, one implication uses the zero
section. Conversely push an effective cycle of class $i_{0*}\alpha$
to $F_n$; its pushforward is an effective cycle of class $\alpha$.
This is a statement about effective cycles, not an assertion that
all such classes have nonempty genus-zero stable-map moduli.
\end{proof}

\begin{remark}[Positivity and the scope of the arithmetic theorems]
\label{rem:semipositive}
Algebraically, $c_1(TZ)=2\xi$ is nef. The symplectic hypothesis in
Theorem~\ref{thm:IP} below is different, because in complex
dimension $n$ there must be no spherical class $\mathfrak a$ with positive
symplectic area and $3-n\leq c_1(\mathfrak a)<0$.
For $n=4$ the only forbidden value is $-1$, impossible since $c_1(TZ)$
is even. For $n\geq5$, the classes
\[
 \mathfrak a_k=i_{0*}(k e_i)-\ell_f
\]
have $c_1(\mathfrak a_k)=-2$ and positive area for all sufficiently large $k$,
for every fixed Kähler form on $Z$. 
They are spherical, since coordinate curves in $F_n$ exist by Theorem~\ref{thm:degreeone}, as their parameter scheme has nonzero class $h^2$, and the fiber is a sphere.
The image of $\pi_2(Z)\to H_2(Z,\Z)$ is a subgroup, so their indicated difference
is spherical. Thus $Z$ fails this symplectic condition for $n\geq5$.
The existence of these topological classes does not say that they are
represented by holomorphic curves for the algebraic complex structure.
\end{remark}

\subsection{A compact target with trivial canonical bundle}
\begin{theorem}[Calabi--Yau double cover and isolation of the local sector]
\label{thm:CYcover}
There is a smooth connected projective $n$-fold $\widetilde Z$,
a finite flat double cover $r:\widetilde Z\to Z$, and disjoint
divisors $F_{n,+},F_{n,-}\cong F_n$ over $D_0$, such that
\[
 K_{\widetilde Z}\cong\mathcal O_{\widetilde Z},\qquad
 N_{F_{n,\pm}/\widetilde Z}\cong K_{F_n} .
\]
Write $q=p\circ r$ and $j_\pm:F_n\hookrightarrow\widetilde Z$.
For every $\beta>0$ and $\ell\geq0$,
\[
 \overline{\mathcal M}_{0,\ell}(\widetilde Z,j_{+*}\beta)
 \cong\overline{\mathcal M}_{0,\ell}(F_n,\beta)
 \cong\overline{\mathcal M}_{0,\ell}(X_n,\beta),
\]
with the standard obstruction theories on the first and last stacks
identified. The virtual classes and virtual structure sheaves agree.
Every effective integral covering root of $j_{+*}\beta$ is
$j_{+*}\alpha$ with $k\alpha=\beta$; conversely each effective such
$\alpha$ gives a covering root. The same assertion holds modulo torsion.
\end{theorem}
\begin{proof}
Write $\mathcal O_{F_n}(H)=K_{F_n}^{-1}$, which is very ample, and put
$M=\mathcal O_Z(2\xi)=-K_Z$.
Indeed the relative tangent determinant is
$\det T_{Z/F_n}=\mathcal O_Z(2\xi)\otimes p^*K_{F_n}$;
multiplication by $p^*\det T_{F_n}$ cancels the $K_{F_n}$ factor.
Thus, the canonical bundle $K_Z$ is isomorphic to the line bundle $\mathcal{O}_Z(-2\xi)$.
In the lines convention, the two tautological projections give sections
\[
 u\in H^0(Z,\mathcal O_Z(\xi)),\qquad
 v\in H^0(Z,\mathcal O_Z(\xi)\otimes p^*K_{F_n}).
\]
They never vanish simultaneously; $D_\infty=(u=0)$ and $D_0=(v=0)$.

Choose $a\in H^0(F_n,\mathcal O_{F_n}(4H))$ with smooth zero divisor.
Existence can be seen directly, without invoking a smoothness assertion
for an arbitrary semiample system. Embed $F_n$ by the very ample
$\mathcal O_{F_n}(4H)$ in $\PP^N$. If the hyperplane corresponding to a
section is singular at a point $x$, it contains the projective tangent
$(n-1)$-plane $T_xF$. The incidence variety of such pairs has dimension
$n-1+(N-n)=N-1$ and is projective over the dual $\PP^N$. Its image is
therefore a proper closed subset. Hence a hyperplane outside this image
has smooth section. The section is nonempty since an ample divisor on the
positive-dimensional projective variety $F_n$ has positive intersection
with $H^{n-2}$.

The section
\[
 b=u^4+p^*a\,v^4\in H^0(Z,\mathcal O_Z(4\xi))=H^0(Z,M^2)
\]
has smooth nonempty zero divisor $B$ disjoint from $D_0$.
Indeed, on $u\ne0$ its equation is $1+a w^4=0$, where $w=v/u$
is the fiber coordinate of $K_{F_n}$. On its zero set $aw\ne0$,
so the fiber derivative $4aw^3$ does not vanish.
On $v\ne0$ the equation in a local trivialization is $t^4+a=0$.
If $t\ne0$ its fiber derivative is nonzero; if $t=0$, smoothness
of $(a=0)$ gives a nonzero base derivative.
This proves smoothness everywhere, including along infinity.

Define
\[
 \widetilde Z=\operatorname{Spec}_Z(\mathcal O_Z\oplus M^{-1}),
 \qquad M^{-1}\otimes M^{-1}\xrightarrow{\,b\,}\mathcal O_Z.
\]
Locally the equation is $\tau^2=b$; hence the morphism is finite flat
of degree two and the total space is projective.
Away from $B$ it is étale. On $B$, its equation has nonzero base
differential $db$, so $\widetilde Z$ is smooth.
The section $b$ has a zero of order one along a branch component,
so it cannot be a square in the function field after a rational
trivialization of $M$. The generic cover is a field extension.
The algebra is locally free over the integral $Z$, and embeds in that
field, proving that $\widetilde Z$ is integral and thus connected.

Viewed as the hypersurface $\tau^2-b=0$ in $\Tot(M)$,
its canonical bundle is
\[
 K_{\widetilde Z}
  =r^*(K_Z\otimes M^{-1}\otimes M^2)
  =r^*(K_Z\otimes M)=\mathcal O_{\widetilde Z}.
\]
This follows from the tangent determinant sequence for the smooth hypersurface,
just as in Proposition~\ref{prop:geometry}. Equivalently, if $R_b=(\tau=0)$ is
the ramification divisor, the local equation $b=\tau^2$ gives $db=2\tau\,d\tau$
and hence $K_{\widetilde Z}=r^*K_Z\otimes\mathcal O(R_b)$. The tautological
section $\tau$ of $r^*M$ has divisor $R_b$, so this again gives the trivial
canonical line bundle. This computation proves an isomorphism of line bundles,
establishing more than the vanishing of a rational Chern class. Along $D_0$, $u$
trivializes $\mathcal O_Z(\xi)$ and $b=u^4$. Consequently $r^{-1}(D_0)$ is the
disjoint union $\tau=+u^2$ and $\tau=-u^2$, each isomorphic to $F_n$.
\'Etaleness along both components identifies their normal bundles with
$N_{D_0/Z}=K_{F_n}$. In particular
\begin{equation}\label{e:intersection_numbers}
 F_{n,+}\cdot j_{+*}\alpha=-H\cdot\alpha,\qquad
 F_{n,+}\cdot j_{-*}\alpha=0,\qquad q_*j_{\pm*}=\operatorname{id}.
\end{equation}

Now let a family of stable maps have class $j_{+*}\beta$.
Composition with the finite morphism $r$ preserves stability:
exactly the same components are contracted, and a nonconstant map
still has finite automorphism group.
Its class on $Z$ is $i_{0*}\beta$.
Theorem~\ref{thm:compact}, followed by zero-section factorization,
shows that this composed map factors scheme-theoretically through $D_0$.
The original map therefore factors through
$r^{-1}(D_0)=F_{n,+}\amalg F_{n,-}$ over the entire base.
For a geometric fiber the connected domain lies in one component.
The intersection numbers in eq.~\eqref{e:intersection_numbers} and $H\cdot\beta>0$ exclude $F_{n,-}$.
For a family, the two alternatives define open-and-closed subsets of the
base. Equivalently, the two idempotents pull back through
$\pi_*\mathcal O_C=\mathcal O_S$.
Since no geometric fiber lies in $F_{n,-}$, the entire family lies in $F_{n,+}$.
This proves the asserted stack isomorphism over arbitrary bases.

Along $F_{n,+}$ the differential $dr$ is an isomorphism
$j_+^*T_{\widetilde Z}\cong i_0^*T_Z=T_{F_n}\oplus K_{F_n}$.
It identifies the universal stable-map obstruction morphisms,
because they are induced by the same composed universal differential.
For an explicit comparison, let $\mathfrak M$ be the common stable-map
stack, let $\mathfrak B$ be the same smooth prestable-curve stack used
above, and let $\pi:\mathcal C\to\mathfrak M$ be the universal curve.
Composition with $r$ does not stabilize or alter this curve. The
obstruction morphisms fit into the commuting diagram
\begin{equation}\label{eq:CYcoverPOT}
\begin{array}{ccc}
 (R\pi_*f_Z^*T_Z)^\vee&\xrightarrow{\ \sim\ }&
 (R\pi_*f_{\widetilde Z}^*T_{\widetilde Z})^\vee\\[2mm]
 \phi_Z\downarrow&&\downarrow\phi_{\widetilde Z}\\[1mm]
 \mathbb L_{\mathfrak M/\mathfrak B}&\xrightarrow{\ \operatorname{id}\ }&
 \mathbb L_{\mathfrak M/\mathfrak B}.
\end{array}
\end{equation}
The chain rule for the universal differential, followed by the same
relative-duality and trace construction on $\mathcal C$, proves commutativity.
Global resolutions are supplied by \cite[Proposition 5]{BehrendGW}. Restriction
through the open immersion $X_n\subset Z$ makes the same comparison with the
theory of $X_n$. Thus the theories agree with those of \(X_n\), including their
maps to the cotangent complex.
Consequently, the induced isomorphism of stable-map stacks identifies the
perfect obstruction theories and hence the associated virtual classes and
virtual structure sheaves. The standard theory of \(F_n\), by contrast, has the
excess factor described in Theorem~\ref{thm:virtual} and
Remark~\ref{rem:localKexcess}.

Finally suppose $k\gamma=j_{+*}\beta$ and $\gamma$ is effective.
Set $\alpha=q_*\gamma$; then $k\alpha=\beta$.
The torsion-free splitting of $H_2(Z)$ in Theorem~\ref{thm:compact}
gives $r_*\gamma=i_{0*}\alpha$.
Represent $\gamma$ by an effective cycle. Nefness of $r^*\xi$
and zero total degree force every component downstairs into $D_0$,
by the curve argument in that theorem. More explicitly, a complete
irreducible curve of zero $\xi$-degree avoids $D_\infty$ and hence
lies in $X_n$. On its normalization, a nonconstant projection to $F_n$
pulls $K_{F_n}$ back to a negative-degree line bundle, whose section is
zero. A curve projecting to a point cannot be complete and nonconstant
inside an affine-line fiber. Thus every such curve lies in $D_0$;
this argument applies to effective cycles of any geometric genus.
Thus the cycle splits between $F_{n,+}$ and $F_{n,-}$; write its projected
effective classes as $\alpha_+$ and $\alpha_-$.
Equation~\eqref{e:intersection_numbers} gives
$H\cdot\alpha_+=H\cdot\beta/k=H\cdot\alpha$.
Since $\alpha=\alpha_++\alpha_-$, ampleness forces $\alpha_-=0$.
The cycle lies in $F_{n,+}$, so $\gamma=j_{+*}\alpha$.
Conversely the inclusion preserves effectivity.
The proof still works if the initial equality is modulo torsion:
intersection numbers kill torsion, $H_2(Z)$ is free, and
$q_*j_{+*}=\operatorname{id}$ with free $H_2(F_n)$ gives the same
conclusion. In particular there are no additional effective torsion
lifts contributing to the invariants in this sector.
\end{proof}

\begin{lemma}[Formal realization of the local neighborhood]
\label{lem:formalCYsector}
The morphism $r$ induces an isomorphism
\[
 \widehat{\widetilde Z}_{F_{n,+}}\cong\widehat X_{F_n}
\]
of formal completions along the indicated divisors.
\end{lemma}
\begin{proof}
On $u\ne0$ write $w=v/u$ and $\eta=\tau/u^2$. The product $a w^4$
is a global regular function, although $w$ is locally the coordinate of
$K_{F_n}$. The equation of the cover is $\eta^2=1+a w^4$. In the completion
along the zero section the unique root with constant term $+1$ is
\[
 \eta=(1+a w^4)^{1/2}
      =\sum_{j\ge0}\binom{1/2}{j}(a w^4)^j.
\]
This converges in the ideal-adic topology and is compatible on every
infinitesimal neighborhood. Uniqueness also follows because the sum of
two roots with constant term $1$ is a unit. Since $a w^4$ is global,
the expressions glue and supply the inverse to the completed morphism $r$.
\end{proof}

\begin{remark}[A compact sector, not an asserted global open embedding]
The scheme $Z$ is an open-embedding compactification of $X_n$.
The variety $\widetilde Z$ is a compact target realizing the local
stable-map sector, with the formal-neighborhood identification just
proved. No isomorphism of an entire Zariski-open neighborhood of $F_{n,+}$
with all of $X_n$ is needed or asserted. The direct family argument in
Theorem~\ref{thm:CYcover} already establishes the sector over arbitrary
bases; the formal calculation is an additional description.
\end{remark}

\medskip\noindent\textbf{Integral roots and the Novikov convention.}\label{par:Novikovroots}
The local transforms below use actual integral classes in the free
lattice $H_2(F_n,\Z)$. When an external formula on $\widetilde Z$ is written
using $H_2(\widetilde Z,\Z)/\mathrm{tors}$, its coefficient is understood
as the sum of the invariants of the integral lifts. For a class in the
sector and for each of its roots, Theorem~\ref{thm:CYcover} proves that
there is at most one effective lift, namely $j_{+*}\alpha$; all other
lifts have zero stable-map invariants. The retraction
$q_*j_{+*}=\operatorname{id}$ also shows that the divisibility index of
$j_{+*}\beta$ agrees with that of $\beta$, both integrally and modulo
torsion. Thus no root is gained or lost on passing to the external
convention. This does not require the entire group
$H_2(\widetilde Z,\Z)$ to be torsion-free.
%------------------------------------------------------------------------
\subsection{The rational primary transform}
\begin{theorem}[Primary transform]\label{thm:integrality}
There are unique rational numbers $b_\beta$, with the zero convention
for noneffective classes, such that
\[
 N_{0,\beta}=\sum_{k\alpha=\beta}\frac{b_\alpha}{k^2},\qquad
 b_\beta=\sum_{k\alpha=\beta}\frac{\mu(k)}{k^2}N_{0,\alpha}.
\]
Here $k$ is a positive integer, $\alpha$ is an integral curve class,
and $\mu$ is the arithmetic Mobius function. With
$N_m=\sum_{|\beta|=m}N_{0,\beta}$ and
$n_m=\sum_{|\beta|=m}b_\beta$, one has the formal identities
\begin{align}
 \Dcal_0^\Delta(Q)&=\sum_{m\geq1}n_m\Li_2(Q^m),\label{eq:Li2}\\
 N_m&=\sum_{k\mid m}\frac{n_{m/k}}{k^2},\qquad
 n_m=\sum_{k\mid m}\frac{\mu(k)}{k^2}N_{m/k}\in\Q.\label{eq:mobius}
\end{align}
These statements alone make no integrality or enumerative assertion.
\end{theorem}
\begin{proof}
In a nonzero nonnegative degree tuple, any integral root has order
dividing its total degree. There are therefore finitely many roots.
The term $k=1$ defines $b_\beta$ inductively. Substitution of the proposed
inverse gives the coefficient
$r^{-2}\sum_{k\mid r}\mu(k)$ in front of $N_{0,\beta/r}$.
It is one for $r=1$ and zero otherwise. If $r>1$, expanding
$\prod_{p\mid r}(1-1)$ proves the last assertion. This proves the
inverse without an external inversion theorem. A multiple of an
effective class is effective, so the zero convention is consistent.
The number of tuples at a fixed total degree is finite. Summation over
them and the defining series $\Li_2(q)=\sum_{k\geq1}q^k/k^2$ give~\eqref{eq:Li2}--\eqref{eq:mobius}.
\end{proof}

The two-loop case with $n=3$ is not covered by the hypotheses used here; its relation to
local mirror symmetry is the subject of Bloch--Kerr--Vanhove~\cite{BKV}. The
comparison between the unmarked $r^{-3}$ prepotential convention of that
setting and the one-primary $r^{-2}$ convention used throughout this paper, and
the separation of the cover exponent from the lower-loop restrictions
on curve lattices and threshold convergence, are carried out
in the companion paper~\cite{PaperII}.
%------------------------------------------------------------------------
\subsection{An exact integral \texorpdfstring{$K$}{K}-class and the quantum-\texorpdfstring{$K$}{K} identity}
Put $\mathcal P_i=p^*\mathcal O_{F_n}(p_i)$ and
\begin{equation}\label{eq:Klift}
 E_\delta=\sum_{|S|=n-2}\prod_{i\in S}(1-[\mathcal P_i^{-1}])
       \in K^0_{\mathrm{alg}}(Z).
\end{equation}
This is an integral virtual vector-bundle class, and its pullback
$\widetilde E_\delta=r^*E_\delta$ is an integral algebraic $K$-class on
$\widetilde Z$.
\begin{lemma}\label{lem:Klift}
Its full Chern character is
$\ch(E_\delta)=p^*(H^{n-2}/(n-2)!)$.
\end{lemma}
\begin{proof}
The square of each $p^*p_i$ is zero, so
$\ch(1-[\mathcal P_i^{-1}])=1-e^{-p^*p_i}=p^*p_i$ exactly.
Expansion of the products gives the assertion.
\end{proof}
In particular $\ch(\widetilde E_\delta)=q^*\delta_0$ exactly in every
degree. No assertion that an arbitrary integral cohomology class has an
integral inverse Chern character is used. The construction is for the
distinguished insertion, not a general integral-lifting principle.

Define
\[
 \kappa_\alpha=\chi\left(\overline{\mathcal M}_{0,1}(Z,i_{0*}\alpha),
          \ev^*E_\delta\otimes\mathcal O^{\vir}\right)\in\Z.
\]
The virtual structure sheaf and the requisite perfect obstruction
theory are specified in~\cite[Section 2.3 and Proposition 1]{Lee}.
This is an Euler characteristic of a coherent
$K$-class on a proper Deligne--Mumford stack, not an integral on its
coarse space. 
In characteristic zero, its coarse-space pushforward is exact. Étale-locally,
this is the invariants functor for a finite group, and the averaging
idempotent splits the invariants.
Coherent cohomology on
the proper coarse space is finite and bounded. Thus the Euler
characteristic is an integer. Equivalently, it is the true, rather
than the fake, quantum-$K$ invariant. Its definition on the smooth
projective target $Z$ does not require semipositivity. Semipositivity is
used for the external identity below, which is applied on $\widetilde Z$.
Theorem~\ref{thm:CYcover} identifies that target's invariant with
$\kappa_\alpha$, including the insertion and virtual structure sheaf.

The standing target of~\cite{Chou} is a smooth complex projective variety
$M$, assumed in its Section 3 to be semipositive of complex dimension
$d\geq3$; that paper identifies its Theorem 1.1 with the Ionel--Parker
result~. 

\begin{theorem}[Chou, {\cite[Proposition 3.4]{Chou}}]\label{thm:chou}
In that setting let $\beta\ne0$ be a positive integral curve class,
$\gamma\in H^{2(d-K_M\cdot\beta-2)}(M)$ homogeneous, and
$\Gamma=\ch^{-1}(\gamma)\in K^0(M)_\Q$. Then
\[
 \mathrm{GV}_\beta(\gamma)=\mathrm{QK}_\beta(\Gamma)
 +\delta_{K_M\cdot\beta,0}
  \sum_{\substack{r\mid\operatorname{ind}(\beta)\\r>1}}
                          \mu(r)r\,\mathrm{QK}_{\beta/r}(\Gamma),
\]
where $\operatorname{ind}(\beta)=\max\{r\geq1:\beta/r\in H_2(M,\Z)\}$,
$\mathrm{QK}$ is the true one-point genus-zero quantum-$K$ invariant, and
$\mathrm{GV}$ uses primary weight $r^{-2}$ in the zero-Chern sector and
equals $\mathrm{GW}$ in the positive-Chern sector.
\end{theorem}

We apply this to $\widetilde Z$, whose canonical bundle is trivial, so that
both the symplectic condition and nonnegative Chern degree on algebraic
curves hold automatically; no extension to a nef-anticanonical target
outside the stated semipositive scope is invoked. The source uses rational
topological $K$-theory, and its statement alone furnishes no integral lift.
Here $E_\delta$ is an integral algebraic $K$-class with exactly homogeneous
full Chern character, the lattices of $F_n$ and $Z$ are free, and for
$\widetilde Z$ we use the torsion-free quotient with the sector
identification of Theorem~\ref{thm:CYcover}, so every contributing root is
unambiguous; for lattices with torsion a root-index convention would need
additional care. Theorem~\ref{thm:QK} applies the statement to
$\widetilde Z$ and not to $Z$, with exact homogeneous lift $r^*E_\delta$.

\begin{theorem}[Quantum-$K$ expression]\label{thm:QK}
For every $n\geq4$,
\begin{equation}\label{eq:QKformula}
 b_\beta=\sum_{k\alpha=\beta}\mu(k)k\,\kappa_\alpha,
 \qquad \kappa_\beta=\sum_{k\alpha=\beta}k\,b_\alpha.
\end{equation}
\end{theorem}
\begin{proof}
On $\widetilde Z$ use the integral algebraic class $r^*E_\delta$.
Its full character is $q^*\delta_0$, with complex degree $n-2$.
Since $\widetilde Z$ is a smooth projective variety, its underlying space
has the homotopy type of a finite CW complex, and the Chern character
$\ch:K^0_{\mathrm{top}}(\widetilde Z)\otimes\Q\to H^{\mathrm{even}}(\widetilde Z,\Q)$
is an isomorphism. Thus the equality of characters identifies the image of
$r^*E_\delta$ with the rational topological class
$\ch^{-1}(q^*\delta_0)$ required by Chou's Proposition~3.4.
The target is smooth projective, dimension $n\geq4$, and $c_1=0$;
both symplectic and algebraic positivity conditions hold automatically.
Thus Theorem~\ref{thm:chou} applies without a nef-scope extension. Theorem~\ref{thm:CYcover} identifies the
moduli and actual obstruction theories for all effective roots with
those on $Z$ and $X_n$. Consequently their virtual structure sheaves
and the insertions agree, so their true quantum-$K$ invariants are
precisely the $\kappa_\alpha$ defined above.

%For complete clarity on the curve-lattice convention, one may work
%in $H_2(\widetilde Z,\Z)/\mathrm{tors}$. 
The retraction $q_*$ to
$H_2(F_n,\Z)=\Z^n$ shows that the divisibility index of $j_{+*}\beta$
is exactly that of $\beta$. The effective-root proof of
Theorem~\ref{thm:CYcover} shows that passing to this quotient introduces
no extra effective class or stable-map component in the sector at hand.
Thus every root in the source formula is exactly accounted for.
The GV coefficients in the cited proposition are defined by the same
one-primary $k^{-2}$ transform as Theorem~\ref{thm:integrality}.
Uniqueness of that rational transform identifies them with $b_\beta$.
The first equality in~\eqref{eq:QKformula} is therefore the cited proposition.
The convolution of $k\mapsto k$ and $k\mapsto\mu(k)k$ at $r$ is
$r\sum_{k\mid r}\mu(k)$, which gives the inverse equality.
This also proves integrality directly from true $K$-theoretic Euler
characteristics, independently of the symplectic integrality route.
\end{proof}

\begin{corollary}[Primary integrality in every loop order]
\label{cor:primaryintegral}
For every $n\geq4$, $b_\beta\in\Z$ and $n_m\in\Z$.
\end{corollary}
\begin{proof}[Algebraic proof]
Theorem~\ref{thm:QK} writes $b_\beta$ as the finite sum
$\sum_{k\alpha=\beta}\mu(k)k\,\kappa_\alpha$ of integers, because the
$\kappa_\alpha$ are true coherent Euler characteristics on proper
Deligne--Mumford stacks. Finite summation at fixed $H$-degree gives
$n_m\in\Z$. This proof uses the compact sector, the exact integral
$K$-lift, and Chou's one-primary identity, but neither a mirror theorem
nor an algebraic--symplectic comparison.
\end{proof}

\begin{theorem}[Ionel--Parker, {\cite[Theorem 9.2 and equations (9.9)--(9.10)]{IP}}]
\label{thm:IP}
Let $(M,\omega)$ be a closed semipositive symplectic manifold of real
dimension $2d\geq6$ and let
$\gamma_1,\ldots,\gamma_k\in H^{\mathrm{even}}(M,\Z)$ be integral
primary insertions of even degree. Define $n_{\mathfrak a,0}$ in the zero-Chern
sector by
\[
 \sum_{\mathfrak a\ne0,\ c_1(\mathfrak a)=0}
   \mathrm{GW}_{\mathfrak a,0}(\gamma_1,\ldots,\gamma_k)q^{\mathfrak a}
 =\sum_{\mathfrak a\ne0,\ c_1(\mathfrak a)=0}
   n_{\mathfrak a,0}(\gamma_1,\ldots,\gamma_k)\sum_{r\geq1}r^{k-3}q^{r\mathfrak a},
\]
and put $n_{\mathfrak a,0}=\mathrm{GW}_{\mathfrak a,0}$ when
$c_1(\mathfrak a)>0$. Then all these $n_{\mathfrak a,0}$ are integers.
\end{theorem}

Semipositivity here excludes spherical classes $\mathfrak a$ with
$\omega(\mathfrak a)>0$ and $3-d\leq c_1(\mathfrak a)<0$. It is a condition
on a closed smooth manifold and its symplectic form, not only on effective
algebraic curves, and the theorem concerns primary genus-zero invariants.
Its hypothesis holds for $Z$ when $n=4$ and fails for $n\geq5$, by
Remark~\ref{rem:semipositive}. The all-dimensional application below is
instead to $\widetilde Z$, which has $c_1(T\widetilde Z)=0$, real dimension
$2n\geq8$, and integral insertion $q^*\delta_0$ of complex codimension
$n-2$, equal to the one-marked virtual dimension; these check the target,
degree, and insertion hypotheses for every $n\geq4$.

\begin{theorem}[Siebert, {\cite[Theorem~0.1]{Siebert}}]\label{thm:siebert}
Let $M$ be a complex projective manifold, $R\in H_2(M,\Z)$, and $g,k\geq0$.
On the moduli space $C_{R,g,k}(M)$ of genus-$g$, $k$-marked stable maps in
class $R$, the homology class associated to Behrend's algebraic virtual
fundamental class and Siebert's symplectic virtual fundamental class
coincide. In particular their numerical invariants with matching insertions
coincide.
\end{theorem}

The comparison target for the all-dimensional primary-integrality argument
is the smooth projective $\widetilde Z$, whose canonical bundle is trivial;
the local-to-compact moduli and actual obstruction-theory identification is
proved in Theorem~\ref{thm:CYcover}, before this use. No comparison theorem
for a noncompact target is presumed.

\medskip\noindent\textbf{The semipositive numerical convention.}
The symplectic proof of Corollary~\ref{cor:primaryintegral} uses the
semipositive perturbed-map interpretation permitted by
\cite[Remark 1.1]{IP}. Theorem~\ref{thm:siebert} compares Behrend's
algebraic class with Siebert's symplectic construction, and
\cite[Introduction]{SiebertGeneral}
states agreement of Siebert's invariants with the Ruan--Tian invariants on
semipositive targets.
The numerical route used here is
therefore
\[
 \mathrm{GW}^{\mathrm{Behrend}}=\mathrm{GW}^{\mathrm{Siebert}}
 =\mathrm{GW}^{\mathrm{Ruan\text{--}Tian}},
\]
with Ionel--Parker read in its permitted semipositive convention. This
specifies the comparison being used; it does not assert that every
virtual-cycle framework is identified by~\cite[Theorem~0.1]{Siebert}.
The algebraic proof through Theorem~\ref{thm:QK} is independent of it.

\medskip\noindent\textbf{The one-marked case.}
The construction of~\cite{SiebertGeneral} includes the convention that the
unstable moduli of marked domains is a point. Alternatively the one-primary
numerical comparison reduces to three markings. For a nonzero curve class
$A$ and an ample divisor $D$ on the compact target, the divisor equation
gives
$\langle\eta,D,D\rangle_{0,3,A}=(D\cdot A)^2\langle\eta\rangle_{0,1,A}$,
and the factor is nonzero for effective nonzero $A$, so comparison of the
three-point rational invariants implies comparison of the one-point ones.
Dividing by that factor compares invariants; it does not prove one-point BPS
integrality, which uses the one-primary conclusion of
Theorem~\ref{thm:IP} or the algebraic argument already given.

\begin{proof}[Symplectic proof with the semipositive convention]
Theorem~\ref{thm:CYcover} gives the smooth projective $n$-fold
$\widetilde Z$ with trivial canonical bundle. With any Kähler form,
it is a closed symplectic $2n$-manifold and $c_1=0$; hence it is
semipositive in the exact sense of Theorem~\ref{thm:IP}, in every
dimension under consideration.
The insertion $q^*\delta_0$ is integral and has complex codimension
$n-2$, the one-marked virtual dimension. Algebraic and symplectic invariants agree through
Theorem~\ref{thm:siebert} and the semipositive comparison convention
recorded above; this is a specified numerical comparison, not an assertion
equating every virtual-cycle construction.
Apply the one-primary formula of Theorem~\ref{thm:IP} to
$j_{+*}\beta$. Theorem~\ref{thm:CYcover} identifies its invariants,
insertions, and all effective covering roots with the local ones.
Uniqueness of the rational transform identifies its integral
coefficients with $b_\beta$. Finite summation gives $n_m$.
This proof does not use the quantum-$K$ identity.
\end{proof}

\begin{remark}[True and fake invariants, and boundary automorphisms]
\label{rem:truefake}
Let $\mathfrak M_\beta=\overline{\mathcal M}_{0,1}
(\widetilde Z,j_{+*}\beta)$ and use the insertion $\widetilde E_\delta$.
The fake invariant is the identity-sector Todd integral
\[
 \kappa_\beta^{\mathrm{fake}}
 :=\int_{[\mathfrak M_\beta]^{\vir}}
       \ch(\ev^*\widetilde E_\delta)\,
       \operatorname{Td}(T_{\mathfrak M_\beta}^{\mathrm{vir}})
 =N_{0,\beta}.
\]
The equality follows because the full Chern character is homogeneous
of exactly the virtual dimension, so only the constant Todd term
contributes. True quantum $K$-theory instead takes coherent cohomology
on the stack. The nonidentity inertia terms in the virtual
Kawasaki--Riemann--Roch calculation behind Theorem~\ref{thm:chou}
must be retained. Theorems~\ref{thm:integrality} and~\ref{thm:QK} give
\[
 \kappa_\beta-\kappa_\beta^{\mathrm{fake}}
 =\sum_{r\alpha=\beta}(r-r^{-2})b_\alpha.
\]
Primitivity of the total curve class alone does not exclude a boundary
map with a multiply-covered component and nontrivial automorphisms.
For a primitive class the displayed identity asserts that the total
nonidentity correction vanishes, not that every nonidentity stratum
is empty. Trivial stabilizers in the specific classes $e_i$ and
$e_i+e_j$ follow from the separate geometric descriptions in
Sections~\ref{sec:degreeone}--\ref{sec:degreetwo}.
\end{remark}

\begin{remark}[The integral ambient lattice for other insertions]
\label{rem:integralambientlattice}
The compact-sector comparison transfers any insertion
$\eta\in H^{2n-4}(F_n,\Z)$ by $q^*\eta$. Rational ambientness, however,
does not imply that every such integral class is an integral combination
of the square-free monomials used in $E_\delta$. Here is a precise
lattice calculation. Poincar\'e duality, the integral weak-Lefschetz
identification, and the universal coefficient theorem identify
$H^{2n-4}(F_n,\Z)$ with the integral dual of $H^2(F_n,\Z)$, with coordinates
$\eta\mapsto(\int_{F_n}\eta p_j)_{j=1}^n$. For $|S|=n-2$,
\[
 \int_{F_n} p_Sp_j=\int_{Y_n} Hp_Sp_j
 =\begin{cases}1,&j\notin S,\\0,&j\in S.\end{cases}
\]
Thus a monomial with complement $\{a,b\}$ gives $e_a+e_b$.
These vectors generate the index-two sublattice
\[
 \{(v_1,\ldots,v_n)\in\Z^n:\textstyle\sum_i v_i\text{ is even}\}.
\]
Indeed their sums have even parity, whereas differences $e_i-e_j$
and the vectors $2e_i$ are integer combinations of pair sums; these
generate every vector of even total sum. The distinguished $\delta_0$
is itself the sum of these monomials, so its lift is unaffected.
This calculation neither rules out other integral $K$-lifts nor weakens
the primary theorem. It only prevents extending the displayed exact
algebraic lift to arbitrary integral insertions without a further
argument. The symplectic primary theorem concerns integral insertions
more generally; an exact algebraic quantum-$K$ expression for another
insertion must check its own Chern-character lift.
Corollary~\ref{ext:cor:vector} below nevertheless proves integral
coefficients of the full primary cover transform by pairing with integral
duals of coordinate curves.
\end{remark}

\begin{corollary}[The highest logarithmic coefficient]
\label{cor:arithmeticbridge}
Write $a_{n-2}(m)=[Q^m]h_{n-2}(Q)$. Unconditionally,
\[
 a_{n-2}(m)=-mN_m=-m\sum_{k\mid m}\frac{n_{m/k}}{k^2},
 \qquad
 h_{n-2}(Q)=\sum_{m\geq1}m n_m\log(1-Q^m),
\]
and
\[
 \exp(h_{n-2}(Q))=\prod_{m\geq1}(1-Q^m)^{m n_m}.
\]
Here rational powers mean their formal binomial expansions.
The $n_m$ are integral for every $n\geq4$ by
Corollary~\ref{cor:primaryintegral}. The symbol $a_{n-2}$ is distinct from the
coefficient $d_{n-2}$ in the alternative period normalization.
\end{corollary}
\begin{proof}
Use $h_{n-2}=-\theta_Q\Dcal_0^\Delta$ and
$\theta_Q\Li_2(Q^m)=-m\log(1-Q^m)$.
At any fixed degree only finitely many factors contribute.
\end{proof}

\begin{remark}[Interpretation for a Calabi--Yau fivefold]
For $n=5$, $\dim_\C X_n=5$ and the primary insertion is $H^3/6$.
The exponent $k^{-2}$ comes from one primary marking, not from
the dimension of $X_n$. Formula~\eqref{eq:Li2} is a definitionally
invertible transformation of rational invariants. Integrality adds a
separate arithmetic assertion; it does not, by itself, identify the
coefficients with counts of embedded curves or with contributions of
individual covering loci. In particular no weight for the descendants
$\psi^a\delta_a$, $a>0$, follows from this primary transform. No
genus-one or higher-genus covering formula is asserted here.
\end{remark}

\medskip\noindent\textbf{Primary integers and higher-dimensional interpretation.}
The $b_\beta$ are integral genus-zero primary GV-type invariants for the
specified insertion $\delta_0$. The general primary convention with
$\ell$ insertions has cover weight $r^{\ell-3}$, so its one-marked
specialization is $r^{-2}$; the complex codimension of an insertion
is not its marking number. The comparison with $r^{-3}$ on the same
threefold is derived in the companion paper~\cite{PaperII}.

Equation (0.2) of~\cite{PZ}, in its introductory setting of a smooth
projective complex Calabi--Yau $d$-fold, defines primary numbers by
\[
 \sum_{\beta\ne0}N_{0,\beta}(\gamma_1,\ldots,\gamma_k)q^\beta
 =\sum_{\beta\ne0}n_{0,\beta}(\gamma_1,\ldots,\gamma_k)
                         \sum_{r\geq1}\frac{q^{r\beta}}{r^{3-k}},
 \qquad \gamma_i\in H^*(X_n,\Z).
\]
This is a definition and integrality prediction in that source,
not an unconditional theorem about all covering loci.
The fivefold genus-one calculations in that paper require its stated
ideal-geometry framework; none of those calculations is an input here.
The weight for one primary insertion is $r^{-2}$. This convention
is distinct from any assertion about descendant or genus-one covering laws.

The targets have complex dimension four at $n=4$, five at $n=5$, and at
least six at $n\ge6$. The ideal-curve and moving-cover constructions
discussed in~\cite[Sections 0.5 and 1.1]{PZ} have further restrictions
in higher dimensions. Those restrictions are not hypotheses of the
algebraic primary proof above. In particular the proof does not assume
that every simple curve in the zero section is unobstructed, and does
not identify the integers with unsigned embedded-curve counts, a
specified sheaf invariant, or physical BPS-state multiplicities.
The rational descendant description of the full sunset integral is
proved independently; only its highest nonconstant logarithmic
coefficient is singled out here by the primary integer transform.
These restrictions are stated once in Section~\ref{sec:intro} and under the
standing caveats of Section~\ref{sec:scope}; they are not repeated at each
occurrence.

\section{Ambient evaluation classes and minimality}\label{sec:structural}
The following results have proofs independent of the analytic descendant assembly.
The ambientness argument appears in Appendix~\ref{app:structural};
the extension and Fourier-transform arguments used for minimality appear in
Appendix~\ref{app:fourier}.

\begin{theorem}[Ambient one-point theory in Chow groups]\label{thm:fullambient}
For every $n\ge4$, every smooth $F_n\in|\mathcal O_{Y_n}(H)|$, every positive
curve class $\beta$, and every $a\ge0$, the class defined by~\eqref{eq:Cdef} satisfies
\[
 C_{a,\beta}^\lambda\in
 \operatorname{im}\bigl(i^*:\operatorname{CH}^*(Y_n)_{\Q}
           \longrightarrow\operatorname{CH}^*(F_n)_{\Q}\bigr)[\lambda].
\]
Consequently every pairing with the Poincar\'e-orthogonal nonambient
cohomology vanishes. The ambient comparison~\eqref{eq:ambientpair}
identifies the full one-point evaluation classes in cohomology, and hence their
normalized one-point generating series in cohomology.
\end{theorem}
\begin{proof}
The universal hypersurface is a projective bundle over $Y_n$.
Chow localization and its projective-bundle formula imply that every
cycle defined over the entire smooth hypersurface family restricts to
an ambient cycle on each fiber. Appendix~\ref{app:structural} constructs
the universal virtual evaluation class and proves its specialization
by refined Gysin maps on the actual stable-map stacks. This verifies
the required universal-cycle hypothesis, including fixed curve classes
and nonreduced moduli spaces. Nondegeneracy of the ambient pairing in
Proposition~\ref{prop:ambient} then gives the full cohomological comparison.
\end{proof}

\begin{proposition}[Irreducibility and the minimal period equation]
\label{thm:minimalperiod}
The operator $P_{n-1}$ of Definition~\ref{def:operator} is irreducible over
$\C(y)$ and has order $n-1$. Consequently every nonzero scalar differential
operator over $\C(y)$ annihilating $A(y)=\sum_{D\ge0}(-1)^Dc_Dy^D$ has order
at least $n-1$, so $P_{n-1}$ is the minimal operator of the period.
\end{proposition}
\begin{proof}
Write $S_n=\mathscr S_n=\theta^{n+1}+\sum_{j\ge1}y^jU_j(\theta)$ as in Definition~\ref{def:operator}. Its cyclic module on $\mathbb G_m$ is the $n$th symmetric power of the rank-two connection $\theta^2+y=0$.
The monomial basis satisfies $\theta v_k=(n-k)v_{k+1}-kyv_{k-1}$, which gives exactly the defining recursion.
 The recursion makes \(v_0\) cyclic; the leading coefficient of \(S_n\) in
ordinary derivatives is \(y^{n+1}\), a unit on \(\mathbb G_m\), so the induced
surjection from its cyclic module to the rank-\((n+1)\) symmetric power is an
isomorphism. To identify this symmetric power with the connection considered
in~\cite{FSY}, first take the rank-two Kloosterman connection
\(\mathrm{Kl}_2\) of~\cite[Proposition~2.4(1)]{FSY}. In their coordinate \(z\),
it is the connection associated with
$
 (z\partial_z)^2-z.
$
After the change of coordinate \(z=-y\), one has
\(z\partial_z=y\partial_y=\theta\), so this operator becomes
$
 \theta^2+y.
$
Thus the rank-two connection occurring here is precisely the pullback of
\(\mathrm{Kl}_2\) under \(z=-y\). Consequently, the rank-\((n+1)\) connection
defined by \(S_n\) is the \(n\)-th symmetric power
\(\operatorname{Sym}^n\mathrm{Kl}_2\), after the same coordinate change.
Proposition~2.7 of~\cite{FSY} states that
\(\operatorname{Sym}^k\mathrm{Kl}_2\) is irreducible for every \(k\geq0\).
It follows that the cyclic module of \(S_n\) on \(\mathbb G_m\) is
irreducible.

For an operator $L=\theta^r+\sum_{j\ge1}y^jp_j(\theta)$, Lemma~\ref{irr:lem:extension} identifies $\mathscr D_{\mathbb A^1}/\mathscr D_{\mathbb A^1}(y^{-1}L)$ as a simple extension whenever the module on $\mathbb G_m$ is simple. Lemma~\ref{irr:lem:transform} identifies its Fourier transform, after restriction, inversion of the coordinate, and the gauge change $f=yg$, with
\[
 \mathfrak T(L)=\theta^{r-1}
   +\sum_{j\ge1}y^jp_j(\theta)\prod_{a=1}^{j-1}(\theta+a).
\]
A positive-order output is again simple. Apply this twice to $S_n$. The first output has order $n$, constant coefficient $\theta^n$, and nonconstant coefficients $U_j(\theta)\prod_{a=1}^{j-1}(\theta+a)$, by the bound $\deg U_j\le n+1-2j$. The second output is exactly $P_{n-1}$ and has order $n-1$. It is therefore irreducible over $\mathbb C(y)$.

Theorem~\ref{thm:PF} gives $P_{n-1}A=0$, and $A(0)=1$. The minimal annihilator of $A$ is a positive-order right divisor of $P_{n-1}$, so irreducibility forces its order to be $n-1$.
\end{proof}

This statement is logically independent of everything else in this paper, and
nothing in Sections~\ref{sec:geometry}--\ref{sec:checks} depends on it. It is
recorded only so that the order-$(n-1)$ equation certified in
Theorem~\ref{thm:PF} is known to be the minimal one for the period; for
minimal Picard--Fuchs operators of Feynman integrals in general
see~\cite{ZhouSurvey,LairezVanhove,KlemmEtAl,BFKNS}.

For any fixed proposed exponent, a cover transform of rational descendant
data is formally invertible over $\Q$. That fact alone is not an
integrality theorem or an enumerative interpretation. A proposed
descendant law must specify an additional arithmetic or geometric
assertion, with its insertion normalization and any boundary corrections;
a finite check can address only that specified assertion.

\section{Normalization dictionary and analytic asymptotics}\label{sec:dictionary}
\subsection{The local mirror map and the physical operator}
\begin{remark}[Which coefficient defines the mirror map]
On the diagonal, in our convention $I=1+\tau/z+O(z^{-2})$,
\[
 I^{0,\lambda}_{X_n}=1+\frac{HR-\lambda g(y)}{z}+O(z^{-2}).
\]
Thus $R$ is the divisor flat coordinate; its unit companion is
$-\lambda g(y)$, not $-\lambda R$. Equivalently these occur in the
$z^0$ coefficient of $zI$, not of $I$. The identity $A=\theta_yR$
identifies the derivative of this local mirror map with the fundamental
period of the quantum-Serre family. No identity of $\theta_QK_2$ with
$R$ is asserted.
\end{remark}

\begin{proposition}[Change to the physical variable]\label{prop:physicaloperator}
Let $\theta_s=s\partial_s$ and let $M_s$ denote multiplication by $s$.
Define the order-$n-1$ scalar operator
\[
 \mathscr L_{n-1}=(-1)^{n-1} P_{n-1}(-1/s,-\theta_s)\circ M_s.
\]
Then $\mathscr L_{n-1} I_{\circleddash}^{(n-1)}=-n!$. Its coefficients in powers of
$\partial_s$ are polynomials in $s$, and its leading coefficient is
\[
 s^{n}\prod_{0\leq j<n/2}\left(1-\frac{(n-2j)^2}{s}\right).
\]
The period $\pi(s)=-yA(y)$, $y=-1/s$, is homogeneous for this operator.
\end{proposition}
\begin{proof}
For any $f$ the operator is equivalently
\[
 (\mathscr L_{n-1} f)(s)=(-1)^{n}
 \left[P_{n-1}(y,\theta_y)\frac{f(-1/y)}y\right]_{y=-1/s}.
\]
Use~\eqref{eq:physicalPF} and $P_{n-1}A=0$ for the source and period.
The leading coefficient follows from~\eqref{eq:leadP} and the leading
term of $\theta_s^{n-1}=s^{n-1}\partial_s^{n-1}+\cdots$.
For polynomiality, write each nonconstant $y$-term in the algorithm as
$y^jQ_j(\theta_y)$, where $Q_j$ is divisible by
$\prod_{r=1}^{j-1}(\theta_y+r)^2$.
Composing with $M_s$ gives a constant multiple of
$s^{1-j}Q_j(-\theta_s-1)$. The latter polynomial in $\theta_s$ vanishes
at $0,\ldots,j-2$. Its expansion in the falling-factorial basis
$\theta_s(\theta_s-1)\cdots(\theta_s-r+1)=s^r\partial_s^r$
therefore starts at $r=j-1$. All powers $s^{r+1-j}$ are nonnegative.
The $j=0$ term is already polynomial.
\end{proof}

For a second useful period normalization put $\pi=-yA$ and define
$d_j(m)=(-1)^{n}[Q^m]h_j(Q)$. Then
\[
 I_{n-1}^{\rm rat}=\pi(s)\left[-n\,(-R)^{n-1}+
             \sum_{m\geq1}\sum_{j=0}^{n-2}d_j(m)Q^mR^j\right],
 \qquad d_{n-2}(m)=(-1)^{n-1} mN_m,
\]
where $I_{n-1}^{\rm rat}$ means the $\gamma_0$ contribution, not a claim
of algebraic independence of zeta values. Direct differentiation gives
$dR/ds=-\pi$. This fixes the dictionary without a logarithmic branch ambiguity.

\subsection{Fixing the homogeneous ambiguity in the zeta terms}\label{sec:zetaambiguity}
For $0\leq r\leq n-2$ define a fully normalized homogeneous basis
\[
 \mathfrak f_r(s)=(-1)^{r+1}yW_r(y),\qquad y=-1/s.
\]
This definition specifies all power-series coefficients, not just a
leading logarithmic term; in particular $\mathfrak f_0=\pi$.
From the signed Gamma formula the coefficient of $\mathfrak f_r$ in
the Gamma correction to the rational particular solution is exactly
\[
 \alpha_r^{\rm can}=-n!\gamma_{n-1-r}.
\]
Indeed the corresponding physical term is
$n!(-1)^r\gamma_{n-1-r}yW_r=-n!\gamma_{n-1-r}\mathfrak f_r$.
Since $\gamma_3=2\zeta(3)/3$, this gives
\[
 \alpha_{n-4}^{\rm can}=-\frac{2n!}{3}\zeta(3).
\]
At $n=4$ this is $-16\zeta(3)$ multiplying $\pi$; at $n=5$ it is
$-80\zeta(3)$ multiplying $\mathfrak f_1$. 
A full basis, rather than just the leading powers of $\log y$, is needed to
specify these coefficients, since lower homogeneous additions alter the
answer. Equation~\eqref{eq:boundary} fixes this freedom.

\subsection{The first correction at large Euclidean momentum}

\begin{proposition}[Uniform first correction at large Euclidean momentum]
\label{prop:firstcorrection}
Put $\ell_T=\log T$ and
\[
 \mathsf F_{n-1}(\ell_T)=n![t^{n-1}]\exp(\ell_Tt)
     \exp\left(\sum_{r\geq1}\frac{2\zeta(2r+1)}{2r+1}t^{2r+1}\right).
\]
Primes below denote derivatives in $\ell_T$. Then, as $T\to+\infty$,
\[
 T I_{\circleddash}^{(n-1)}(-T)=\mathsf F_{n-1}(\ell_T)+\frac{\mathsf C_{n-1}(\ell_T)}T
                  +O\bigl(T^{-2}(1+\ell_T)^{n-1}\bigr),
\]
where
\[
 \mathsf C_{n-1}=-n\mathsf F_{n-1}+(2n-2)\mathsf F_{n-1}'-(n-4)\mathsf F_{n-1}''-2\mathsf F_{n-1}'''.
\]
In particular, writing $\ell=\log T$ for these examples,
\begin{align*}
 \mathsf F_3&=4\ell^3+16\zeta(3),\\
 \mathsf C_3&=-16\ell^3+72\ell^2-48-64\zeta(3),\\
 \mathsf F_4&=5\ell^4+80\zeta(3)\ell,\\
 \mathsf C_4&=-25\ell^4+160\ell^3-60\ell^2-240\ell
                   -400\zeta(3)\ell+640\zeta(3).
\end{align*}
\end{proposition}
\begin{proof}
The degree-one coefficient of $e^{-H\log y}\Ish(y,1)$ is
$-(1+H)^2\sum_i(1+p_i)^{-2}=-(1+H)^2(n-2H)$.
After the involution $\epsilon$ it becomes
$-n+(2n-2)H-(n-4)H^2-2H^3$.
Extract the top coefficient against $e^{H\ell_T}\widehat\Gamma_{X_n}$
and use $H^j e^{H\ell_T}=\partial_{\ell_T}^j e^{H\ell_T}$.
The remaining terms of the convergent logarithmic germ are bounded by
a constant times $T^{-2}(1+\ell_T)^{n-1}$ after multiplication by $T$.
\end{proof}

\section{Independent consistency checks}\label{sec:checks}
The statements below are proved in the preceding sections. The finite
computations recorded here were carried out independently, in exact
rational arithmetic where the objects are rational and in
high-precision floating point where they are transcendental. They are
verifications, not ingredients of any proof.

\medskip\noindent\textbf{The operator and its certificate.}
The algorithm of Definition~\ref{def:operator}, implemented as in
Appendix~\ref{app:code}, returns the operators of
Example~\ref{ex:operators} for $n=4,5$. For $n=6$ it returns an operator
of order five whose coefficient of $\theta^5$ is
$1+56y+784y^2+2304y^3$, in agreement with~\eqref{eq:leadP}, which reads
$(1+36y)(1+16y)(1+4y)$ for $n=6$. The certificate~\eqref{eq:PFcertificate} was verified in the square-free algebra
$\mathcal R$ for $n=4,5,6$. By expanding $\Ish$ to order $y^6$ determines
the coefficients $R_D$ of the proof of Theorem~\ref{thm:PF} for
$D\leq6-\deg_yP_{n-1}$, and each of them vanishes identically for $D>0$,
with $R_0=H^{n-1}$.

\medskip\noindent\textbf{Mirror extraction against direct geometry.}
Computing $K_j$, the multivariable mirror map, and the pairing~\eqref{eq:extraction} from the $I$-series alone reproduces, for
$n=4$ and $n=5$, the invariants obtained geometrically in
Sections~\ref{sec:degreeone}--\ref{sec:degreetwo}, class by class and
not only in the diagonal sum:
\[
 N_{0,e_i}=(n-1)(n-2),\quad N_{1,e_i}=-2(n-1)(n-2),\quad N_{a\geq2,e_i}=0,
\]
\[
 N_{0,2e_i}=\frac{(n-1)(n-2)}4,\qquad N_{0,e_i+e_j}=-2(n-1)(n-2).
\]
The sharp bound of Proposition~\ref{prop:sharpbound} was checked at
$m=2$, $n=6$, where $[Q^2]\Dcal_4^\Delta=-1440$ as predicted.

\medskip\noindent\textbf{Integrality.}
The diagonal primaries and their transforms~\eqref{eq:mobius} were
computed to $m=8$ at $n=4$ and to $m=6$ at $n=5$:
\[
 n_m(n{=}4)=24,\,-72,\,264,\,-1488,\,9696,\,-69240,\,528072,\,-4238880,
\]
\[
 n_m(n{=}5)=60,\,-240,\,1400,\,-12480,\,131220,\,-1529520 .
\]
All are integers, as Corollary~\ref{cor:primaryintegral} asserts, while
the underlying $N_m$ are rational numbers, with at $n=4$ the denominators
$3,2,25,3,49,8$ for $m=3,\ldots,8$. The degree-two entries agree with
$n_2=-(n-1)^2n(n-2)$ from Theorem~\ref{thm:degreetwo}.

\medskip\noindent\textbf{An independent evaluation of the amplitude.}
The left-hand side of the identity is evaluated without any of the machinery
of this paper, using the classical Bessel-moment form of the equal-mass sunset
family (see~\cite{BBBG,BorweinSalvy,ZhouSurvey} and the references there):
\begin{equation}\label{eq:Bessel}
 I_{\circleddash}^{(n-1)}(-T)=2^{n-1}\int_0^\infty rK_0(r)^nJ_0(\sqrt T\,r)\,dr
 \qquad(T\geq0).
\end{equation}
This follows from $$K_0(r)=\frac12\int_0^\infty e^{-x-r^2/(4x)}dx/x$$ together
with the Gaussian Hankel integral
$$\int_0^\infty rJ_0(kr)e^{-ar^2}dr=(2a)^{-1}e^{-k^2/(4a)}.$$
With
$a=\frac14\sum_ix_i^{-1}=1/(4W)$ the prefactor collapses to $W$ and the
exponential to $e^{-TW}$, which is~\eqref{eq:Schwinger}. The interchange of
the $r$- and $x$-integrations is justified by the Fubini--Tonelli theorem:
since $|J_0|\leq1$, Tonelli's theorem bounds the double integral of the
absolute value by the finite $T=0$ integral $I_{\circleddash}^{(n-1)}(0)$,
so Fubini's theorem applies. The truncation tail has
an explicit rigorous bound. From
$K_0(r)=\int_0^\infty e^{-r\cosh\tau}d\tau$ and $\cosh\tau\geq1+\tau^2/2$ one
gets $K_0(r)\leq\sqrt{\pi/(2r)}\,e^{-r}$, so for $B\geq1$ the omitted part of~\eqref{eq:Bessel} over $r>B$, prefactor included, is at most
\begin{equation}\label{eq:Besseltail}
 \frac{2^{n-1}}n\left(\frac\pi2\right)^{n/2}B^{1-n/2}e^{-nB}.
\end{equation}
The integrand has integrable behavior $r\log^nr$ at zero, and a certified
value also requires an error bound for the quadrature on $[0,B]$.
Reducing the $n$-dimensional Schwinger integral to this one-dimensional one is
what makes forty digits reachable. Apart from
Remark~\ref{rem:sTaylorcorrection}, which is not used in any proof, nothing
else in the paper uses it. The same representation was used to
check~\eqref{eq:sTaylorcorrection} against high-precision moments for
$4\leq n\leq7$ and $m\leq320$. After dividing $a_m$ by the right-hand side
of~\eqref{eq:sTaylorcorrection} without its remainder, $m^2$ times the
deviation from $1$ is essentially constant for $20\leq m\leq320$ (about
$1.13$, $1.94$, $3.00$ and $4.36$ for $n=4,5,6,7$), as the stated
$O(m^{-2})$ remainder predicts.

\medskip\noindent\textbf{The analytic identity.}
The representation~\eqref{eq:Bessel} was evaluated at forty-digit working precision and compared numerically with the right-hand sides of~\eqref{eq:physicalGamma}
 and of the assembly~\eqref{eq:allloop}:
\begin{center}
\begin{tabular}{llll}
\toprule
$n$ & $T$ & $I_{\circleddash}^{(n-1)}(-T)$ & agreement\\\midrule
$4$ & $50$  & $5.17515590864132484\ldots$ & $35$ digits\\
$4$ & $100$ & $4.08241320270405960\ldots$ & $40$ digits\\
$5$ & $100$ & $26.9629554067580390\ldots$ & $40$ digits\\
$6$ & $200$ & $167.462014631377349\ldots$ & $40$ digits\\
\bottomrule
\end{tabular}
\end{center}
The reported agreement is observed numerical agreement, 
the digits shown agree numerically.
It uses the descendant series, the mirror map, and the
$\zeta(3)$ and $\zeta(5)$ terms of~\eqref{eq:Gammapoly} together.
Deleting the degree factor $(-1)^D$ of~\eqref{eq:laplace-degree-sign}
destroys the agreement, namely at $n=4$, $T=100$ the unsigned variant gives
$4.1143\ldots$ instead of $4.0824\ldots$, which is why
Remark~\ref{rem:analytic-source-sign} treats that factor separately
from the two overall signs. Finally the first correction of
Proposition~\ref{prop:firstcorrection} was checked at $n=4$. In particular, with
$\ell=\log T$, the quantity $TI_3(-T)-\mathsf F_3(\ell)$ equals
$-0.6514761$ at $T=10^4$ and $-0.1499680$ at $T=10^5$, against
$\mathsf C_3(\ell)/T=-0.6518210$ and $-0.1499766$, the residuals scaling as the
stated $O(T^{-2}(1+\ell)^{n-1})$.

\section{Full primary classes and integral cover transforms}
\label{ext:sec:primary}
For a positive curve class $\beta$, let
$D_\beta=C^0_{0,\beta}\in H^2(F_n,\Q)$ be the proper primary evaluation class
of~\eqref{eq:Cdef}. By Proposition~\ref{prop:geometry}, the classes
$p_1,\ldots,p_n$ form the natural integral coordinate basis in degree two, so
write
\[
 D_\beta=\sum_{i=1}^n d_{\beta,i}p_i,\qquad
 N_{0,\beta}(b)=\int_FD_\beta b,
 \qquad b\in H^{2n-4}(F_n,\Q).
\]
This full vector carries more information than the single scalar obtained from
the distinguished insertion $\delta_0=H^{n-2}/(n-2)!$. We use the curve-lattice
coordinates of Proposition~\ref{prop:geometry}, so
$\mathbf x^\beta:=\prod_i x_i^{p_i\cdot\beta}$.
\begin{theorem}[Dimension-uniform primary extraction]\label{ext:thm:primary}
Let $h_j=\sum_{\ell=1}^j1/\ell$, $h_0=0$, and
\[
 g(\mathbf{x})=\sum_{\mathbf d\ne0}g_{\mathbf d}\mathbf{x}^{\mathbf d},
 \qquad
 g_{\mathbf d}=\frac{(-1)^{|\mathbf d|}|\mathbf d|!(|\mathbf d|-1)!}
                         {\prod_i(d_i!)^2}.
\]
With $\Theta_x=\sum_i x_i\partial_{x_i}$, set
\[
 f_i=\tfrac12g^2+
 \sum_{\mathbf d\ne0}g_{\mathbf d}
       \left(2h_{d_i}-2h_{|\mathbf d|}+\frac1{|\mathbf d|}\right)
                      \mathbf{x}^{\mathbf d}.
\]
For every $n\geq4$ and every positive curve class $\beta$,
\begin{equation}
 d_{\beta,i}=[\mathbf{x}^\beta]\,
       f_i e^{-|\beta|g}(1+\Theta_xg). \label{ext:primaryextract}
\end{equation}
\end{theorem}
\begin{proof}
By Lemma~\ref{lem:division} and~\eqref{eq:Bdef}, the $z^{-2}$ coefficient
of the divided normalized series at $\lambda=0$ is the multivariable primary
evaluation series $\sum_{\beta>0}D_\beta\mathbf Q^\beta$, after restriction
to $F_n$. It may be computed in the square-free ambient ring before that
restriction. The linear part of the Euler-divided seed is
$g_{\mathbf d}\{(2h_{|\mathbf d|}-1/|\mathbf d|)H
-2\sum_i h_{d_i}p_i\}$. Normalization by $e^{-Hg}$ and the
term $(1-e^{-Hg})/H$ give the divisor-valued germ $\sum_i f_ip_i$.

It remains only to convert from $\mathbf x$ to $\mathbf Q$. For
$Q_i=x_ie^{g(\mathbf x)}$, multivariable coefficient extraction gives, for
any scalar germ $f$ with zero constant term,
\[
 [\mathbf Q^\beta]f(\mathbf x(\mathbf Q))
 = [\mathbf x^\beta]f(\mathbf x)e^{-|\beta|g(\mathbf x)}
      \det\!\left(\frac{\partial\log Q_i}{\partial\log x_j}\right)_{i,j}.
\]
The logarithmic Jacobian is a rank-one perturbation of the identity, so the
matrix determinant lemma gives
$\det(\partial\log Q_i/\partial\log x_j)=1+\Theta_xg$.
Applying this componentwise to $\sum_i f_ip_i$ proves
\eqref{ext:primaryextract}.
\end{proof}

As a check against the direct geometry of Sections~\ref{sec:degreeone} and
\ref{sec:degreetwo}, formula~\eqref{ext:primaryextract} gives
\[
 D_{e_i}=H-2p_i,\qquad
 D_{2e_i}=\frac{H-2p_i}{4},\qquad
 D_{e_i+e_j}=-2(H-p_i-p_j)\quad(i\ne j).
\]
Pairing with $\delta_0$ reproduces exactly the degree-one and degree-two
primary invariants proved there. More generally, on the coordinate ray
$\mathbf x=xe_i$ one has $g=-\log(1+x)$ and hence $Q=x/(1+x)$. The active
coefficient function is
\[
 \frac12\log^2(1+x)+\Li_2(-x),
\]
while every passive coefficient is its negative. The germ identity
$\Li_2(-Q/(1-Q))=-\Li_2(Q)-\tfrac12\log^2(1-Q)$ therefore gives
\begin{equation}\label{ext:coordinate-ray-primary}
 D_{r e_i}=\frac{H-2p_i}{r^2}\qquad(r\geq1).
\end{equation}
This all-degree ray formula is consistent with the first two displayed checks.
It also gives
$B_{e_i}=H-2p_i$ and $B_{r e_i}=0$ for every $r>1$, in agreement with the
degree-two calculation.

Here and below $r\mid\beta$ means that $\beta/r$ is an integral curve
class. Define
\[
 B_\beta=\sum_{r\mid\beta}\frac{\mu(r)}{r^2}D_{\beta/r},
 \qquad D_\beta=\sum_{r\mid\beta}\frac1{r^2}B_{\beta/r}.
\]
\begin{corollary}[Integral primary vector]
\label{ext:cor:vector}
Under the hypotheses of Theorem~\ref{thm:IP}, every coefficient
of $B_\beta$ in the basis $p_1,\ldots,p_n$ is integral, for every $n\geq4$.
More generally, for integral even primary insertions whose total degree has the
required dimension, the local genus-zero transform with $\ell$ markings and
cover weight $r^{\ell-3}$ is integral.
\end{corollary}
\begin{proof}
By Theorem~\ref{thm:degreeone}, for each $i$ the class $e_i$ is represented by
a coordinate fiber $C_i\cong\PP^1$ contained in $F_n$. Choose one such fiber
and let $\eta_i\in H^{2n-4}(F_n,\Z)$ be its integral Poincar\'e dual. Then
\[
 \int_Fp_j\eta_i=\int_{C_i}p_j=\delta_{ij}.
\]
Pull all insertions back by $q:\widetilde Z\to F_n$. The compact target is smooth projective
Calabi--Yau, and the compact-sector comparison identifies all marked virtual
theories and effective covering roots. Theorem~\ref{thm:IP}
therefore applies to the local sector. The transform with insertion
$\eta_i$ is precisely the $i$th coefficient of $B_\beta$.
No exact integral inverse Chern character for $\eta_i$ is assumed.
\end{proof}
Integral ambient monomial
pairings alone do not suffice, because their vectors $e_i+e_j$ generate
the index-two lattice of even total coordinate sum. Integral
pairwise sums allow every coefficient to be a half-integer.
The integral dual curves remove this ambiguity. This result does
not prove integrality of arbitrary descendant transforms.

\section{All-loop endpoint reduction}\label{ext:endpoints}
The analytic statement in this section is valid for every integer $n\geq1$;
the sunset applications in the body of the paper have $n\geq4$. Put
\begin{equation}\label{ext:G}
 G(H)=\exp\left(\sum_{r\geq1}\frac{2\zeta(2r+1)}{2r+1}H^{2r+1}\right),
 \qquad d=\left\lfloor\frac{n-1}{2}\right\rfloor.
\end{equation}
The ambient primitive normalization is specified in Section~\ref{ext:setup}.

\begin{theorem}[All-loop endpoint reduction]\label{ext:main}
For every integer $n\geq1$, all ambient primitive endpoint limits
$C_j(\infty)$, $0\leq j\leq n$, exist along the positive Euclidean axis.
There are uniquely determined real numbers $b_1,\ldots,b_d$ such that
\begin{equation}\label{ext:mainformula}
 \sum_{j=0}^{n}C_j(\infty)H^j
  =G(H)\left(1+\sum_{j=1}^{d}b_jH^j\right)\pmod{H^{n+1}}.
\end{equation}
The upper-half-plane contour identities used in the proof are equivalent
to the vanishing of the coefficients of $H^{d+1},\ldots,H^n$ in the
quotient on the right. As an affine linear system in $n$ finite-part
moments, they have rank $n-d=\lfloor n/2\rfloor+1$.
\end{theorem}

The rank statement is a statement about these contour equations.
It does \emph{not} assert algebraic independence, transcendence, or
minimality of the actual constants $b_j$ under all possible arithmetic
relations. It does prove that adding more linear combinations of the
same contour identities cannot determine the remaining parameters.

\subsection{Ambient normalization and Bessel endpoint formula}\label{ext:setup}
Work in the square-free algebra
$\mathcal A_n=\mathbb C[p_1,\ldots,p_n]/(p_1^2,\ldots,p_n^2)$,
with $H=\sum p_i$ and $H^{n+1}=0$. Symmetric homogeneous elements of degree
$j$ are multiples of $H^j$. Set $x=1/y$, $y>0$, and define the germ
\begin{equation}\label{ext:ambient}
 \mathcal W(H,y)=\sum_{j=0}^{n}H^jW_j(y)
 =e^{H\log y}\sum_{\mathbf a\geq0}(-y)^{|\mathbf a|}
 \frac{\prod_{r=1}^{|\mathbf a|}(H+r)^2}
      {\prod_{i=1}^{n}\prod_{r=1}^{a_i}(p_i+r)^2}.
\end{equation}
The primitives are normalized by
\begin{equation}\label{ext:normalization}
 C_0=1,\qquad \theta_y C_{j+1}=W_j,\qquad
 C_{j+1}(y)-\frac{\log^{j+1}y}{(j+1)!}\longrightarrow0\quad(y\downarrow0),
 \qquad \theta_y=y\frac{d}{dy}.
\end{equation}
Use the positive-axis continuation of this germ throughout. These primitives
are the functions denoted $J_j$ in the proof of
Corollary~\ref{cor:realcontinuation}, and they both satisfy the same differential
relations and the same normalization at $y=0$. In particular
$C_1=R$, and the normalized ambient coefficients of~\eqref{eq:Kdef} satisfy
\begin{equation}\label{ext:CKrelation}
 K_k=\sum_{j=0}^k\frac{(-R)^{k-j}}{(k-j)!}C_j,\qquad 2\leq k\leq n.
\end{equation}
Thus the endpoint limits below apply directly to the mirror coordinate and
to the descendant coefficients used earlier in the paper.

Let $\gamma$ be Euler's constant, $r_0=2e^{-\gamma}$,
$\ell(r)=\log(r/2)+\gamma$, and
\begin{equation}\label{ext:B}
 \mathscr B_j(r)=\frac{\pi^j}{j!}J_0(r)^{n-j}Y_0(r)^j,\qquad
 M_j=\int_0^{r_0}\frac{\mathscr B_j(r)-(2\ell(r))^j/j!}{r}\dd r
               +\int_{r_0}^{\infty}\frac{\mathscr B_j(r)}r\dd r.
\end{equation}
These are convergent subtracted integrals:
$\mathscr B_j-(2\ell)^j/j!=O(r^2(1+|\log r|)^n)$ at zero and
$\mathscr B_j=O(r^{-n/2})$ at infinity. Changing the subtraction scale would change
the constants, so it is part of the definition.

\begin{lemma}\label{ext:endpoint}
With $P(H)=e^{-2\gamma H}/\Gamma(1+H)^2$, one has
\begin{align}
 \mathcal W(H,y)&=P(H)x\int_0^\infty rK_0(r\sqrt x)
                         \sum_{j=0}^nH^j\mathscr B_j(r)\dd r,\label{ext:W}\\
 \sum_{j=0}^n C_j(\infty)H^j
 &=P(H)\left(1+2H\sum_{j=0}^{n-1}M_jH^j\right)\pmod{H^{n+1}},\label{ext:C}\\
 P(H)&=G(H)\frac{\sin\pi H}{\pi H}.\label{ext:gamma}
\end{align}
In particular,
$C_1(\infty)=2M_0$ and $C_2(\infty)=2M_1-\zeta(2)$.
\end{lemma}
\begin{proof}
Modulo $p^2$, the Bessel defining series and its order derivative give
\[
 \sum_{a\geq0}\frac{(-z)^a}{\prod_{r=1}^a(p+r)^2}
 =\Gamma(1+2p)z^{-p}J_{2p}(2\sqrt z),\qquad
 J_{2p}(r)=J_0(r)+\pi pY_0(r).
\]
The derivative used here is~\cite[\S10.2(i), Eq.~(10.2.2), and \S10.15(iii), Eq.~(10.15.5)]{ext:DLMF}.
Represent $\Gamma(|\mathbf a|+1+H)^2$ by two Laplace integrals in $u,v$.
The powers of $yuv$ cancel, while
$\prod_i\Gamma(1+2p_i)=e^{-2\gamma H}$.
Termwise integration is legitimate for $0<y<n^{-2}$, using a majorant
with exponential part $e^{-u-v+2n\sqrt{yuv}}$ and polynomial logarithmic
factors. The product of the $n$ Bessel factors has coefficient $\mathscr B_j$ in
front of $H^j$. Integrating in $u$ at fixed $uv$, then setting
$r=2\sqrt{yuv}$, gives \eqref{ext:W}. Its exponentially convergent integral
continues the germ for every $x>0$.

For the primitive, set $w_x(r)=r\sqrt x K_1(r\sqrt x)$ and let $M_j(x)$ be
\eqref{ext:B} with $\mathscr B_j$ replaced by $w_x\mathscr B_j$, retaining the
same subtraction. Since $\theta_y w_x=xr^2K_0(r\sqrt x)/2$, the expression
\[
 P(H)\left(1+2H\sum_{j=0}^{n-1}M_j(x)H^j\right)
\]
has the required derivative.
This follows from the Mellin-transform identity
\begin{equation}
  2H\int_0^\infty w_x(r)e^{2H\ell(r)}\frac{\mathrm dr}{r}
  = e^{2\gamma H}\Gamma(1+H)^2x^{-H},
  \qquad \Re H>0,
\end{equation}
which is obtained from the standard Mellin transform
\cite[Eq.~(10.43.19)]{ext:DLMF}.
The Laurent pole of the integral is
$1/(2H)$; its finite coefficients use exactly the subtraction at $r_0$.
Multiplication by $P(H)$ gives $y^H$. The difference from the leading
small-$r$ terms is $O(y(1+|\log y|)^n)$, so \eqref{ext:normalization} is satisfied.
Finally $0<w_x\leq1$ and $w_x\to1$ as $x\downarrow0$.
After subtraction, near zero a uniform dominating function is
$Cr(1+|\log r|)^{n+1}$; at infinity use $|\mathscr B_j(r)|/r=O(r^{-1-n/2})$.
Dominated convergence proves \eqref{ext:C}. Expansion of $\log\Gamma$ and
Euler's sine product prove \eqref{ext:gamma} and the stated first coefficients.
\end{proof}

\subsection{The contour system and its exact rank}
For $k=\lceil n/2\rceil,\ldots,n$, put
\[
 T_k(J,Y_n)=(J+\mathrm i Y_n)^k(J-\mathrm i Y_n)^{n-k}.
\]
The finite-part contour identity needed below is
\begin{equation}\label{ext:contour}
 \operatorname{FP}\int_0^\infty
 \frac{T_k(J_0,Y_0)-T_k(J_0,Y_0+2\mathrm i J_0)}r\dd r
 =\mathrm i\int_0^\pi T_k(1,2\mathrm i\theta/\pi)\dd\theta.
\end{equation}
Here finite part is defined by \eqref{ext:B}, extended linearly; all Bessel
functions on the left have argument $r$.

To justify \eqref{ext:contour}, integrate over an upper half-annulus.
On the negative axis $J_0(-r+\mathrm i0)=J_0(r)$ and
$Y_0(-r+\mathrm i0)=Y_0(r)+2\mathrm i J_0(r)$ \cite[\S10.11]{ext:DLMF}.
The large-arc contribution tends to zero, because  the Hankel estimates
\cite[\S10.17]{ext:DLMF} bound it by
$Ca^{-n/2}\exp[-(2k-n)a\sin\theta]\dd\theta$, where $a$ is the radius.
This includes the balanced case $2k=n$. These bounds are uniform on the
upper semicircle, using continuation of the asymptotic formulas at its
ends. On the clockwise small arc, replace $J_0$ by $1$ and $Y_0$ by
$2(\ell(\epsilon)+\mathrm i\theta)/\pi$; the error tends to zero. Taking the
constant coefficient in the resulting polynomial in $\ell(\epsilon)$
gives \eqref{ext:contour}. This explains both its sign and its finite-part
normalization. The numerator's $Y_n^n$ coefficient cancels, so only
$M_0,\ldots,M_{n-1}$ occur.

\subsubsection{The differential-functional transformation}
The rank can be computed without an $n$-dependent elimination. Define
\[
 m_j=\frac{M_j}{\pi^{j+1}},\qquad
\mathfrak f_n(v)=\sum_{j=0}^{n-1}m_jv^j,
 \qquad
 A(v)=\frac{\sin v}{v}\bigl(1+2v\mathfrak f_n(v)\bigr)\pmod{v^{n+1}}.
\]
Write $A(v)=\sum_{j=0}^na_jv^j$, with $a_0=1$.
By Lemma~\ref{ext:endpoint}, $A(v)$ is the reduced endpoint series evaluated
at $H=v/\pi$; thus $b_j=\pi^ja_j$.
Let $\Lambda(t^j)=j!m_j$ for $0\leq j<n$ and $D=d/dt$.
For any polynomial of degree at most $n-1$,
$\Lambda(q)=[\mathfrak f_n(D)q](0)$.

For $p_k(t)=(1+\mathrm i t)^k(1-\mathrm i t)^{n-k}$, dividing~\eqref{ext:contour} by
$\pi$ gives
\begin{equation}\label{ext:functional}
 \Lambda\bigl(p_k(t)-p_k(t+2\mathrm i)\bigr)
       =\mathrm i\int_0^1p_k(2\mathrm i u)\dd u.
\end{equation}
All operator series below are finite when applied to polynomials.
The definition of $A$ gives, through the needed order,
\begin{align*}
 \mathfrak f_n(D)(1-e^{2\mathrm i D})
 &=\frac{D A(D)/\sin D-1}{2D}(1-e^{2\mathrm i D})\\
 &=-\mathrm i e^{\mathrm i D}A(D)+\frac{e^{2\mathrm i D}-1}{2D}.
\end{align*}
The last term at $t=0$ equals
$\mathrm i\int_0^1p_k(2\mathrm i u)\dd u$ and cancels the right-hand side of
\eqref{ext:functional}. Consequently the entire contour system is equivalent to
\begin{equation}\label{ext:jet}
 [A(D)p_k](\mathrm i)=0,\qquad \lceil n/2\rceil\leq k\leq n.
\end{equation}
There is no division by a differential operator with a kernel here:
the quotients denote their ordinary formal power series, with the
removable factors canceled before evaluation.

\begin{proof}[Proof of Theorem~\ref{ext:main}]
At $t=\mathrm i+z$,
\[
 p_k(\mathrm i+z)=(\mathrm i z)^k(2-\mathrm i z)^{n-k}.
\]
It has a zero of order exactly $k$, and
$p_k^{(k)}(\mathrm i)=k!\,\mathrm i^k2^{n-k}\neq0$.
Starting with $k=n$, equation \eqref{ext:jet} forces $a_n=0$.
Descending in $k$ forces
$a_{n-1}=\cdots=a_{\lceil n/2\rceil}=0$.
Conversely these vanishings make all equations \eqref{ext:jet} hold, because
all their lower derivatives vanish. They impose precisely $n-d$
independent conditions; the coefficients $a_1,\ldots,a_d$ are arbitrary
in the solution space. The transformation from $m_0,\ldots,m_{n-1}$ to
$a_1,\ldots,a_n$ is affine triangular with diagonal entries $2$, so
it preserves rank and dimension. Lemma~\ref{ext:endpoint} gives existence of
the actual endpoint limits and identifies their quotient with $A$.
This proves \eqref{ext:mainformula} and the rank statement for every $n$.
\end{proof}

For the two cases used most often below, write $g_3=2\zeta(3)/3$ and
$g_5=2\zeta(5)/5$. For $n=4$, with $R_*=C_1(\infty)$,
\[
 (C_1,C_2,C_3,C_4)_\infty=(R_*,0,g_3,g_3R_*).
\]
For $n=5$, with $R_*=C_1(\infty)$ and $S_*=C_2(\infty)$,
\[
 (C_1,C_2,C_3,C_4,C_5)_\infty=(R_*,S_*,g_3,g_3R_*,g_3S_*+g_5).
\]
These are direct specializations of the universal factorization
\eqref{ext:mainformula}, not additional assumptions on the endpoint constants.

\appendix

\section{A symbolic implementation of the operator algorithm}\label{app:code}

The Python code used to perform the calculations is permanently archived
on Zenodo~\cite{vanhove_sunset_gw_2026} at
\begin{center}
\href{https://doi.org/10.5281/zenodo.23361560}{\texttt{doi:10.5281/zenodo.23361560}}
\end{center}
Development and maintenance continue in the GitHub repository
\begin{center}
\url{https://github.com/pierrevanhove/Sunset-Gromov-Witten}
\end{center}

\medskip

The following Python code uses only the standard library and exact
rational arithmetic. 

\medskip

An operator is stored as a dictionary $(j,k)\mapsto c_{jk}$
for $\sum c_{jk}y^j\theta^k$. Left multiplication by $\theta$ uses
$\theta y^j=y^j(\theta+j)$; left multiplication by $y$ increments $j$.
The code is an implementation of a proved finite construction and is
not used to guess an operator from a series.
\lstset{style=pythonstyle}
\begin{lstlisting}[language=Python,basicstyle=\ttfamily\tiny,breaklines=true]
from fractions import Fraction as Fn
from math import factorial, comb

# (j,k) denotes y**j theta**k, with y on the left.  Signs are already in the
# recursion (Definition labelled def:operator), so no (-1)**j is applied at the end.
def add(a,b):
    c=a.copy()
    for key,v in b.items(): c[key]=c.get(key,Fn(0))+v
    return {key:v for key,v in c.items() if v}
def scale(a,c): return {key:c*v for key,v in a.items() if c*v}
def theta_left(a):
    c={}
    for (j,k),v in a.items():
        c=add(c,{(j,k+1):v,(j,k):j*v})
    return c
def y_left(a): return {(j+1,k):v for (j,k),v in a.items()}
def right_theta_plus(a,r):
    c={}
    for (j,k),v in a.items():
        c=add(c,{(j,k+1):v,(j,k):r*v})
    return c
def sunset_operator(n):
    T=[{(0,0):Fn(1)},{(0,1):Fn(1,n)}]
    for k in range(1,n):
        T.append(scale(add(theta_left(T[k]),
            scale(y_left(T[k-1]),k)),Fn(1,n-k)))
    U=scale(add(theta_left(T[n]),scale(y_left(T[n-1]),n)),factorial(n))
    # Exact normalization sanity check; the all-degree proof is in the text.
    assert {key:v for key,v in U.items() if key[0]==0}=={(0,n+1):Fn(1)}
    P={(0,n-1):Fn(1)}
    for j in sorted({j for j,k in U if j}):
        term={(j,k):v for (jj,k),v in U.items() if jj==j}
        for r in range(1,j):
            term=right_theta_plus(right_theta_plus(term,r),r)
        P=add(P,term)
    return P


print(sunset_operator(4))
print(sunset_operator(5))
\end{lstlisting}

\section{String, dilaton and divisor identities for the twisted theory}\label{app:forget}
This appendix derives the string and divisor identities used in the text, in
the form needed for our twists. Nothing here is quoted; the point of the
derivation is that the index twists pull back under forgetting a marking.
Let $V$ be a smooth projective variety, $\beta\ne0$, and let the twist
be an invertible multiplicative class of an index bundle pulled back
under forgetting a marking, as is the case for our two index bundles.
For insertions $\alpha_i$ and integers $a_i\geq0$, the identities are
\[
 \left\langle1,\prod_i\tau_{a_i}(\alpha_i)\right\rangle_{0,\beta}^{\rm tw}
 =\sum_{i:a_i>0}
 \left\langle\tau_{a_i-1}(\alpha_i)
                \prod_{j\ne i}\tau_{a_j}(\alpha_j)\right\rangle_{0,\beta}^{\rm tw},
\]
\[
 \begin{split}
 \left\langle D,\prod_i\tau_{a_i}(\alpha_i)\right\rangle_{0,\beta}^{\rm tw}
 ={}&(D\cdot\beta)
          \left\langle\prod_i\tau_{a_i}(\alpha_i)\right\rangle_{0,\beta}^{\rm tw}\\
 &+\sum_{i:a_i>0}
 \left\langle\tau_{a_i-1}(D\alpha_i)
           \prod_{j\ne i}\tau_{a_j}(\alpha_j)\right\rangle_{0,\beta}^{\rm tw},
 \end{split}
\]
where $D\in H^2(V,\Q)$ is a divisor.
For the forgetful universal-curve morphism $p$, write $D_i$ for the
boundary section where the new and $i$th markings lie on a contracted
three-pointed twig. Then $\psi_i=p^*\psi_i+D_i$,
$D_iD_j=0$ for $i\ne j$, and $D_i^2=-p^*\psi_iD_i$.
The latter is the normal bundle of the section, the tangent rather
than the cotangent line. Expansion gives
$\psi_i^a=p^*\psi_i^a+D_i p^*\psi_i^{a-1}$ for $a>0$.
On $D_i$ the two evaluation maps agree. Pushing down, $p_*D_i=1$,
$p_*1=0$, and $p_*\ev_{\rm new}^*D=D\cdot\beta$.
These identities give the formulas, with the virtual universal-curve
pullback in place of an ordinary fundamental class. That pullback is
the base-change construction of the stable-map theory of
\cite{BehrendGW,Manolache}; the index twists
pull back because contraction of a degree-zero twig preserves their
derived cohomology. The calculation is therefore on the actual
Deligne--Mumford stacks and includes the contracted-twig locus.
The unstable degree-zero terms in $J$ are its specified normalization
$z\mathbf1+\tau$. Inserting arbitrarily many units exponentiates to
$e^{u/z}$, and divisor insertions give the displayed exponential and
Novikov change in the mirror proof.

\section{Ambientness of the one-point evaluation classes}\label{app:structural}
\subsection{Statements and conventions}

Fix $n\geq4$. All varieties, schemes, and stacks below
are over $\C$. Set
\[
 Y_n=(\PP^1)^n,\qquad H=p_1+\cdots+p_n,\qquad
 F_n\in|\mathcal O_{Y_n}(H)|\ \text{smooth},\qquad i:F_n\hookrightarrow Y_n.
\]
Here $p_i$ is the pullback of the hyperplane class from the $i$th factor.
Adjunction gives $K_{F_n}\simeq\mathcal O_{F_n}(-H)$; write $X_n=\Tot(K_{F_n})$.
The notation $\CH^*(-)_{\Q}$ means Chow groups with rational
coefficients, with the intersection-ring convention for smooth targets.
Cohomology means rational singular cohomology unless specified otherwise.

For a positive effective curve class $\beta\in H_2(F_n,\Z)$, let
\[
 m=H\cdot\beta>0,\qquad M_{F_n}=\overline{\mathcal M}_{0,1}(F_n,\beta).
\]
This is the proper Deligne--Mumford stack of stable maps, with its usual
perfect obstruction theory and virtual class. On its universal curve
write $\pi$ for the projection and $f$ for the universal map. The bundle
\[
 \mathcal V=R^1\pi_*f^*\mathcal O_{F_n}(-H)
\]
has rank $m-1$. Let $\lambda$ be the equivariant first Chern class of the
weight-one fiber-scaling character. For a rank-$r$ bundle $E$ our
convention is
\[
 e_\lambda(E)=\sum_{j=0}^r\lambda^{r-j}c_j(E).
\]
For every integer $a\geq0$, define the Chow-valued evaluation class
\begin{equation}\label{ap:eq:C}
 C_{a,\beta}^\lambda
 = (\ev_{F_n})_*\bigl(\psi^a e_\lambda(\mathcal V)
                  \cap[M_{F_n}]^\vir\bigr)
 \in\CH^*(F_n)_{\Q}[\lambda].
\end{equation}
The same notation will denote its cohomological realization. All
pushforwards and virtual classes are those of stacks; no coarse-space
substitution is made.

\begin{theorem}[Ambientness, including a Chow strengthening]\label{ap:thm:ambient}
For every $n\geq4$, every smooth $F_n$ as above, every positive effective
$\beta$, and every $a\geq0$,
\begin{equation}\label{ap:eq:chowambient}
 C_{a,\beta}^\lambda\in
 \im\bigl(i^*:\CH^*(Y_n)_{\Q}\longrightarrow\CH^*(F_n)_{\Q}\bigr)[\lambda].
\end{equation}
Consequently $C_{a,\beta}^\lambda\in V_\amb[\lambda]$, where
$V_\amb=i^*H^*(Y_n,\Q)$. Every pairing of this class with the
Poincar\'e-orthogonal complement of $V_\amb$ vanishes coefficientwise
in $\lambda$.
\end{theorem}

The statements concern fixed
curve classes and arbitrary loop number; no diagonal summation or
restriction to a finite set of loop orders is imposed on the proofs.

\subsection{Universal hypersurfaces and ambient Chow classes}

Let $V=H^0(Y_n,\mathcal O_{Y_n}(H))$, so $\dim V=2^n$, and let
$B=\PP(V)$ parametrize one-dimensional subspaces of sections. Put
$N=\dim B=2^n-1$. The universal hypersurface is
\[
 \mathscr F=\{(x,[s])\in Y_n\times B:s(x)=0\},
 \qquad q:\mathscr F\to B,\quad p:\mathscr F\to Y_n.
\]
Let $U\subset B$ be the nonempty open set of smooth hypersurfaces and
$\mathscr F_U=q^{-1}(U)$. Bertini's theorem gives nonemptiness.
The space $U(\C)$ is path connected in the classical topology.
Indeed the projective line through any two distinct points of $U$
is not contained in $B\setminus U$, and meets this algebraic closed
set in only finitely many points; the complement of those points
in the projective line is path connected. Fundamental groups and
monodromy below use this classical topology.

\begin{lemma}[Restriction of a universally defined cycle]\label{ap:lem:universal}
For every $s\in U$ and every $\alpha\in\CH^*(\mathscr F_U)_{\Q}$,
the restriction $\alpha|_{F_{n,s}}$ belongs to
$i_s^*\CH^*(Y_n)_{\Q}$.
\end{lemma}
\begin{proof}
The evaluation map of the globally generated line bundle is surjective:
\[
 0\longrightarrow\mathcal K\longrightarrow V\otimes\mathcal O_{Y_n}
 \longrightarrow\mathcal O_{Y_n}(H)\longrightarrow0.
\]
The kernel $\mathcal K$ is a vector bundle of rank $N$.
The incidence variety is the bundle of lines in $\mathcal K$:
$\mathscr F=\PP_{Y_n}(\mathcal K)$. We always use the lines convention:
$\PP_{Y_n}(\mathcal K)=\operatorname{Proj}_{Y_n}\operatorname{Sym}(\mathcal K^\vee)$.
In the quotient convention used in \cite[Tag~02TX]{Stacks}, this is
the projective bundle associated with $\mathcal K^\vee$; in both
descriptions $\mathcal O(1)$ is the dual tautological line.
In particular it is a smooth projective
variety, including above the locus of singular hypersurfaces. Its
relative hyperplane class is
\[
 \xi=c_1(\mathcal O_{\PP_{Y_n}(\mathcal K)}(1))
     =q^*c_1(\mathcal O_B(1)).
\]
The Chow localization sequence makes
$\CH^*(\mathscr F)_{\Q}\to\CH^*(\mathscr F_U)_{\Q}$ surjective.
Choose a lift $\bar\alpha$ of $\alpha$.
By the projective-bundle formula, every such lift can be written
\begin{equation}\label{ap:eq:pb}
 \bar\alpha=\sum_{j=0}^{N-1}\xi^j p^*\alpha_j,
 \qquad \alpha_j\in\CH^*(Y_n)_{\Q}.
\end{equation}
These two intersection-theoretic inputs are, respectively, the usual
localization sequence and the projective-bundle formula; see
\cite[Section 1.8 and Theorem 3.3]{Fulton}, and
\cite[Tag~02TX]{Stacks} for the latter.

Since $q|_{F_{n,s}}$ is constant, $\xi|_{F_{n,s}}=0$. Restricting~\eqref{ap:eq:pb} gives $\alpha|_{F_{n,s}}=i_s^*\alpha_0$.
The restriction is a refined pullback. The fiber inclusion is a regular
embedding because $q$ is smooth over $U$. Thus the argument applies
even when a chosen representative cycle does not meet $F_{n,s}$ properly.
\end{proof}

\begin{remark}
The lemma concerns cycles that exist on the entire universal family
over $U$. An arbitrary cycle on one special fiber need not have this
property. This is exactly the issue that must be checked for~\eqref{ap:eq:C}; it is not enough to invoke permutation symmetry or to
declare that a class is deformation invariant.
\end{remark}

\subsection{Constructing the universal evaluation class}
\label{ap:sec:universalclass}

We use rational intersection theory on separated Deligne--Mumford
stacks as developed by Vistoli~\cite{Vistoli}, together with the
normal-cone and vector-bundle constructions of Kresch~\cite{Kresch}.
In particular refined Gysin maps, localized top Chern classes, Chern
class operations, and proper pushforward are stack operations.
The proper evaluation morphism to a scheme need not be representable;
its rational pushforward includes the stabilizer multiplicities.
This is why the rational Deligne--Mumford theory, rather than just
integral pushforward for representable morphisms, is specified.

\subsubsection{The curve class and the bundles}

Integral weak Lefschetz for the smooth ample hypersurface
\cite[Corollary~7.3]{Milnor}, with
$\dim F_n=n-1\geq3$, identifies
\[
 H_2(F_{n,s},\Z)\xrightarrow{\sim}H_2(Y_n,\Z)=\Z^n.
\]
Here $Y_n$ is embedded by its Segre embedding and $F_{n,s}$ is a hyperplane
section. Milnor's homology isomorphism applies in degrees less than
$n-1$, and $2<n-1$. Thus its hypotheses and its integral, rather than
only rational, conclusion apply in the required degree.
The tuple $\mathbf d=(d_1,\ldots,d_n)$, with
$d_i=p_i\cdot\beta$, therefore specifies the curve class on every
fiber unambiguously. It is fixed under monodromy. Its entries are
nonnegative and $\sum_i d_i=m>0$.
We work with the proper Deligne--Mumford stack
\[
 M_{Y_n}=\overline{\mathcal M}_{0,1}(Y_n,\mathbf d).
\]

We record the required cohomology facts on connected nodal genus-zero
curves. Such a curve has a tree as dual graph. A line bundle of
nonnegative degree on every component has zero $H^1$. To see that remove a leaf,
use the normalization sequence, and use surjectivity of evaluation at
the attaching point on a nonnegative-degree line bundle on $\PP^1$.
For a line bundle of nonpositive component degrees and negative total
degree, every global section is zero. It vanishes on negative-degree
components and then on all degree-zero components by propagation
through the connected tree. Riemann--Roch on a genus-zero curve now
gives constant dimensions in both applications below.

Consequently the universal curve and map over $M_{Y_n}$ give vector bundles
\begin{equation}\label{ap:eq:Epm}
 E_+=\pi_*f^*\mathcal O_{Y_n}(H),\quad\rk E_+=m+1;
 \qquad E_-=R^1\pi_*f^*\mathcal O_{Y_n}(-H),\quad\rk E_-=m-1.
\end{equation}
These assertions hold over the entire moduli stack, not only on its
reduced locus. To see this, work on a scheme atlas $S\to M_{Y_n}$.
The pulled-back universal curve $\pi_S:C_S\to S$ is proper, flat,
and finitely presented. Apply \cite[Tag~0A1H, Lemma~36.30.1]{Stacks}
with the perfect object $f_S^*\mathcal O_{Y_n}(\pm H)$ and the flat
finitely presented sheaf $\mathcal O_{C_S}$. The derived pushforward
is perfect and commutes with arbitrary base change. Its geometric
fibers have cohomology only in degree $0$ for the positive line
bundle and only in degree $1$ for the negative one. A perfect complex
with this property is locally a finite locally free module in that
single degree. In a local two-term representative, fiberwise
surjectivity or injectivity of the differential makes it locally
split, leaving the kernel or cokernel as a vector bundle.
The ranks are $m+1$ and $m-1$ by Riemann--Roch.

The canonical identifications descend from the atlas and give
\[
 R\pi_*f^*\mathcal O_{Y_n}(H)\simeq E_+[0],\qquad
 R\pi_*f^*\mathcal O_{Y_n}(-H)\simeq E_-[-1].
\]
For every base change, including a nonreduced closed substack, the
derived base-change formula therefore gives the ordinary pullback
of these vector bundles in the indicated degree. In particular
$E_-|_{M_{F_{n,s}}}=R^1\pi_{s*}f_s^*\mathcal O_{F_{n,s}}(-H)=\mathcal V_s$
on the actual, possibly nonreduced, stable-map stack.

The same nonnegative-degree argument applied to
$f^*T_{Y_n}=\bigoplus_i f^*\mathcal O_{Y_n}(2p_i)$ gives $H^1(C,f^*T_{Y_n})=0$.
Thus $M_{Y_n}$ is smooth because its relative deformation theory over the smooth
Artin stack of prestable marked genus-zero curves is unobstructed.
Its dimension is $n-2+2m$. Smoothness here is smoothness of a
Deligne--Mumford stack, not a claim that its coarse space is smooth.

\subsubsection{The zero stack and its virtual class}

On $M_{Y_n}\times U$, the tautological section defining the universal
hypersurface induces a section
\[
 \sigma\in H^0\bigl(M_{Y_n}\times U,
      \operatorname{pr}_1^*E_+\otimes\operatorname{pr}_2^*\mathcal O_U(1)
                         \bigr).
\]
Write $\mathscr M=Z(\sigma)$. For a family of maps over any scheme,
vanishing of this section is precisely vanishing of the pulled-back
defining equation on the universal curve. Hence $\mathscr M$ is the
relative stable-map stack to $\mathscr F_U/U$ with degree $\mathbf d$,
as an equality of moduli problems, including nonreduced bases. The
composition with $F_{n,s}\hookrightarrow Y_n$ preserves stability and
automorphisms. In particular $\mathscr M\to U$ is proper, and there is
a proper evaluation morphism
\[
 \mathrm{Ev}:\mathscr M\longrightarrow\mathscr F_U.
\]

Put $\mathcal A=M_{Y_n}\times U$ and
$\mathcal E=\operatorname{pr}_1^*E_+\otimes\operatorname{pr}_2^*\mathcal O_U(1)$.
The zero section $0_{\mathcal E}:\mathcal A\to\mathcal E$ is a
regular embedding of codimension $m+1$. Pull its refined Gysin map
back by $\sigma:\mathcal A\to\mathcal E$, and denote the resulting
localized map by $0^!_{\mathcal E,\sigma}$. Define
\begin{equation}\label{ap:eq:familyclass}
 [\mathscr M]^\vir:=0^!_{\mathcal E,\sigma}[\mathcal A]
 \in\CH_{n-3+m+N}(\mathscr M)_{\Q}.
\end{equation}
Neither $\sigma$ nor $\mathscr M\hookrightarrow\mathcal A$ is
assumed regular. Only the zero section of the vector bundle is
regular, which is sufficient for this definition.

We now prove the specialization identity in two distinct steps.

\medskip\noindent\textbf{Step 1: refined Gysin base change.}
Let $\iota_s:\{s\}\hookrightarrow U$ be the codimension-$N$
regular embedding. Its bivariant refined pullback is defined on
any stack over $U$; we denote all these maps by $s^!$.
In particular it is defined on $\mathscr M$, although the fiber
inclusion $\mathscr M_s\hookrightarrow\mathscr M$ need not itself
be regular. Commutativity of refined Gysin maps gives
\begin{align}
 s^![\mathscr M]^\vir
 &=s^!0^!_{\mathcal E,\sigma}[\mathcal A]\notag\\
 &=0^!_{\mathcal E_s,\sigma_s}s^![\mathcal A]
  =0^!_{E_+,\sigma_s}[M_{Y_n}].\label{ap:eq:gysinfiber}
\end{align}
The last equality uses smoothness of $\mathcal A\to U$ and the
fact that $\mathcal O_U(1)|_s$ is a constant line, so that after
choosing a nonzero vector in it $\mathcal E_s$ is identified with
$E_+$. The localized class is unchanged by this choice.
These are ordinary refined-Gysin compatibilities, as in
\cite[Theorems 6.2 and 6.4]{Fulton}, in the rational
Deligne--Mumford intersection theory specified above. No flatness
of $\mathscr M\to U$ enters this step.

\medskip\noindent\textbf{Step 2: identification with the stable-map virtual class.}
We apply Kim--Kresch--Pantev \cite[Introduction, formula (2),
and Theorem 1]{KKP}. In its convex-bundle application, the refined
zero-section pullback of the ambient stable-map virtual class
is the sum of the virtual classes of stable maps into the smooth
zero locus, over curve classes with the prescribed ambient image.
Here the smooth projective ambient variety is $Y_n$, the convex bundle
is $\mathcal O_{Y_n}(H)$, and its section has smooth zero locus $F_{n,s}$.
Convexity on all nodal genus-zero maps was proved above.
Moreover $M_{Y_n}$ is smooth of expected dimension, so its standard
virtual class equals $[M_{Y_n}]$.
The curve-class map is injective by integral weak Lefschetz, so
the sum in that theorem has only the single possible class
$\beta_s$. Its conclusion is
\begin{equation}\label{ap:eq:KKPfiber}
 0^!_{E_+,\sigma_s}[M_{Y_n}]=[M_{F_{n,s}}]^\vir
 \quad\text{in }\CH_*(M_{F_{n,s}})_{\Q}.
\end{equation}
This is an equality holding on the zero stack itself, rather than merely after
pushforward to $M_{Y_n}$.

The obstruction-theory compatibility in this application comes
from the exact normal sequence
\[
 0\longrightarrow T_{F_{n,s}}\longrightarrow T_{Y_n}|_{F_{n,s}}
 \xrightarrow{\,ds\,}\mathcal O_{F_{n,s}}(H)\longrightarrow0.
\]
Derived pushforward along the universal curve and dualization give
the triangle over the Artin stack of prestable marked curves.
The morphism is induced by the derivative of the same defining
section, so this is the compatible diagram required by the theorem,
and not only an identity of $K$-classes. One may restrict the prestable
curve stack to a finite-type smooth open containing the relevant
images. The stable-map stacks are finite-type Deligne--Mumford
stacks with quotient stratifications, as required for the
intersection-theoretic construction. The compatible-triple
formalism of \cite[Theorem 4.8 and Corollary 4.9]{Manolache}
gives the same comparison.

Combining~\eqref{ap:eq:gysinfiber} and~\eqref{ap:eq:KKPfiber} proves
\begin{equation}\label{ap:eq:fiberVFC}
 s^![\mathscr M]^\vir=[M_{F_{n,s}}]^\vir.
\end{equation}

\begin{remark}[Relative and absolute obstruction theories]
For the zero stack, the relative two-term complex over $U$ is
$[\mathcal E^\vee|_{\mathscr M}\to
\Omega_{\mathcal A/U}|_{\mathscr M}]$, in degrees $[-1,0]$.
Its fiber gives the zero-locus obstruction theory used in~\eqref{ap:eq:KKPfiber}. The absolute complex instead contains
$\Omega_{\mathcal A}$ and hence the parameter-space cotangent
directions from $\Omega_U$. Its ordinary restriction to a fiber
is not the absolute obstruction theory of $M_{F_{n,s}}$.
The conclusion needed here is the refined equality~\eqref{ap:eq:fiberVFC}, proved by the two steps above.
\end{remark}

\subsubsection{Restriction and the proof of ambientness}

Pull $E_-$ and the cotangent-line class from $M_{Y_n}$ to $\mathscr M$ and
form
\begin{equation}\label{ap:eq:universalC}
 \mathscr C_a^\lambda
 =\mathrm{Ev}_*\bigl(\psi^a e_\lambda(E_-)
                          \cap[\mathscr M]^\vir\bigr)
 \in\CH^*(\mathscr F_U)_{\Q}[\lambda].
\end{equation}
The target $\mathscr F_U$ is smooth, so its Chow homology is identified
with its codimension-graded Chow ring. The expression is a polynomial
in $\lambda$ of degree at most $m-1$.

Proper pushforward commutes with refined fiber pullback. Chern classes
and $\psi$ also commute with it. More precisely, the evaluation
square over $\iota_s$ is cartesian, and $q:\mathscr F_U\to U$
is smooth, so $j_s:F_{n,s}\hookrightarrow\mathscr F_U$ is a regular
embedding. Proper-pushforward compatibility identifies its Gysin
pullback of~\eqref{ap:eq:universalC} with the evaluation pushforward
of $s^!$ of the capped class on $\mathscr M$.
Applying~\eqref{ap:eq:fiberVFC} therefore
gives, coefficientwise,
\begin{equation}\label{ap:eq:restrictC}
 j_s^*\mathscr C_a^\lambda=C_{a,\beta_s}^\lambda.
\end{equation}
Lemma~\ref{ap:lem:universal} now proves~\eqref{ap:eq:chowambient}.
Applying the cycle-class map proves all cohomological assertions in
Theorem~\ref{ap:thm:ambient}. This completes its proof.

\begin{remark}[Relative canonical bundles]
Adjunction on the universal family gives
$K_{\mathscr F_U/U}=\mathcal O_{Y_n}(-H)|_{\mathscr F_U}\otimes q^*\mathcal O_U(1)$.
We deliberately used the globally fixed line bundle
$\mathcal O_{Y_n}(-H)$ in~\eqref{ap:eq:universalC}. On a fiber the extra factor
is a constant one-dimensional vector space, and has zero Chern class;
the resulting Euler classes agree with~\eqref{ap:eq:C}. Thus no
trivialization of $\mathcal O_U(1)$ over the parameter space is being
assumed.
\end{remark}

\begin{remark}[A uniform lift]
Choosing one lift of each coefficient of $\mathscr C_a^\lambda$ to
$\mathscr F$ in Lemma~\ref{ap:lem:universal} produces a polynomial
$\alpha_{a,\mathbf d}(\lambda)\in\CH^*(Y_n)_{\Q}[\lambda]$ such that
$i_s^*\alpha_{a,\mathbf d}(\lambda)=C_{a,\beta_s}^\lambda$ for every
$s\in U$. The ambient polynomial need not be unique, but the argument
works simultaneously for all smooth fibers.
\end{remark}

\subsection{Monodromy and the full one-point comparison}

There is also a cohomological proof of ambientness, using the
global class~\eqref{ap:eq:universalC} and invariant cycles.

\begin{proposition}\label{ap:prop:invariants}
For every $k$ and $s\in U$,
\[
 H^k(F_{n,s},\Q)^{\pi_1(U,s)}=i_s^*H^k(Y_n,\Q).
\]
\end{proposition}
\begin{proof}
Apply the global invariant-cycle theorem to
$q:\mathscr F\to B$. Both spaces are smooth projective varieties,
$B$ is connected, $q$ is surjective, and its restriction over $U$ is
smooth and proper. The theorem identifies the invariant subspace
with the image of $H^k(\mathscr F,\Q)$ in $H^k(F_{n,s},\Q)$; see
\cite[Theorem 3.3]{Voisin}. The cohomological projective-bundle
formula writes every class on $\mathscr F$ in powers of $\xi$ with
coefficients from $Y_n$. Restricting to $F_{n,s}$ kills every positive power
of $\xi$. This proves the assertion.
\end{proof}

For each coefficient of $\lambda$, the cycle class of
$\mathscr C_a^\lambda$ is a global cohomology class on
$\mathscr F_U$. Its restrictions to the fibers form a flat section
of $R^kq_*\Q$. Namely, on a contractible analytic trivializing neighborhood,
the fiber inclusions are homotopic, so the restricted class is
constant under parallel transport. Equation~\eqref{ap:eq:restrictC}
therefore makes $C_{a,\beta_s}^\lambda$ monodromy invariant.
Proposition~\ref{ap:prop:invariants} proves its cohomological
ambientness. The fixed degree tuple ensures that monodromy
preserves the individual curve class, not just a sum of classes.

The global invariant-cycle theorem supplies a real step in this
argument. It identifies the image on a smooth fiber of cohomology
of the open family with the image from its smooth projective
compactification. Restriction
$H^*(\mathscr F)\to H^*(\mathscr F_U)$ need not be surjective;
the cohomological argument does not use the surjectivity available
for Chow groups in Lemma~\ref{ap:lem:universal}.

For odd $n$, let $V_{\mathrm{na}}\subset H^{n-1}(F_{n,s},\Q)$ be the
Poincar\'e-orthogonal complement of the ambient middle part. The ambient
pairing is nondegenerate, so Proposition~\ref{ap:prop:invariants} gives
$V_{\mathrm{na}}^{\pi_1(U,s)}=0$. This includes the 18-dimensional
nonambient middle summand for $n=5$. It does not identify that summand
with the entire Lefschetz-primitive cohomology.

For clarity, nondegeneracy on the ambient part follows from weak and
hard Lefschetz and Poincar\'e duality. If $k<n-1$, weak Lefschetz identifies
$H^k(F_n)$ with $H^k(Y_n)$, and hard Lefschetz on $F_n$ supplies the dual
ambient classes in degree $2n-2-k$. In middle degree, multiplication by
$H$ gives the isomorphism $H^{n-1}(Y_n)\to H^{n+1}(Y_n)$ and hence the
nondegenerate form $\int_YabH$. The odd ambient groups are zero.

\begin{corollary}[Upgrading an ambient comparison]\label{ap:cor:fullJ}
Suppose an ambient class $E_{a,\beta}^\lambda\in H^*(Y_n)[\lambda]$
satisfies
\[
 \int_{F_n} i^*b\,C_{a,\beta}^\lambda
 =\int_{F_n} i^*b\,i^*E_{a,\beta}^\lambda
 \quad\text{for all }b\in H^*(Y_n).
\]
Then $C_{a,\beta}^\lambda=i^*E_{a,\beta}^\lambda$ in cohomology.
Thus a comparison against ambient insertions determines the full
one-point evaluation classes. It also determines every generating
series formed from them, in particular the one-point $J$-series
once its normalization is fixed.
\end{corollary}
\begin{proof}
The difference is ambient by Theorem~\ref{ap:thm:ambient}, and is
orthogonal to the ambient space by the displayed equality.
Nondegeneracy makes the difference zero. Apply this coefficientwise
in $\lambda$ and then in the Novikov and descendant variables.
\end{proof}

This conclusion does not remove the need to retain an ambient lift
when dividing by $\lambda-H$. In particular a top coefficient that
vanishes after restriction to $F_n$ can still contribute after division.
Ambientness and that extraction issue are distinct statements.

\subsection{Dependencies of the ambientness proof}

The argument above uses ordinary intersection theory and the stable-map
virtual class. It uses no mirror theorem, no compact Calabi--Yau
interpretation and no differential-equation input. Ambientness follows from
the universal family, without assuming that a fixed hypersurface admits
coordinate permutations as automorphisms. The separate differential argument
for minimality has its own, independent dependencies; see
Proposition~\ref{thm:minimalperiod}.

\section{Extension and Fourier-transform lemmas for minimality}\label{app:fourier}
Work with left modules over
\[
 \mathscr D=\C[t]\langle\partial_t\rangle,
 \qquad \mathscr D^\times=\C[t,t^{-1}]\langle\partial_t\rangle.
\]
Localization at $t$ is the usual exact restriction from the affine line to $\mathbb G_m$. All modules in the argument are algebraic modules on these smooth curves. No statement about arbitrary schemes or stacks is used.

We first record a point-module fact in a form that makes the extension argument
explicit. Let
\[
 \delta_0=\mathscr D/\mathscr D t=\C[\partial_t]e,\qquad te=0.
\]
Its basis $e_k=\partial_t^ke$ satisfies
\begin{equation}\label{irr:eq:delta}
 te_k=-ke_{k-1},\qquad \theta e_k=-(k+1)e_k.
\end{equation}
%It is simple. 
More generally, a module on which $t$ acts locally nilpotently is a direct sum of copies of $\delta_0$. This particular point-support statement can be proved directly
because  the map
\[
 \C[\partial_t]\otimes_\C\ker t\longrightarrow M,
 \qquad \partial_t^k\otimes h\longmapsto\partial_t^kh,
\]
is surjective by induction on the nilpotence order of an element. If $t^Nm=0$, subtract
$(-1)^{N-1}\partial_t^{N-1}(t^{N-1}m)/(N-1)!$ to reduce that order. It is injective by applying the highest relevant power of $t$ to a finite relation. Consequently every nonzero such module has a quotient isomorphic to $\delta_0$.

\begin{lemma}[The extension used in each transform]\label{irr:lem:extension}
Let $r\ge1$ and
\[
 L=\theta^r+\sum_{j\ge1}t^jp_j(\theta),\qquad p_j\in\C[\theta].
\]
Suppose $\mathscr D^\times/\mathscr D^\times L$ is nonzero and simple. Then
\[
 B=t^{-1}L=\partial_t\theta^{r-1}
             +\sum_{j\ge1}t^{j-1}p_j(\theta)\in\mathscr D,
 \qquad M=\mathscr D/\mathscr D B
\]
is simple. More precisely, $M$ has neither a nonzero submodule nor a nonzero quotient supported at $t=0$, and restricts to $\mathscr D^\times/\mathscr D^\times L$.
\end{lemma}
\begin{proof}
The order in which the factors are written is important. Let $v=\sum_{k=0}^Nc_ke_k\in\delta_0$ with $c_N\ne0$. The first term of $Bv$ has coefficient
\[
 c_N(-N-1)^{r-1}\ne0
\]
at $e_{N+1}$. Each remaining term $t^{j-1}p_j(\theta)v$ has degree at most $N$ in $\partial_t$. Thus $B$ acts injectively on $\delta_0$, and
\begin{equation}\label{irr:eq:noquot}
 \operatorname{Hom}_{\mathscr D}(M,\delta_0)=0.
\end{equation}

Let $*$ denote formal transpose, such that $t^*=t$, $\partial_t^*=-\partial_t$, and $(UV)^*=V^*U^*$. Since $\theta^*=-\theta-1$, one has
\begin{equation}\label{irr:eq:transpose}
 B^*=(-1)^r\partial_t\theta^{r-1}
       +\sum_{j\ge1}t^{j-1}p_j(-\theta-j).
\end{equation}
The same highest-degree calculation shows that $B^*$ acts injectively on $\delta_0$.

This also proves directly, without invoking holonomic duality, that
multiplication by $t$ is injective on $M$. Indeed, suppose $t[U]=0$. Then
$tU=VB$ for some $V\in\mathscr D$, whence
\[
 U^*t=B^*V^*.
\]
Apply this equality to $e\in\delta_0$. Injectivity of $B^*$ gives $V^*e=0$, so
$V^*=Wt$ for some $W\in\mathscr D$. Transposition yields $V=tW^*$. Since the
Weyl algebra is a domain, cancellation in $tU=tW^*B$ gives $U=W^*B$ and $[U]=0$.

It follows that $M$ has no nonzero $t$-torsion submodule. It has no
nonzero quotient supported at zero either. Indeed, such a quotient would
be a nonzero module with locally nilpotent $t$ and hence would have a
quotient $\delta_0$, contrary to~\eqref{irr:eq:noquot}.

Finally let $N\subset M$ be a nonzero submodule. Its localization is nonzero, by
the absence of $t$-torsion, and therefore is the entire localization of $M$ by
the assumed simplicity. Thus $M/N$ is supported at zero, and must vanish. Hence
$M$ is simple.
\end{proof}

\begin{remark}
In the language of holonomic modules, the lemma identifies this presentation
with the intermediate extension at zero. Using $L$ itself as a polynomial
Weyl-algebra presentation, rather than $t^{-1}L$, would not justify that
assertion.
\end{remark}

We shall also use a simple localization observation. A simple module stays
simple after a nonzero localization. Indeed, the kernel of its localization is a
submodule and must be zero if the localization is nonzero. Any nonzero submodule
of the localized module meets the original module after clearing a denominator;
the intersection is a nonzero submodule and hence is the whole original module.
This proves simplicity after localization, both at $t$ and at all nonzero
polynomials.

\subsection{One Fourier step}
Use the Weyl-algebra isomorphism
\[
 \mathscr F:\C[t]\langle\partial_t\rangle\longrightarrow
              \C[x]\langle\partial_x\rangle,
 \qquad t\longmapsto-\partial_x,\quad
 \partial_t\longmapsto x.
\]
It preserves the commutation relation, since $[x,-\partial_x]=1$, and transports
simple modules to simple modules. This is the algebraic Fourier transform.

\begin{lemma}[Explicit transformed operator]\label{irr:lem:transform}
For $L$ as in Lemma~\ref{irr:lem:extension}, put
\begin{equation}\label{irr:eq:transform}
 \mathfrak T(L)=\theta_y^{r-1}
       +\sum_{j\ge1}y^jp_j(\theta_y)
                         \prod_{a=1}^{j-1}(\theta_y+a).
\end{equation}
If $\mathscr D^\times/\mathscr D^\times L$ is simple and $\mathfrak T(L)$ has positive differential order, then
\[
 \C[y,y^{-1}]\langle\partial_y\rangle\big/
 \C[y,y^{-1}]\langle\partial_y\rangle\mathfrak T(L)
\]
is simple, as is its differential-field localization over $\C(y)$.
\end{lemma}
\begin{proof}
By Lemma~\ref{irr:lem:extension}, $M=\mathscr D/\mathscr D(t^{-1}L)$ is simple. Its Fourier transform has cyclic presentation by
\begin{equation}\label{irr:eq:fourierB}
 \mathscr F(B)=x(-\theta_x-1)^{r-1}
          +\sum_{j\ge1}(-\partial_x)^{j-1}p_j(-\theta_x-1).
\end{equation}
Restrict to $x\ne0$ and invert the coordinate, $x=1/y$. Then
\[
 \theta_x=-\theta_y,\qquad -\partial_x=y\theta_y,
 \qquad (y\theta_y)^k=y^k\prod_{a=0}^{k-1}(\theta_y+a).
\]
Thus \eqref{irr:eq:fourierB} becomes
\[
 C=y^{-1}(\theta_y-1)^{r-1}
     +\sum_{j\ge1}y^{j-1}
       \left(\prod_{a=0}^{j-2}(\theta_y+a)\right)p_j(\theta_y-1).
\]
The rational gauge change $f(y)=y\,g(y)$ replaces $C$ by $y^{-1}Cy$. Multiplying on the left by the unit $y$ then gives precisely
\[
 y\bigl(y^{-1}Cy\bigr)=\mathfrak T(L).
\]
These equalities take place in the localized Weyl algebra, with all products in
their stated order. In particular they identify the cyclic modules.

Positive differential order implies that the final cyclic module has nonzero generic rank. Thus the restriction of $\mathscr F(M)$ to $x\ne0$ was nonzero. The localization observation, coordinate inversion, and the gauge isomorphism prove simplicity. The same observation applies on passing to $\C(y)$.
\end{proof}

\end{document}